\documentclass[10pt,reqno]{amsart}
\usepackage[margin=1in]{geometry}

\usepackage{amsmath,amsthm,wasysym,bbold,tensor,mdframed}
\usepackage{amssymb}
\usepackage{amsfonts}
\usepackage{mathscinet}
\usepackage[dvipsnames]{xcolor}
\usepackage{tikz-cd}
\usepackage{tikz,keyval,pgfsys,graphicx}
\usepackage{mathtools, enumerate, combelow}
\usepackage[shortlabels]{enumitem}

\usetikzlibrary{arrows, decorations.pathreplacing}
\usepgflibrary{arrows}

\newtheorem{thm}{\normalfont\scshape Theorem}[section]

\newtheorem{prop}[thm]{\normalfont\scshape Proposition}
\newtheorem{lem}[thm]{\normalfont\scshape Lemma}
\newtheorem{cor}[thm]{\normalfont\scshape Corollary}

\theoremstyle{definition}
\newtheorem{defn}[thm]{\normalfont\scshape Definition}

\theoremstyle{remark}
\newtheorem{rem}[thm]{Remark}
\theoremstyle{remark}

\allowdisplaybreaks

\DeclareMathOperator{\TT}{\mathcal{T}}

\numberwithin{equation}{section}

\newcommand{\ccc}{\mathfrak{c}}
\newcommand{\qqq}{\mathfrak{q}}
\newcommand{\ddd}{\mathfrak{d}}

\newcommand{\ZZ}{\mathbb{Z}}

\newcommand{\CC}{\mathbb{C}}

\newcommand{\ppp}{\mathfrak{p}}

\newcommand{\NN}{\mathbb{N}}

\newcommand{\HH}{\mathcal{H\kern-.44em H}}
\newcommand{\git}{/\kern-.35em/}

\newcommand{\quot}{\mathrm{quot}}
\newcommand{\core}{\mathrm{core}}
\newcommand{\Sym}{\mathrm{Sym}}
\newcommand{\FF}{\mathbb{F}}

\newcommand{\UTor}{U_{\qqq,\ddd}(\ddot{\mathfrak{sl}}_r)}
\newcommand{\Sss}{\mathcal{S}}

\title{Tesler identities for wreath Macdonald polynomials}


\author{Marino Romero}
\address{Fakultät für Mathematik, Universität Wien, Austria}
\email{marino.romero@univie.ac.at}

\author{Joshua Jeishing Wen}
\address{Fakult\"{a}t f\"{u}r Mathematik, Universit\"{a}t Wien, Austria}
\email{joshua.jeishing.wen@univie.ac.at}

\dedicatory{In memory of Adriano M. Garsia}

\begin{document}

\begin{abstract}
We give an explicit formula for an operator that sends a wreath Macdonald polynomial to the delta function at a character associated to its partition.
This allows us to prove many new results for wreath Macdonald polynomials, especially pertaining to reciprocity: Macdonald--Koornwinder duality, evaluation formulas, etc.
Additionally, we initiate the study of wreath interpolation Macdonald polynomials, derive a plethystic formula for wreath $(q,t)$-Kostka coefficients, and present series solutions to the bispectral problem involving wreath Macdonald operators.
Our approach is to use the eigenoperators for wreath Macdonald polynomials that have been produced from quantum toroidal and shuffle algebras.
\end{abstract}
\maketitle

\section{Introduction}
Let $\Lambda$ be the ring of symmetric functions (in infinitely many variables), and set $\Lambda_{q,t} \coloneq \CC(q,t)\otimes\Lambda$.
The \textit{modified Macdonald polynomials} $\{H_\lambda\}$ are a distinguished basis of $\Lambda_{q,t}$ that have played a large role in algebraic combinatorics, algebraic geometry, integrable systems, knot theory, probability, quantum algebra, and other fields.
Macdonald polynomials satisfy an interesting and useful duality that is hinted at by \textit{Macdonald reciprocity}.
In terms of $H_\lambda$, it may be expressed as follows.
For a partition $\lambda$, let us view $\lambda$ as its set of boxes in the French convention; we assign each box a coordinate in $\ZZ^2$, with the box in the bottom-left corner assigned to $(0,0)$.
Write
\begin{align*}
B_\lambda  \coloneq  \sum_{(a,b)\in\lambda}q^at^b && \text{ and } &&
D_\lambda  \coloneq  (1-q)(1-t) B_\lambda -1.
\end{align*}
Reciprocity then takes the following form: for partitions $\lambda$ and $\mu$,
\begin{equation}
\frac{H_\lambda[D_\mu]}{H_\lambda[-1]}=\frac{H_\mu[D_\lambda]}{H_\mu[-1]}.
\label{MacRec}
\end{equation}
Here, the evaluations should be interpreted plethystically.
Equation (\ref{MacRec}) was proved by Koornwinder; a proof can also be found in \cite{Mac}.

The perspective we will take in accessing (\ref{MacRec}) and its deeper aspects comes from work of Garsia--Haiman--Tesler \cite{GHT}, which developed and refined ideas underlying the Garsia--Tesler proof of \textit{Macdonald integrality} \cite{GarsiaTesler}.
In \cite{GHT}, the authors prove what we call the \textit{Tesler identity}.
Namely, let
\begin{align*}
\Omega[X]& \coloneq \exp\left( \sum_{k>0} \frac{p_k}{k} \right),\\
\TT(f)& \coloneq  f[X+1],\\
\nabla H_\lambda& \coloneq  \left( \prod_{(a,b)\in\lambda}(-q^at^b) \right)H_\lambda,\\
\mathsf{V}& \coloneq \nabla\Omega\left[ \frac{X}{(1-q)(t-1)} \right]\TT\nabla.
\end{align*}
The identity is then
\begin{equation}
\Omega\left[ \frac{XD_\lambda}{(1-q)(t-1)} \right]=\mathsf{V}\left( \frac{H_\lambda}{H_\lambda[-1]} \right).
\label{TeslerId}
\end{equation}
Conceptually, the left-hand-side of (\ref{TeslerId}) is a delta function with respect to the \textit{modified Macdonald pairing} $\langle -, -\rangle_{q,t}'$ on $\Lambda_{q,t}$:
\[
\left\langle\Omega\left[ \frac{XD_\lambda}{(1-q)(t-1)} \right], f\right\rangle_{q,t}'=f[D_\lambda].
\]
The reciprocity (\ref{MacRec}) is a combination of (\ref{TeslerId}) and the fact that $\mathsf{V}$ is self-adjoint with respect to $\langle -,-\rangle_{q,t}'$:
\[
\left\langle \frac{H_{\lambda}}{H_\lambda[-1]},\mathsf{V}\left( \frac{H_\mu}{H_\mu[-1]} \right)\right\rangle_{q,t}'
=\left\langle \mathsf{V}\left(\frac{H_{\lambda}}{H_\lambda[-1]}\right), \frac{H_\mu}{H_\mu[-1]} \right\rangle_{q,t}'.
\]
Playing similar games with $\mathsf{V}$ and $\langle -,-\rangle_{q,t}'$, one can also prove \textit{Macdonald--Koornwinder duality}: introducing an extra variable $u$,
\begin{equation}
\frac{H_\lambda[1+uD_\mu]}{\displaystyle\prod_{(a,b)\in\lambda}(1-uq^at^b)}
=\frac{H_\mu[1+uD_\lambda]}{\displaystyle\prod_{(a,b)\in\mu}(1-uq^at^b)}.
\label{MK}
\end{equation}
Finally, we note that setting $\mu=\varnothing$ in (\ref{MK}) yields the \textit{evaluation formula}:
\begin{equation}
H_\lambda[1-u]=\prod_{(a,b)\in\lambda}(1-uq^at^b).
\label{MacEv}
\end{equation}
These are but a small fraction of the results one can tease out of the Tesler identity (\ref{TeslerId}).

Beyond type $A$, an analogue of (\ref{MacRec}) was proved by Cherednik \cite{ChereEv}.
Cherednik worked in the finite-variable setting, and he provided an underlying conceptual structure that explains the duality: the spherical \textit{double affine Hecke algebra} (DAHA) and its automorphism $\varphi$.  
To describe this, let us highlight two ways $f\in\Lambda_{q,t}$ can act on $\Lambda_{q,t}$: (1) via multiplication and (2) via its \textit{Delta operator} $\Delta[f]$, which acts diagonally on the basis $\{H_\lambda\}$ by setting
\[ 
\Delta[f](H_\lambda)=f[D_\lambda]H_\lambda.
\]
In the unmodified, finite-variable setting, the Delta operators are related to the \textit{Macdonald operators}.
The spherical subalgebra of the DAHA can be viewed as the algebra generated by multiplication and Macdonald operators, and $\varphi$ sends the former to the latter (but not vice-versa).
From this perspective, $\mathsf{V}$ also plays a role (cf. \cite{Garsia-Mellit}): for $f\in\Lambda_{q,t}$,
\begin{equation}
\Delta[f]\mathsf{V}=\mathsf{V}f,
\label{VMiki}
\end{equation}
where on the right-hand-side, we view $f$ as its multiplication operator.
Thus, commuting past $\mathsf{V}$ right-to-left performs a version of $\varphi$ in infinitely many variables.
In the infinite-variable setting, the spherical DAHA is replaced with the \textit{elliptic Hall algebra} (EHA) and $\varphi$ is replaced with the \textit{Miki automorphism} \cite{SchiffVass}.
Quite interestingly, shadows of the EHA are all over \cite{GHT}, a decade before the explosion of interest for this algebra in the 2010s.
This is because the EHA is the algebraic structure built out of Jing's vertex operators for Macdonald polynomials \cite{JingMac}, and these vertex operators are a key ingredient in \cite{GHT}. 

\subsection{Wreath Macdonald polynomials}
This paper is concerned with generalizing the picture above to the \textit{wreath Macdonald polynomials} defined by Haiman \cite{Haiman}.
Fix an integer $r\ge 1$.
The wreath Macdonald polynomials are elements of the tensor power $\Lambda_{q,t}^{\otimes r}$.
They are still indexed by a single partition and reduce to usual modified Macdonald polynomials when $r=1$.
We will therefore denote them again by $\{H_\lambda\}$.
An important ingredient of their definition is the \textit{core-quotient decomposition} of a partition, which yields a bijection:
\begin{align*}
\{\hbox{partitions}\}&\leftrightarrow\{\hbox{$r$-core partitions}\}\times\{\hbox{$r$-tuples of partitions}\}\\
\lambda &\mapsto \left( \core(\lambda), \quot(\lambda) \right).
\end{align*}
We review this decomposition in \ref{Partitions} below.
The decomposition depends on $r$, but as $r$ is fixed throughout, we will drop it from the notation.
Fixing an $r$-core $\alpha$, $\{H_\lambda\, |\, \core(\lambda)=\alpha\}$ gives a basis of $\Lambda_{q,t}^{\otimes r}$. 
The existence of $H_\lambda$ was first proved by \cite{BezFink}, which also proved an analogue of \textit{Macdonald positivity}.

The notion of ``colors'' permeates wreath Macdonald theory.
First, we index the tensorands of $\Lambda_{q,t}^{\otimes r}$ with $\ZZ/r\ZZ$ and denote by $p_n[X^{(i)}]$ the element with $p_n$ in the $i$th tensorand, i.e. the color $i$ power sum. With this indexing, $X^{(i)} = X^{(j)}$ if $i = j\hbox{ mod }r$.
A general element $f\in\Lambda_{q,t}^{\otimes r}$ depends on all $r$ families of variables, and so we denote it by $f[X^\bullet]$.
Plethysm in this setting takes on a richer character because we can take nontrivial linear maps across different colors.
Some important examples of this are the following:
\begin{align*}
p_n[\iota X^{(i)}]& \coloneq  p_n[X^{(-i)}],\\
p_n[\sigma X^{(i)}]& \coloneq  p_n[X^{(i+1)}],\\
p_n[(1-q\sigma)X^{(i)}]& \coloneq p_n[X^{(i)}]-q^np_n[X^{(i+1)}],\\
p_n\left[\frac{X^{(i)}}{(1-q\sigma)}\right]& \coloneq \frac{\sum_{j=0}^{r-1}q^jp_n[X^{(i+j)}]}{1-q^{nr}}=p_n[(1-q\sigma)^{-1}X^{(i)}].
\end{align*}
Likewise, we have a translation operator $\TT$ for each variable color, which we denote by $\TT[X^{(i)}]$.

To each box $\square=(a,b)$ in a partition, we assign its color to be $\bar{c}_\square \coloneq  b-a\hbox{ mod }r$.
In French convention, this is the SW-to-NE diagonal that the box lies on, modulo $r$.
Correspondingly, we assign the color $b-a\hbox{ mod }r$ to the character $q^at^b$, and set
\begin{align*}
B_\lambda^{(i)}& \coloneq \sum_{\substack{\square=(a,b)\in\lambda \\\bar{c}_\square =i}}q^at^b, &
B_\lambda^\bullet& \coloneq \sum_{i\in\ZZ/r\ZZ} B_\lambda^{(i)},\\
D_\lambda^{(i)}& \coloneq (1+qt)B^{(i)}-qB^{(i+1)}-tB^{(i-1)}-\delta_{i,0}, &
D_\lambda^\bullet& \coloneq \sum_{i\in\ZZ/r\ZZ}D_\lambda^{(i)}= (1-q)(1-t)B_\lambda^\bullet -1.
\end{align*}
When performing plethystic evaluations, we will assign $X^{(i)}$ to characters of color $i$.
For instance, 
\[
f[D_\lambda^\bullet] \coloneq f[X^\bullet]\,\bigg|_{X^{(i)}\mapsto D_\lambda^{(i)}}.
\]

Finally, color also plays a role in our notion of nabla operators.
Recall that the wreath Macdonald polynomials give a basis of $\Lambda_{q,t}^{\otimes r}$ only after restricting to partitions with a fixed $r$-core $\alpha$.
For $\lambda$ with $\core(\lambda)=\alpha$, we then define
\[
\nabla_\alpha H_\lambda=\left( \prod_{\substack{\square=(a,b)\in\lambda\backslash\alpha\\ \bar{c}_\square=0}}(-q^at^b) \right)H_\lambda
\]
and then define $\nabla_\alpha f$ by linearity.
Our notion of the nabla operator is significantly simpler than the one posed in \cite{OSWreath}.

\subsubsection{Main result}
Let
\[
\mathbb{E}_\lambda =\Omega\left[ \sum_{i\in\ZZ/r\ZZ}X^{(i)}\left( \frac{D_\lambda^\bullet}{(1-q)(t-1)} \right)^{(i)} \right],
\]
where in assigning colors to $D_\lambda^\bullet/(1-q)(1-t)$, we expand it into a power series assuming $|q|,|t|<1$.
There is a natural analogue of the modified Macdonald pairing, which we also denote by $\langle-,-\rangle_{q,t}'$, and $\mathbb{E}_\lambda$ plays the role of a delta function (Corollary \ref{DeltaFunc}):
\[
\langle\mathbb{E}_\lambda, f\rangle_{q,t}'=f[X^\bullet]\, \bigg|_{X^{(i)}\mapsto D_\lambda^{(-i)}} \eqcolon f[\iota D_\lambda^\bullet].
\]
The main theorem of this paper is the following wreath analogue of (\ref{TeslerId}):
\begin{thm}\label{TeslerThm}
Fix $r>2$.
For an $r$-core $\alpha$, let
\[
\mathsf{V}_\alpha \coloneq \nabla_\alpha\Omega\left[ \frac{X^{(0)}}{(1-q\sigma^{-1})(t\sigma-1)} \right]\TT[X^{(0)}]\nabla_\alpha.
\]
We then have $\nabla_{\core(\lambda)} H_\lambda = H_\lambda[\iota D_{w_0\core( \lambda)}^\bullet] H_\lambda$ and
\begin{equation}
\mathbb{E}_\lambda=\mathsf{V}_{\core(\lambda)}\left(\frac{H_\lambda}{H_\lambda[\iota D_{w_0\core(\lambda)}^\bullet]}\right).
\label{TeslerWr}
\end{equation}
\end{thm}

\subsubsection{Symmetry of nabla}
The $w_0\core(\lambda)$ appearing in the denominator on the right-hand-side of (\ref{TeslerWr}) comes from an action of the symmetric group $\Sigma_r$ on partitions, defined in \ref{Reverse}; the element $w_0\in\Sigma_r$ is the longest element.
Another permutation that plays well with our machinery is the long cycle $\sigma$.
For any permutation $w$, we have $w\core(\lambda)=\core(w\lambda)$.

Concerning nabla, there are two significant new features in the wreath case:
\begin{itemize}
\item Besides $1$, every eigenvalue of $\nabla_\alpha$ is degenerate.
\item $\nabla_\alpha$ is not self-adjoint with respect to $\langle -, -\rangle_{q,t}'$; its adjoint is $\nabla_{w_0\alpha}$.
\end{itemize}
The second feature is basically responsible for $w_0\core(\lambda)$ appearing in (\ref{TeslerWr}).
The first feature gives room to produce variations on (\ref{TeslerWr}).
For instance, for $k\in\ZZ/r\ZZ$, we have the following (Lemma \ref{NablaK}):
\[
\nabla_{\sigma^{-k}\core(\lambda)}H_\lambda[\sigma^{-k}X^\bullet]=\left( \prod_{\substack{\square=(a,b)\in\lambda\backslash\alpha\\ \bar{c}_\square=0}}(-q^at^b) \right)H_\lambda[\sigma^{-k}X^\bullet].
\]
Thus, we can compensate for the $\sigma^{-k}$ plethysm by also changing the core parameter on nabla; we emphasize that this is a nontrivial symmetry.
This leads to the following shifted variant of Theorem \ref{TeslerThm}:

\begin{thm}\label{TeslerThmK}
Fix $r>2$ and $k\in\ZZ/r\ZZ$.
Let
\[
\mathbb{E}_\lambda^{(k)} =\Omega\left[ \sum_{i\in\ZZ/r\ZZ}X^{(i)}\left( \frac{D_\lambda^\bullet}{(1-q)(t-1)} \right)^{(i+k)} \right].
\]
For a $r$-core $\alpha$, we set
\[
\mathsf{V}_\alpha^{(k)} \coloneq \nabla_{\sigma^{-k}\alpha}\Omega\left[ \frac{X^{(-k)}}{(1-q\sigma^{-1})(t\sigma-1)} \right]\TT[X^{(-k)}]\nabla_{\sigma^{-k}\alpha}.
\]
We then have
\begin{equation}
 \mathbb{E}_\lambda^{(k)}=\mathsf{V}_{\core(\lambda)}^{(k)}\left( \frac{H_\lambda[\sigma^{-k}X^\bullet]}{H_\lambda[\iota D_{\core(w_0\lambda)}^\bullet]} \right).   
 \label{TeslerWrK}
\end{equation}
\end{thm}
Let us note that $\mathbb{E}_\lambda^{(k)}$ is also some form of delta function (Proposition \ref{ShiftDelta}):
\[
\langle\mathbb{E}_{\lambda}^{(k)}, f\rangle_{q,t}'= f[X^\bullet]\,\bigg|_{X^{(i)}\mapsto D_\lambda^{(k-i)}} \eqcolon f[\sigma^k\iota D_\lambda^\bullet].
\]

\subsubsection{Applications}\label{Apps}
As in \cite{GHT}, Theorem \ref{TeslerThm} and its shifted generalization, Theorem \ref{TeslerThmK}, have many immediate applications.
In the wreath case, however, these all yield new results.
We cover these consequences in Section \ref{Consequence}.
Here is a list:
\begin{enumerate}
\item We prove an analogue of Macdonald--Koornwinder duality (\ref{MK}).
For $k\in\ZZ/r\ZZ$ and partitions $\lambda$ and $\mu$ with $\core(\mu)=w_0\sigma^{-k}\core(\lambda)$, this takes the following form:
\begin{equation}
\frac{H_\lambda[1+u\sigma^k\iota D_{\mu}^\bullet]}{\displaystyle\prod_{\substack{\square=(a,b)\in\lambda\backslash\core(\lambda)\\\bar{c}_\square=k}}(1-uq^at^b)}
=
\frac{H_{\mu}[1+u\sigma^k\iota D_{\lambda}^\bullet]}
{\displaystyle\prod_{\substack{\square=(a,b)\in \mu\backslash\core(\mu)\\\bar{c}_\square=k}}(1-uq^at^b)}.
\label{MKWr}
\end{equation}
Note that if we fix $\lambda$, $\core(\mu)$ may change as we vary $k$.
\item Setting $\mu=w_0\sigma^{-k}\core(\lambda)$ in (\ref{MKWr}), we obtain the evaluation formula:
\begin{equation}
H_\lambda[1+u\sigma^{k}\iota D_{w_0\sigma^{-k}\core(\lambda)}^\bullet]=\prod_{\substack{\square=(a,b)\in\lambda\backslash\core(\lambda)\\\bar{c}_\square=k}}(1-uq^at^b).
\label{EvalWr}
\end{equation}
This proves an evaluation conjecture used in \cite{AyersDinkins}.
\item Taking $u\rightarrow\infty$ in (\ref{EvalWr}), we obtain
\[
H_\lambda[\sigma^k\iota D_{w_0\sigma^{-k}\core(\lambda)}^\bullet]=\prod_{\substack{\square=(a,b)\in\lambda\backslash\core(\lambda)\\\bar{c}_\square=k}}(-q^at^b).
\]
Note that the case $k=0$ gives a formula for the evaluation $H_\lambda[\iota D_{w_0\core(\lambda)}^\bullet]$ appearing in the denominator of \eqref{TeslerWr} and \eqref{TeslerWrK}.
It also coincides with the eigenvalue of $\nabla_{\core(\lambda)}$ at $H_\lambda$.

\item Taking $u\rightarrow\infty$ in (\ref{MKWr}), we obtain a shifted reciprocity: if $\core(\mu)=w_0\sigma^{-k}\core(\lambda)$,
\[
\frac{H_\lambda[\sigma^k\iota D_\mu^\bullet]}{H_\lambda[\sigma^k\iota D_{\core(\mu)}^\bullet]}
=\frac{H_\mu[\sigma^k\iota D_\lambda^\bullet]}{H_\mu[\sigma^k\iota D_{\core(\lambda)}^\bullet]}.
\]
\item We define an analogue of the Fourier pairing from \cite{ChereEv} and prove an analogue of \cite[Theorem 1.2]{ChereMM}: for $\lambda$ and $\mu$ with $\core(\mu)=w_0\core(\lambda)$,
\begin{align*}
&\left\langle\TT[X^{(0)}]\nabla_{\core(\mu)}H_\mu, \TT[X^{(0)}]\nabla_{\core(\lambda)}H_\lambda\right\rangle_{q,t}'\\
&=H_\lambda[\iota D_{\mu}^\bullet]H_\mu[\iota D_{\core(\lambda)}^\bullet]=H_\mu[\iota D_\lambda^\bullet]H_\lambda[\iota D_{\core(\mu)}^\bullet].
\end{align*}
\item We define wreath analogues of \textit{interpolation Macdonald polynomials} (cf. \cite{SahiInt, KnopInt}) and give an explicit plethystic formula for them.
\item Given an $r$-tuple of partitions $\vec{\gamma}=(\gamma^{(0)},\gamma^{(1)},\ldots, \gamma^{(r-1)})$, one can define the \textit{multi-Schur function}
\[
s_{\vec{\gamma}} \coloneq  s_{\gamma^{(0)}}[X^{(0)}]s_{\gamma^{(1)}}[X^{(1)}]\cdots s_{\gamma^{(r-1)}}[X^{(r-1)}].
\]
The set $\{s_{\vec{\gamma}}\}$ forms a basis of $\Lambda_{q,t}^{\otimes r}$.
The \textit{wreath $(q,t)$-Kostka coefficients} $\{K_{\vec{\gamma},\lambda}\}$ are then the expansion coefficients in
\[
H_\lambda=\sum_{\vec{\gamma}}K_{\vec{\gamma},\lambda}s_{\vec{\gamma}}.
\]
We give a plethystic formula for $K_{\vec{\gamma},\lambda}$ in the spirit of \cite{GarsiaTesler,GHT}.
\item Finally, we define a wreath analogue of the \textit{global spherical function} from \cite{ChereMM} (cf. \cite{StokC}).
In a similar vein, for each $r$-core $\alpha$, we give the formula in terms of wreath Macdonald polynomials for a series $\mathcal{F}_\alpha[X^\bullet, Y^\bullet]$ that satisfies
\begin{align*}
\mathcal{F}_\alpha[X^\bullet, Y^\bullet]&=\mathcal{F}_{w_0\alpha}[Y^\bullet, X^\bullet]\\
\mathcal{F}_\alpha[X^\bullet, \iota D_\lambda^\bullet]&=H_\lambda\hbox{ if $\core(\lambda)=\alpha$}.
\end{align*}
\end{enumerate}

\subsection{Methods}
Notably missing in our discussion of the wreath case thus far are the wreath analogues of Delta operators.
They do indeed play a major role in this paper, but introducing them requires some machinery.
Hinted at in the work of Varagnolo--Vasserot \cite{VVCyclic} and made precise in \cite{WreathEigen}, the replacement for the EHA in the wreath setting is the \textit{quantum toroidal algebra} $\UTor$.
Namely, let $Q$ be the $A_{r-1}$ root lattice. 
$\UTor$ acts on the space $\mathcal{W} \coloneq \CC(q^{\frac{1}{2}},t^{\frac{1}{2}})\otimes\Lambda^{\otimes r}\otimes\CC[Q]$-- the so-called \textit{vertex representation} \cite{Saito}.
For $\alpha\in Q$, we denote by $e^\alpha\in\CC[Q]$ the corresponding generator.
In \ref{Partitions}, we specify a bijection
\[
\{\hbox{$r$-cores}\}\leftrightarrow Q.
\]
We will abuse notation and treat $\core(\lambda)\in Q$.
It is then natural to situate all the wreath Macdonald polynomials in $\mathcal{W}$ as
\[
H_\lambda\mapsto H_\lambda\otimes e^{\core(\lambda)},
\]
where they form a basis.

In \cite{WreathEigen}, the second author proved that when $r\ge3$, a large commutative subalgebra of $\UTor$ acts diagonally on $H_\lambda\otimes e^{\core(\lambda)}$.
The work \cite{OSW} then produced explicit formulas for some of these operators; when specialized to finitely-many variables, they yield wreath analogues of the Macdonald operators.
These formulas are almost always very complicated.
We catalogue in \ref{WreathOp} the operators used in this paper.
It is interesting that while in the original $r=1$ case, one can write down eigenoperators for $H_\lambda$ without knowledge of the EHA, it seems much less likely that one would organically stumble upon our eigenoperators in the $r>1$ case\footnote{We expect the formulas for the eigenoperators in \ref{WreathOp} to still hold when $r=2$. When $r=2$, the presentation for $\UTor$ is different. One needs to establish analogues of \cite{Tsym} and \cite{WreathEigen}, which at this point should require more precision than ingenuity.}.
Some calculations that are trivial in the $r=1$ case become far more complicated in the wreath setting (cf. \ref{BaseCase}). 

Almost all of our departures from \cite{GHT} are somehow related to the wreath Delta operators.
We highlight some of them below.

\subsubsection{Dual polynomials}
As hinted by the fact that $\nabla_\alpha$ is not necessarily self-adjoint with respect to $\langle - ,-\rangle_{q,t}'$, the $H_\lambda$ are no longer orthogonal under the pairing.
There seems to be no sensible way to fix this, and thus one must introduce dual versions 
\[
H^\dagger_\lambda \coloneq H_\lambda[-\iota X^\bullet ;q^{-1},t^{-1}].
\]
that satisfy, for $\lambda$ and $\mu$ with $\core(\lambda)=\core(\mu)$,
\[
\langle H^\dagger_\lambda, H_\mu\rangle_{q,t}'=0
\]
if and only if $\lambda\not=\mu$.
We note that in the case $r=1$, $H^\dagger_\lambda$ is a scalar multiple of $H_\lambda$ \cite{GarsiaHaimanqtCat}.

Correspondingly, in addition to the $\Delta$-operators, $\Delta[f]$, which act diagonally on $H_\lambda\otimes e^{\core(\lambda)}$, one has adjoint $\Delta^\dagger$-operators, $\Delta^\dagger[f]$, acting diagonally on $H^\dagger_\lambda\otimes e^{w_0\core(\lambda)}$. 
This is responsible for the appearance of $w_0$ in our reciprocity statements.
By writing $\nabla_\alpha$ in terms of $\Delta[f]$ for some $f$, we are able to compute its adjoint and show that it is $\nabla_{w_0\alpha}$.
In particular, this shows that $\nabla_{w_0\alpha}$ has an exotic diagonal basis given by $\{H_\lambda^\dagger \,|\, \core(\lambda)=\alpha\}$.
One can try to define other versions of nabla by using the characters of boxes with color $k\not=0$, but these seem to lack this adjunction property (or something compensating for its lack).
Finally, we note that the analogue of (\ref{VMiki}) roughly takes the form
\[
\Delta^\dagger[f]\mathsf{V}=\mathsf{V}f.
\]

\subsubsection{The series $\mathbb{D}$}
In \cite{GHT}, a significant role is played by the series of operators $D$: applied to $f\in\Lambda$, we obtain
\[
D(f) \coloneq  f\left[X+(1-q)(1-t)z^{-1}\right]\Omega[-zX] = \sum_{k} z^k D_k(f),
\]
 as a series in the variable $z$. This
$D$ is a modified version of the vertex operator from \cite{JingMac}.
In particular, $D_0$ is a $\Delta$-operator.
A critical technical result in \cite{GHT} is its Theorem 2.1: that any $f\in\Lambda_{q,t}$ can be written in terms of iterated applications of $e_1$ (by multiplication) and $D_1$ to $1$.

Here, an analogous role is played by our series $\mathbb{D}$ (cf. \ref{DD*}), which is now a series in $r$ variables $z_0, z_1,\ldots, z_{r-1}$.
To specify an operator, one now feeds it a vector of degrees $\vec{k}=(k_0,k_1,\ldots, k_{r-1})\in\ZZ^r$ to obtain an operator $\mathbb{D}_{\vec{k}}$.
Certain instances yield $\Delta$-operators, while others yield the adjoint $\Delta^\dagger$-operators.
In proving the shifted Tesler identity (Theorem \ref{TeslerThmK}), we introduce shifted $\Delta$- and $\Delta^\dagger$-operators, and $\mathbb{D}$ also yields instances of those as well.
The analogue of \cite[Theorem 2.1]{GHT} is Theorem \ref{InductThm}.

\subsubsection{A Pieri-free proof}
A significant obstacle in the development of wreath Macdonald theory is the lack of Pieri rules.
Multiple pathways to approach them seem to indicate that they can be very difficult to access \cite{OSW, WreathOrth}.
The proofs in \cite{GHT} utilize certain degree 1 Pieri rules.
We are able to sidestep this, which gives a novel approach even in the $r=1$ case.
Instead, in \ref{DeltaSec}, we study the action of certain higher order $\Delta^\dagger$-operators on the span of delta functions $\{\mathbb{E}_\lambda\}$.
In the $r=1$ case, these $\Delta^\dagger$-operators correspond under duality to $e_n$- and $g_n$-Pieri rules \cite{Mac}.
Our strategy is guided by Anton Mellit's original ideas for rethinking the $r=1$ case.

\subsection{Outline}
The paper is organized as follows:
\begin{itemize}
\item Section \ref{WreathSec} introduces the basic definitions for wreath Macdonald polynomials and their structures.
We review the core-quotient decomposition, discuss matrix plethysm (cf. \cite[2.2]{OSWreath}), define the wreath Macdonald polynomials, and introduce the pairing $\langle -, -\rangle_{q,t}'$.
\item Section \ref{TorShuffSec} introduces $\UTor$ and its corresponding shuffle algebra.
The goal is to review the framework of \cite{WreathEigen, OSW} for producing $\Delta$- and $\Delta^\dagger$-operators.
In \ref{FTSec} and \ref{WreathOp}, we introduce $\mathbb{D}$ and the various eigenoperators used in this paper.
Equipped with the formulas for these operators, we hope that others can work on wreath Macdonald polynomials without reference to $\UTor$.
\item Section \ref{TeslerSec} proves Theorem \ref{TeslerThm}.
We are very much intellectually indebted to \cite{GHT}, but as stated in the introduction, the new setting forces significant technical departures.
\item Section \ref{VarSec} proves variations of Theorem \ref{TeslerThm}.
This includes its shifted version, Theorem \ref{TeslerThmK}.
We also prove a ``star'' version involving inverse characters, i.e. $q^{-a}t^{-b}$.
\item Section \ref{Consequence} proves the consequences of Theorem \ref{TeslerThmK}, as outlined in \ref{Apps}.
\end{itemize}

\section{Wreath Macdonald theory}\label{WreathSec}
Throughout this paper, we fix $r>0$.

\subsection{Partitions}\label{Partitions}
We begin by setting the notation for partitions and reviewing the core-quotient decomposition.

\subsubsection{Basic notation}\label{PartNot}
A partition $\lambda=(\lambda_1,\lambda_2,\cdots)$ is a descending sequence of nonnegative integers such that $\lambda_k=0$ for some $k$.
Denote by $\ell(\lambda)$ the greatest index such that $\lambda_{\ell(\lambda)}\not=0$.
Setting
\[
|\lambda|=\sum_{i=1}^{\ell(\lambda)}\lambda_i,
\]
we say $\lambda\vdash n$ if $|\lambda|=n$.
We can also record a partition $\lambda$ using a sequence
$
(1^{m_1}, 2^{m_2},\ldots,)
$
where
\[
m_k \coloneq \#\left\{ i\, |\, \lambda_i=k \right\}.
\]
The \textit{transposed} partition ${}^t\lambda$ has parts
\[
{}^t\lambda_i \coloneq \#\left\{ k\, |\, m_k\ge i \right\}.
\]
For $\lambda,\mu\vdash n$, $\lambda\le\mu$ will denote dominance order: namely that
\[
\lambda_1+\cdots+\lambda_i\le\mu_1+\cdots\mu_i
\]
for all $i$.

To $\lambda$, we associate its \textit{Young diagram}, which we draw following the French convention.
For example, $\lambda=(4,2,2)$ has the Young diagram
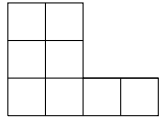
\vspace{.1in}

\centerline{\begin{tikzpicture}[scale=.5]
\draw (0,0)--(4,0)--(4,1)--(2,1)--(2,3)--(0,3)--(0,0);;
\draw (0,1)--(4,1);;
\draw (0,2)--(2,2);;
\draw (1,0)--(1,3);;
\draw (2,0)--(2,1);;
\draw (3,0)--(3,1);;
\end{tikzpicture}}

\vspace{.1in}

\noindent We assign Cartesian coordinates to the boxes in the Young diagram such that the box at the lower left corner has coordinate $(0,0)$.
Thus, $(i,j)\in\lambda$ is in the $(i+1)$-th column and $(j+1)$-th row.
The \textit{content} of $\square=(i,j)$ is the quantity
\[
c_\square \coloneq  j-i
\]
Graphically, it is the SW-NE diagonal that the box lies on.
Note that the diagram of ${}^t\lambda$ is obtained by reflecting $\lambda$ across the content 0 line.
We call the class of $c_\square$ modulo $r$ its \textit{color} and denote it by $\bar{c}_\square$.

Next, we introduce some notation involving $r$.
We will be concerned with $r$-tuples of partitions:
\[
\vec{\lambda}=(\lambda^{(0)},\ldots, \lambda^{(r-1)}).
\]
To avoid confusion, we will index the coordinates within the tuple by superscripts.
It will be useful to view the indexing set of the tuple as $\ZZ/r\ZZ$, so that $\lambda^{(i)} = \lambda^{(j)}$ if $i= j \mod r$.
We denote the size of a tuple by
\[
|\vec{\lambda}| \coloneq \sum_{i\in\ZZ/r\ZZ}|\lambda^{(i)}|
\]
and write $\vec{\lambda}\vdash n$ if $|\vec{\lambda}| = n$.
Concerning a single partition $\lambda$, a \textit{ribbon} of length $r$ is a set of $r$ contiguous boxes on the outer (NE) edge of $\lambda$ whose removal leaves behind another partition. (This is often referred to as a rim-hook in other contexts.)
An $r$-core partition is a partition with no ribbons of length $r$.
An addable corner of a partition is a cell $\square$ for which $\lambda \cup \square$ is a partition. Similarly, a removable corner is a cell $\square$ for which $A-\square$ is a partition.
We denote by $A_{i}(\lambda)$ and $R_i(\lambda)$ the addable and removable boxes of color $i$.

Finally, we discuss characters of cells and partitions.
For $\square=(a,b)$, we denote by $\chi_\square=q^{a}t^{b}$ the character of the cell.
Denote by
\[
B_\lambda^{(i)} \coloneq \sum_{\substack{\square\in\lambda\\ \bar{c}_\square=i}}\chi_\square
\]
the generating function of characters of boxes with color $i$.
The sum over all boxes will be denoted
\[
B_\lambda^\bullet=\sum_{i\in\ZZ/r\ZZ}B_\lambda^{(i)}.
\]
These notations will interact with matrix plethysms in \ref{ColChar} below.

\subsubsection{Young--Maya correspondence}
Another visual representation of a partition is its \textit{Maya diagram}.
A Maya diagram is a function $m:\ZZ\rightarrow\{\pm 1\}$ such that
\[
m(n)=
\begin{cases}
-1 & n\gg0\\
1 & n\ll 0
\end{cases}.
\]
We can visualize $m$ as a string of black and white beads indexed by $\ZZ$: the bead at position $n$ is black if $m(n)=1$ and white if $m(n)=-1$.
These beads will be arranged horizontally with the index increasing towards the \textit{left}.
Between indices $0$ and $-1$, we draw a notch and call it the \textit{central line}.
Notice that beads left of the central line (nonpositive values of $n$) will tend to be white while beads right of the central line (negative values of $n$) will tend to be black.
The \textit{charge} $c(m)$ is the difference in the number of discrepancies:
\[
c(m)=\#\left\{ n\ge 0 \, |\, m(n)=1 \right\}
-\#\left\{n<0\, |\, m(n)= -1\right\}.
\]
The \textit{vacuum diagram} is the Maya diagram with only white beads left of the central line and only white beads right of the central line.

To a partition $\lambda$, we obtain a Maya diagram $m(\lambda)$ via its \textit{edge sequence}---we call this the \textit{Young--Maya correspondence}.
Namely, we tilt $\lambda$ by $45$ degrees counterclockwise and draw the content lines (which are now vertical).
We label by $n$ the gap between content lines $n$ and $n+1$.
Within the gap labeled by $n$, the outer (formerly NE) edge of $\lambda$ either has slope $1$ or $-1$; and $m(\lambda)(n)$ takes the value of the slope of that segment.
Outside of $\lambda$, we assign the default limit assignments (or we can also use a more generous definition of outer edge).
For example, Figure \ref{fig:maya-diagram} gives the Young--Maya correspondence for $\lambda=(4,2,2)$.

\begin{figure}[h]
\begin{tikzpicture}[scale=.5]
\draw[thick] (-2.5,5)--(-0.5,7)--(1.5,5)--(3.5,7)--(4.5,6);;
\draw (-1.5,4)--(0.5,6);;
\draw (-.5,3)--(3.5,7);;
\draw (-2.5,5)--(.5,2)--(4.5,6);;
\draw (-1.5,6)--(1.5,3);;
\draw (1.5,5)--(2.5,4);;
\draw (2.5,6)--(3.5,5);;
\draw (3.5,7)--(4.5,6);;
\draw (-4,0) node {$\cdots$};;
\draw (-3,0) node {3};;
\draw (-3,1) circle (5pt);;
\draw (-2,0) node {2};;
\draw[fill=black] (-2,1) circle (5pt);;
\draw (-1,0) node {1};;
\draw[fill=black] (-1,1) circle (5pt);;
\draw (0,0) node {0};;
\draw (0,1) circle (5pt);;
\draw (.5,.5)--(.5,1.5);;
\draw (1,0) node {-1};;
\draw (1,1) circle (5pt);;
\draw (2,0) node {-2};;
\draw[fill=black] (2,1) circle (5pt);;
\draw (3,0) node {-3};;
\draw[fill=black] (3,1) circle (5pt);;
\draw (4,0) node {-4};;
\draw (4,1) circle (5pt);;
\draw (5,0) node {-5};;
\draw[fill=black] (5,1) circle (5pt);;
\draw (6,0) node {$\cdots$};;
\draw[dashed] (1.5,2)--(1.5,9);;
\draw[dashed] (2.5,2)--(2.5,9);;
\draw[dashed] (3.5,2)--(3.5,9);;
\draw[dashed] (4.5,2)--(4.5,9);;
\draw[dashed] (0.5,2)--(0.5,9);;
\draw[dashed] (-0.5,2)--(-0.5,9);;
\draw[dashed] (-1.5,2)--(-1.5,9);;
\draw[dashed] (-2.5,2)--(-2.5,9);;
\end{tikzpicture}
 \label{fig:maya-diagram}
\caption{The Maya diagram for the partition $(4,2,2)$.}
\end{figure}
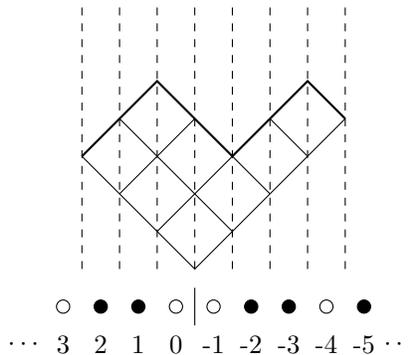

\vspace{.1in}

\begin{prop}\label{YoungMaya}
The Young--Maya correspondence is a bijection between partitions and charge $0$ Maya diagrams.
\end{prop}

\subsubsection{Cores and quotients}\label{CoreQuot}
The decomposition of a partition into its $r$-core and $r$-quotient is a basic ingredient in the definition of wreath Macdonald polynomials.
We describe it here in terms of the Maya diagram.
Given $\lambda$ and $0\le i\le r-1$, we take the following \textit{quotient subdiagrams} of $m(\lambda)$:
\[
m_i(\lambda)(n) \coloneq  m(\lambda)(i+nr).
\]
We call a bead in $m(\lambda)$ that goes to $m_i(\lambda)$ an \textit{$i$-bead}.
This subdiagram may possibly have nonzero charge, $c_i$.
We move the notch in $m_i(\lambda)$ to the left by $c_i$ units (right by $|c_i|$ units if $c_i<0$) and treat it as the central line, obtaining a charge zero Maya diagram.
By Proposition \ref{YoungMaya}, this corresponds to a partition $\lambda^{(i)}$, and we set
\[
\quot(\lambda)=(\lambda^{(0)},\ldots, \lambda^{(r-1)}).
\]
In Figure \ref{fig:quotient}, we show that for $r=3$ and $\lambda=(4,2,2)$, $\quot(\lambda)=( \varnothing, \varnothing, (1,1))$. Here, we have drawn the central line coming from $m(\lambda)$ using solid notches and the $c_i$-shifted central lines using dashed notches.

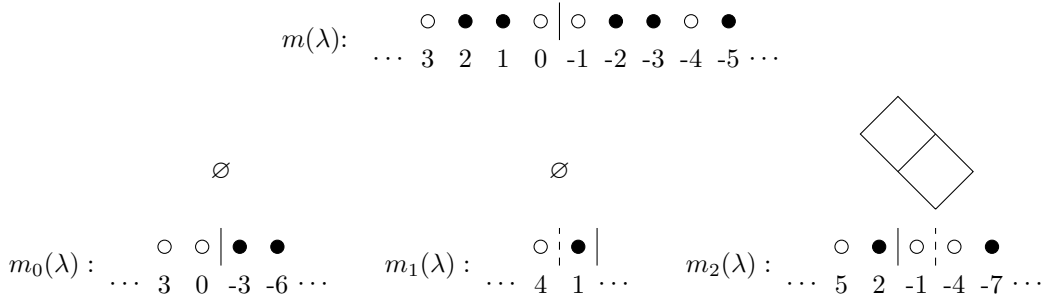
\begin{figure}
\centerline{
\begin{tikzpicture}[scale=.5]
\draw (-8,.5) node {$m(\lambda)$:};;
\draw (-6,0) node {$\cdots$};;
\draw (-5,0) node {3};;
\draw (-5,1) circle (5pt);;
\draw (-4,0) node {2};;
\draw[fill=black] (-4,1) circle (5pt);;
\draw (-3,0) node {1};;
\draw[fill=black] (-3,1) circle (5pt);;
\draw (-2,0) node {0};;
\draw (-2,1) circle (5pt);;
\draw (-1.5,.5)--(-1.5,1.5);;
\draw (-1,0) node {-1};;
\draw (-1,1) circle (5pt);;
\draw (0,0) node {-2};;
\draw[fill=black] (0,1) circle (5pt);;
\draw (1,0) node {-3};;
\draw[fill=black] (1,1) circle (5pt);;
\draw (2,0) node {-4};;
\draw (2,1) circle (5pt);;
\draw (3,0) node {-5};;
\draw[fill=black] (3,1) circle (5pt);;
\draw (4,0) node {$\cdots$};;
\draw (-15,-5.5) node {$m_0(\lambda):$};;
\draw (-13,-6) node {$\cdots$};;
\draw (-12,-6) node {3};;
\draw (-12,-5) circle (5pt);;
\draw (-11,-6) node {0};;
\draw (-11,-5) circle (5pt);;
\draw (-10,-6) node {-3};;
\draw[fill=black] (-10,-5) circle (5pt);;
\draw (-10.5,-3) node {$\varnothing$};;
\draw (-9,-6) node {-6};;
\draw[fill=black] (-9,-5) circle (5pt);;
\draw (-8,-6) node {$\cdots$};;
\draw (-10.5,-5.5)--(-10.5,-4.5);;
\draw (-5,-5.5) node {$m_1(\lambda):$};;
\draw (-3,-6) node {$\cdots$};;
\draw (-2,-6) node {4};;
\draw[dashed] (-1.5,-5.5)--(-1.5,-4.5);;
\draw (-1.5, -3) node {$\varnothing$};;
\draw (-2,-5) circle (5pt);;
\draw (-1,-6) node {1};;
\draw[fill=black] (-1,-5) circle (5pt);;
\draw (-.5,-5.5)--(-.5,-4.5);;
\draw (0,-6) node {$\cdots$};;
\draw (3, -5.5) node {$m_2(\lambda):$};;
\draw (5,-6) node {$\cdots$};;
\draw (6,-6) node {5};;
\draw (6,-5) circle (5pt);;
\draw (7,-6) node {2};;
\draw[fill=black] (7,-5) circle (5pt);;
\draw (7.5,-5.5)--(7.5,-4.5);;
\draw (8.5, -4)--(9.5,-3)--(7.5,-1)--(6.5,-2)--(8.5,-4);;
\draw (7.5,-3)--(8.5,-2);;
\draw (8,-6) node {-1};;
\draw[dashed] (8.5,-5.5)--(8.5,-4.5);;
\draw (8,-5) circle (5pt);;
\draw (9,-6) node {-4};;
\draw (9,-5) circle (5pt);;
\draw (10,-6) node {-7};;
\draw[fill=black] (10,-5) circle (5pt);;
\draw (11,-6) node {$\cdots$};;
\end{tikzpicture}
}
\caption{The quotient decomposition for $\lambda = (4,2,2)$ when $r=3$.}
\label{fig:quotient}
\end{figure}

The $r$-core $\core(\lambda)$ is the result of removing ribbons of length $r$ from $\lambda$ until it is no longer possible.
We can obtain it from $m(\lambda)$ by changing all $m_i(\lambda)$ into the vacuum diagrams centered at the $c_i$-shifted central lines and then reconstituting the total Maya diagram.
For $r=3$ and $\lambda=(4,2,2)$ we obtain $\core(\lambda)= (1,1)$:

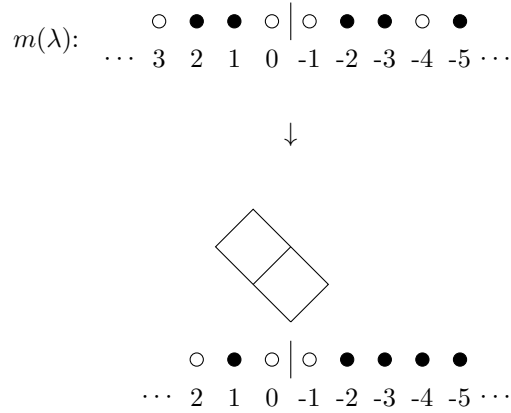
\vspace{.1in}

\centerline{
\begin{tikzpicture}[scale=.5]
\draw (-6,.5) node {$m(\lambda)$:};;
\draw (-4,0) node {$\cdots$};;
\draw (-3,0) node {3};;
\draw (-3,1) circle (5pt);;
\draw (-2,0) node {2};;
\draw[fill=black] (-2,1) circle (5pt);;
\draw (-1,0) node {1};;
\draw[fill=black] (-1,1) circle (5pt);;
\draw (0,0) node {0};;
\draw (0,1) circle (5pt);;
\draw (.5,.5)--(.5,1.5);;
\draw (1,0) node {-1};;
\draw (1,1) circle (5pt);;
\draw (2,0) node {-2};;
\draw[fill=black] (2,1) circle (5pt);;
\draw (3,0) node {-3};;
\draw[fill=black] (3,1) circle (5pt);;
\draw (4,0) node {-4};;
\draw (4,1) circle (5pt);;
\draw (5,0) node {-5};;
\draw[fill=black] (5,1) circle (5pt);;
\draw (6,0) node {$\cdots$};;
\draw (0.5,-2) node {$\downarrow$};;
\draw (-3,-9) node {$\cdots$};;
\draw (-2,-9) node {2};;
\draw (-2,-8) circle (5pt);;
\draw (-1,-9) node {1};;
\draw[fill=black] (-1,-8) circle (5pt);;
\draw (0,-9) node {0};;
\draw (0,-8) circle (5pt);;
\draw (.5,-8.5)--(.5,-7.5);;
\draw (1,-9) node {-1};;
\draw (1,-8) circle (5pt);;
\draw (2,-9) node {-2};;
\draw[fill=black] (2,-8) circle (5pt);;
\draw (3,-9) node {-3};;
\draw[fill=black] (3,-8) circle (5pt);;
\draw (4,-9) node {-4};;
\draw[fill=black] (4,-8) circle (5pt);;
\draw (5,-9) node {-5};;
\draw[fill=black] (5,-8) circle (5pt);;
\draw (6,-9) node {$\cdots$};;
\draw (.5,-7)--(1.5, -6)--(-.5,-4)--(-1.5, -5)--(0.5,-7);;
\draw (-.5,-6)--(.5,-5);;
\end{tikzpicture}
}

\vspace{.1in}

\noindent
We can obtain any $r$-core in this way, and so they are in bijection with the tuples of charges $(c_0,\ldots, c_{r-1})$.
Note that we have
\[
c_0+\cdots+c_{r-1}=0
\]
due to $m(\lambda)$ being charge zero.
Thus, we naturally view $(c_0,\ldots, c_{r-1})$ as an element of the $A_{r-1}$ root lattice $Q$.

\begin{prop}\label{CoreQuotDec}
We have the following:
\begin{enumerate}
\item The core-quotient decomposition gives a bijection
\[
\{\hbox{\rm partitions}\}\leftrightarrow\{\hbox{\rm$r$-core partitions}\}\times\left\{ \hbox{\rm$r$-multipartitions} \right\}
\]
\item The map $\lambda\mapsto (c_0,\ldots, c_{r-1})$ restricts to a bijection between $r$-core partitions and the $A_{r-1}$ root lattice $Q$.
\item The number of addable minus removable boxes of each color can be read from the charges:
\begin{equation}
\#A_i(\lambda)-\#R_i(\lambda)=\delta_{i,0}+c_{i-1}-c_i
\label{ChargeAR}
\end{equation}
\item Adding an equal number boxes of each color does not change the core.
\end{enumerate}
\end{prop}

\subsubsection{Maya symmetries}\label{Reverse}
We can obtain new partitions by applying permutations to the quotient Maya diagrams.
\begin{defn}
Given a partition $\lambda$ with quotient Maya diagrams $\left( m_0(\lambda),m_1(\lambda),\ldots, m_{r-1}(\lambda) \right)$, we view $\pi \in \Sigma_r$ as a permutation of $\{0,1,\dots, r-1\}$ and define $\pi \lambda$ by setting
\[
\left( m_0(\pi\lambda), m_1(\pi\lambda),\ldots, m_{r-1}(\pi\lambda) \right) \coloneq \left( m_{\pi(0)}(\lambda), m_{\pi(1)}(\lambda),\ldots, m_{\pi(r-1)}(\lambda) \right).
\]
In particular, the longest element $w_0 \in \Sigma_r$ gives
\[
\left( m_0(w_0\lambda), m_1(w_0\lambda),\ldots, m_{r-1}(w_0\lambda) \right) \coloneq \left( m_{r-1}(\lambda), m_{r-2}(\lambda),\ldots, m_0(\lambda) \right).
\]
Put another way, we have
$
m_i(w_0\lambda) \coloneq m_{-i-1}(\lambda).
$
For the long cycle $\sigma = (0~1~\cdots ~ r-1)$, we have
\[
\left( m_0(\sigma\lambda),m_1(\sigma\lambda),\ldots, m_{r-1}(\sigma\lambda) \right) \coloneq \left( m_{r-1}(\lambda), m_0(\lambda),m_1(\lambda),\ldots, m_{r-2}(\sigma\lambda) \right).
\]
\end{defn}

Obviously, the two actions from $w_0$ and $\sigma$ send $r$-cores to $r$-cores, and the charges $(c_0,\ldots, c_{r-1})$ undergo the same transformation with respect to their coordinates.
We leave it to the reader to see that in our running example of $r=3$ and $\lambda=(4,2,2)$, we have 
\begin{align*}
w_0\lambda=(6,4), && 
\sigma\lambda= (6,4,1), && \text{ and } &&
\sigma^2\lambda= (5,3).
\end{align*}
Thus, $w_0\lambda\not= {}^t\lambda$, and it is not exactly straightforward to describe the boxes of $w_0\lambda$ in terms of those in $\lambda$.
Nonetheless, we have the following.
\begin{prop}\label{RevProp}
For any $\pi \in \Sigma_r$, the partitions $\lambda$ and $\pi\lambda$ have the same set of color $0$ boxes. Furthermore, if $\{\pi(0),\pi(1),\dots, \pi(k-1)\} = \{0,1,\dots,k-1\}$, then $\lambda$ and $\pi \lambda$ have the same color $k$ boxes.

In particular,
\begin{enumerate}
\item The partitions $\lambda$ and $w_0\lambda$ have the same set of color $0$ boxes.
\item For any $k\in\ZZ/r\ZZ$, $\lambda$ and $\sigma^k\lambda$ have the same set of color $0$ boxes.
\item For any $k\in\ZZ/r\ZZ$, $\lambda$ and $w_0\sigma^{-k}\lambda=\sigma^kw_0\lambda$ have the same set of color $k$ boxes.
\end{enumerate}
\end{prop}
\begin{proof}
Divide $m(\lambda)$ into contiguous segments of length $r$, starting from a $0$-bead on the right and ending with an $(r-1)$-bead on the left. In terms of the Young diagram of $\lambda$, we are considering the sequence of $r$ boundary segments between two content lines with values divisible by $r$. Starting from the SE, such a sequence is a list of north and west steps, $a = a_0 a_1 \cdots a_{r-1}$, starting on the content line $(d-1) r$ for some $d$, and ending on the content line $dr$. 

The action of $\pi \in \Sigma_r$ is given by taking all such segments and sending them to $\pi a = a_{\pi(0)} \cdots a_{\pi(r-1)}$, which has the same number of north and west steps as $a$. Since this is true for every such segment, we have that $a$ in $\lambda$ and $\pi a$ in $\pi \lambda$ start and end at the same point.
But these endpoints determine the boxes of color $0$ in their respective content lines, so $\lambda$ and $\pi\lambda$ have the same color $0$ boxes.

If in particular $a_0\cdots a_{k-1} $ has the same number of north and west steps as $a_{\pi(0)} \cdots a_{\pi(k-1)}$, then the endpoints of these two partial sequences are the same. But the end point after the step $a_{k-1}$ determines the boxes of color $k$. So if $\{\pi(0),\dots, \pi(k-1)\} = \{0,\dots,k-1\}$, then $\lambda$ and $\pi \lambda$ have the same color $k$ boxes.
\end{proof}

\subsection{Symmetric functions}
Let $\Lambda$ denote the ring of symmetric functions.
We recall some standard notation concerning $\Lambda$ \cite{Mac}:
\begin{itemize}
\item $p_k = \sum_i x_i^k$ denotes the $k$th \textit{power sum};
\item $e_n = \sum_{i_1<\dots < i_n} x_{i_1}\cdots x_{i_n}$ denotes the $n$th \textit{elementary symmetric function} and $e_\lambda \coloneq  \prod e_{ {}^t\lambda_i}$;
\item $h_n\sum_{i_1\leq \dots \leq i_n} x_{i_1}\cdots x_{i_n}$ denotes the $n$th \textit{complete homogeneous symmetric function} and $h_\lambda \coloneq \prod h_{\lambda_i}$;
\item $s_\lambda$ denotes the \textit{Schur function} associated to the partition $\lambda$.
\item The \textit{Hall pairing} $\langle -,-\rangle$ on $\Lambda$ has orthonomal basis $\{s_\lambda\}$.
\end{itemize}
Note here that, as opposed to more usual conventions, $e_\lambda$ is not the product of the $e_{\lambda_i}$. There is a distinction for when the subscript is a partition $(k)$ versus a number, so that 
$e_{(k)} = e_1^k \neq e_k = e_{(1^k)}$. This will make our future triangularity results easier to write.

We denote by $\omega$ the plethysm on $\Lambda$ that sends
\[
p_{k}[\omega X] \coloneq  (-1)^{k+1}p_k[X].
\]
It is an involution that exchanges $e_n$ and $h_n$, and more crucial for us, we have
\[
s_\lambda[\omega X]=s_{ {}^t\lambda}[X].
\]

\subsubsection{The Frobenius characteristic}
The irreducible representations of the symmetric group $\Sigma_n$ are indexed by partitions of $n$; given $\lambda\vdash n$, we denote by $V_\lambda$ the corresponding irreducible.
Recall the following basic result about $\Lambda$ (cf. \cite{Mac}):
\begin{thm}\label{FrobCh}
The map $s_\lambda \mapsto [V_\lambda]$ gives an isomorphism 
\begin{align*}
\Lambda&\cong \bigoplus_n \mathrm{Rep}(\Sigma_n)
\end{align*}
of rings,
where the right-hand-side is endowed with the induction product.
Moreover, this isomorphism sends the Hall pairing to the character pairing.
\end{thm}

\subsubsection{Wreath products}
Denote by
\[
\Gamma_n \coloneq \Sigma_n\wr \ZZ/r\ZZ=\Sigma_n\ltimes \left( \ZZ/r\ZZ \right)^n
\]
the wreath product of $\Sigma_n$ with the cyclic group $\ZZ/r\ZZ$.
We will work with a fixed $r$ throughout, so we omit it from our notation.

An analogous result to Theorem \ref{FrobCh} holds for $\Gamma_n$, but now with the tensor product $\Lambda^{\otimes r}$.
The irreducble representations are now indexed by $r$-tuples of partitions with a total of $n$ boxes.
For such an $r$-tuple $\vec{\lambda}$, we denote by $V_{\vec{\lambda}}$ the corresponding irreducible representation.
On the other hand, given $\vec{\lambda}=(\lambda^{(0)},\ldots, \lambda^{(r-1)})$, we can define the \textit{multi-Schur function}
\[
s_{\vec{\lambda}}=s_{\lambda^{(0)}}\otimes s_{\lambda^{(1)}}\otimes\cdots\otimes s_{\lambda^{(r-1)}}\in\Lambda^{\otimes r}.
\] 
We likewise define $e_{\vec{\lambda}}$ and $h_{\vec{\lambda}}$.
We can endow $\Lambda^{\otimes r}$ with the tensor product of the Hall pairing, which we will also denote by $\langle-,-\rangle$.
The multi-Schur basis $\{s_{\vec{\lambda}}\}$ is orthonormal for this pairing.
\begin{thm}[\protect{\cite[I.B]{Mac}}]
The map $s_{\vec{\lambda}}\mapsto [V_{\vec{\lambda}}]$ gives an isomorphism of rings
\begin{align*}
\Lambda^{\otimes r} & \cong \bigoplus_n \mathrm{Rep}(\Gamma_n)
\end{align*}
that intertwines the Hall pairing and the character pairing.
\end{thm}

\subsubsection{Matrix plethysm}\label{MatPleth}
The usual plethystic substitution $f[E]$ of a symmetric function $f\in\Lambda$ at an expression $E(t_1,t_2,\dots)$ is defined by first setting $p_n[E] = E(t_1^n, t_2^n ,\dots)$ and then extending algebraically, since $\{p_n\}$ is a set of algebraically independent generators. 

In the multisymmetric function case, we first index the tensorands in $\Lambda^{\otimes r}$ using $\ZZ/r\ZZ$.
For any $f\in \Lambda$, and $i\in \ZZ/r\ZZ$, we denote by $f[X^{(i)}]\in\Lambda^{\otimes r}$ the element with $f$ in the $i$th tensorand and $1$ elsewhere; note that $X^{(i)} = X^{(j)}$ if $i=j\hbox{ mod } k$.
In particular, we have algebraically independent generators given by $ p_n[X^{(i)}]$.
Plethystic substitution in $\Lambda^{\otimes r}$ can be more complicated because we can take linear maps on $\{p_n[X^{(i)}]\}_{i\in \ZZ/r\ZZ}$.
Specifically, given an $r\times r$ matrix $A=(A_{ij})$ with entries in $\CC(t_1,t_2,\dots)$, we index its coordinates by $\left(\ZZ/r\ZZ\right)^2$ and define (cf. \cite[3.2.2]{OSWreath})
\begin{equation}
p_n[ AX^{(j)} ] \coloneq~ \sum_{i\in \ZZ/r\ZZ}  p_n[ A_{ij} X^{(i)}] ~= ~\sum_{i\in \ZZ/r\ZZ}A_{ij}(t_1^n,t_2^n,\dots) p_n[X^{(i)}].
\label{MatDef}
\end{equation}
For an arbitrary $f\in \Lambda^{\otimes r} \otimes \mathbb{C}(q,t)$, we define its \textit{matrix plethysm} $f[AX^\bullet]$ to be its image under the map generated by (\ref{MatDef}).

We will make heavy use of two particular examples:
\begin{align*}
p_n[\sigma X^{(i)}]& \coloneq  p_n[X^{(i+1)}] \text{ and }\\
p_n[\iota X^{(i)}]& \coloneq  p_n[X^{(-i)}].
\end{align*}

\begin{prop}\label{SigProp}
The following hold:
\begin{enumerate}
\item For any monomial $s$ (in particular, for $s = q,t$), the plethysms $1-s\sigma$ and $1-s\sigma^{-1}$ are invertible, giving
\begin{align}
\nonumber
p_n\left[ \frac{X^{(i)}}{1-s\sigma} \right]&= \frac{p_n[X^{(i)}]+s^n p_n[X^{(i+1)}]+ s^{2n}p_n[X^{(i+2)}]+\cdots+ s^{(r-1)n}p_n[X^{(i-1)}]}{1-s^{rn}} ~\text{ and }\\
p_n\left[ \frac{X^{(i)}}{1-s\sigma^{-1}} \right]&= \frac{p_n[X^{(i)}]+s^n p_n[X^{(i-1)}]+ s^{2n}p_n[X^{(i-2)}]+\cdots+ s^{(r-1)n}p_n[X^{(i+1)}]}{1-s^{rn}}.
\label{OneMinusInv}
\end{align}
\item The transpose of $\sigma$ is its inverse: $\sigma^T=\sigma^{-1}$.
\end{enumerate}
\end{prop}
 
\subsubsection{Colors and characters}\label{ColChar}
Finally, recall our notation $B_\lambda^{(i)}$ and $B_\lambda^\bullet$ introduced at the end of \ref{PartNot}.
We will by default evaluate the variables $X^{(i)}$ at $B_\lambda^{(i)}$, so that the notation
\[
f[B^\bullet_\lambda]
\]
denotes this color-respecting evaluation.
Similarly, when we evaluate a symmetric function at a (possibly infinite) sum of $q,t$-characters,
we treat the color of any $(q,t)$-monomial $q^at^b$ as having color $b-a$ modulo $r$, evaluating with respect to its color.
Note that $1 = q^0 t^0$ has color $0$.
With this set, when performing evaluations, we will include color-changing plethysms such as $f[\iota B_\lambda^\bullet]$ to mean compute $f[\iota X^\bullet]$ first then perform the color-respecting evaluation.

For any sum $S^\bullet$ of characters, we denote by $S^{(i)}$ the partial sum of characters of color $i$.
Thus, multiplication by $q$ shifts the color down, whereas multiplication by $t$ shifts the color up.
An important example of this is
\[
D_\lambda^\bullet \coloneq (1-q)(1-t)B_\lambda^\bullet -1,
\]
where, in particular, one has
\begin{equation}
D_\lambda^{(i)}=qt\sum_{\square\in R_i(\lambda)}\chi_\square-\sum_{\square\in A_i(\lambda)}\chi_\square.
\label{DLambda}
\end{equation}
In this paper, denominators of the form $(1-q^k)^{-1}$ and $(1-t^k)^{-1}$ will appear.
We will by default expand them into power series where $|q|,|t|<1$.

For a sum $S^\bullet$ of characters split into colored summands $\{S^{(i)}\}$, we will denote by $S^{\bullet}_*$ the sum where $q$ and $t$ are inverted: 
\[
S^{(i)}_* \coloneq  S^{(i)}\bigg|_{(q,t)\mapsto (q^{-1},t^{-1})}
\] 
The key point here is that we determine the color of the characters \textit{prior} to inversion.

Another useful substitution we will encounter looks like
\[
\left( \frac{X^{\bullet}}{(1-q \sigma)(1-t\sigma^{-1})} \right)\Bigg|_{X^{(j)} \rightarrow S^{(j)}}.
\]
Since
\begin{align*}
\left( \frac{X^{\bullet}}{(1-q \sigma)(1-t\sigma^{-1})} \right)\Bigg|_{X^{(j)} \mapsto S^{(j)}}
&=   \sum_{a,b \geq 0} q^a t^b X^{(i+a-b)} \Big|_{X^{(j)} \mapsto S^{(j)}}  \\
& =  \sum_{a,b \geq 0} q^a t^b S^{(i+a-b)},
\end{align*}
and $q^a t^b S^{(i+a-b)} $ gives characters of color $i$, we have
\begin{equation}
  \left( \frac{X^{\bullet}}{(1-q \sigma)(1-t\sigma^{-1})} \right)\Big|_{X^{(j)} \mapsto S^{(j)}}
= \left( \frac{S^\bullet}{(1-q)(1-t)} \right)^{(i)}.   
\label{OverMEv}
\end{equation}

\subsection{Wreath Macdonald polynomials}
Let $\Lambda_{q,t}^{\otimes r} \coloneq \Lambda^{\otimes r}\otimes\CC(q,t)$.
To emphasize dependence on $(q,t)$, we may denote $f\in\Lambda_{q,t}^{\otimes r}$ by $f[X^\bullet;q,t]$.

\subsubsection{Definition of $H_\lambda$}
For two partitions $\lambda$ and $\mu$, we will use $\lambda\le_r\mu$ to denote that $\lambda\le\mu$ and $\core(\lambda)=\core(\mu)$.
The following definition is equivalent to the original one proposed by Haiman \cite{Haiman}.

\begin{defn}\label{ModDef}
For $\lambda$ with $\quot(\lambda)\vdash n$, the \textit{modified wreath Macdonald polynomial} $H_\lambda=H_\lambda[X^\bullet;q,t]$ is characterized by the following conditions:
\begin{enumerate}
\item $H_\lambda[(1-q\sigma^{-1})X^\bullet]$ lies in the span of $\left\{ s_{\quot(\mu)}\, \middle|\, \mu\ge_r\lambda  \right\}$;
\item $H_\lambda[(1-t^{-1}\sigma^{-1})X^\bullet]$ lies in the span of $\left\{ s_{\quot(\mu)}\, \middle|\,\mu\le_r\lambda \right\}$;
\item when expanded in terms of $\{s_{\vec{\mu}}\}$, the coefficient of $s_n[X^{(0)}]$ is $1$, i.e. $H_\lambda[1]=1$.
\end{enumerate}
\end{defn}

The existence of $\{H_\lambda\}$ was first proved by Bezrukavnikov--Finkelberg \cite{BezFink} using geometry; an alternative proof is given in \cite{WreathEigen}.
When $\lambda$ is allowed to range over all partitions sharing the same core (so only the quotient varies), we obtain a basis for $\Lambda_{q,t}^{\otimes r}$.
One can view the core as giving an order on multipartitions by constituting them into a single partition and applying dominance order.

\subsubsection{Modified elementary and complete symmetric functions}\label{ModEH}
We can rewrite conditions (1) and (2) of Definition \ref{ModDef} as stating that $H_\lambda$ spans the intersection of two subspaces.
Specifically, define
\begin{align*}
\hat{e}_n^{(p)}& \coloneq  e_n\left[ \frac{X^{(p)}}{1-t^{-1}\sigma^{-1}} \right],&
\hat{e}_{\vec{\lambda}}& \coloneq  e_{\vec{\lambda}}\left[ \frac{X^\bullet}{1-t^{-1}\sigma^{-1}} \right],
\\
\hat{h}_n^{(p)}& \coloneq  h_n\left[ \frac{X^{(p)}}{1-q\sigma^{-1}} \right], \text{ and }&
\hat{h}_{\vec{\lambda}}& \coloneq  h_{\vec{\lambda}}\left[ \frac{X^\bullet}{1-q\sigma^{-1}} \right].
\end{align*}

\begin{prop}[\protect{\cite[Proposition 2.12]{WreathEigen}}]\label{IntDef}
The intersection
\[
\mathrm{span}\big\{ \hat{h}_{\quot(\mu)}\,\big|\, \mu\ge_r\lambda \big\}
\cap
\mathrm{span}\big\{ \hat{e}_{\quot(\mu)}\,\big|\, \mu\le_r\lambda \big\}
\]
is one-dimensional and spanned by $H_\lambda$.
\end{prop}

\subsubsection{The Nabla operator}\label{Nabla}
Fixing an $r$-core $\alpha$, we obtain a basis $\{H_\lambda\, |\,\core(\lambda)=\alpha\}$ of $\Lambda^{\otimes r}_{q,t}$.
We can thus define the operator $\nabla_\alpha$ as follows.
Let $d_0(\lambda)$ be the number of color $0$ boxes in $\lambda$.
For $\lambda$ with $\core(\lambda)=\alpha$,
\begin{equation}
\begin{aligned}
\nabla_\alpha H_\lambda& \coloneq (-1)^{|\quot(\lambda)|}\left( \prod_{\substack{\square\in\lambda\backslash\alpha\\ \bar{c}_\square=0}} \chi_\square\right) H_\lambda\\
&= (-1)^{|\quot(\lambda)|}\frac{e_{d_0(\lambda)}\left[B_\lambda^{(0)}\right]}{e_{d_0(\core(\lambda))}\left[ B_{\core(\lambda)}^{(0)} \right]}H_\lambda
\end{aligned}
\label{NableDef}
\end{equation}

\begin{rem}
The nabla operator we use in this paper is far simpler than the one introduced in \cite[4.5]{OSWreath}.
The latter is attributed to Haiman.
In particular, for $\alpha\not=\varnothing$, the eigenvalue of $\nabla_\alpha$ is not related to the evaluation $H_\lambda[-1]$.
\end{rem}

\subsection{Orthogonality}
Recall that we denote the Hall pairing on $\Lambda^{\otimes r}$ by $\langle -, - \rangle$.
Note that $\sigma$ is \textit{not} self-adjoint with respect to this pairing when $r>2$; it is adjoint to $\sigma^{-1}$.
It is thus advantageous to twist the pairing with $\iota$: we define $\langle -, - \rangle^*$ by
\[
\langle f, g\rangle^* \coloneq \langle f[\iota X^\bullet], g\rangle=\langle f, g[\iota X^\bullet]\rangle
\]
The shift $\sigma$ is then self-adjoint with respect to $\langle - , - \rangle^*$.

\subsubsection{Skewing}
For any $f\in\Lambda_{q,t}^{\otimes r}$, we define the skewing operator $f^\perp$ as the adjoint under $\langle-,-\rangle$ of multiplication by $f$.
Namely,
\[
\langle fg,h\rangle=\langle g, f^\perp h\rangle.
\]
We then have
\[
\left(p_n[X^{(i)}]\right)^\perp=n\frac{\partial}{\partial p_n[X^{(i)}]} \eqcolon p_n^\perp[X^{(i)}].
\]

\subsubsection{The wreath Macdonald pairing}
The following is an analogue of the modified Macdonald pairing:
\begin{defn}\label{PairDef}
The \textit{modified wreath Macdonald pairing} $\langle -, -\rangle_{q,t}'$ is defined by
\begin{align*}
\langle f, g\rangle_{q,t}'& \coloneq \left\langle f, g\left[(1-q\sigma^{-1})(t\sigma-1) X^\bullet \right]\right\rangle^*\\
&= \left\langle f[\iota X^\bullet ;q,t], g\left[(1-q\sigma^{-1})(t\sigma-1) X^\bullet; q,t \right]\right\rangle
\end{align*}
\end{defn}
Because $\sigma$ is self-adjoint with respect to $\langle-,-\rangle^*$, the pairing $\langle -, -\rangle_{q,t}'$ is symmetric.
A key feature (or bug) is that $\{H_\lambda\}$ is no longer an orthogonal basis.
A dual orthogonal basis is given by
\[
H_\lambda^\dagger \coloneq H_{\lambda}[-\iota X^\bullet; q^{-1}, t^{-1}].
\]

\begin{prop}
For $\lambda$ and $\mu$ with $\core(\lambda)=\core(\mu)$, we have
\[
\langle H_\lambda^\dagger, H_\mu\rangle_{q,t}'\not= 0
\]
if and only if $\lambda=\mu$.
\end{prop}

\begin{proof}
If $\lambda\not\ge\mu$, then
\begin{align*}
\langle H^\dagger_\lambda, H_\mu\rangle_{q,t}'
&= \left\langle H_\lambda[-X; q^{-1}, t^{-1}], H_\mu\left[ (1-q\sigma^{-1})(t\sigma-1)X^\bullet \right]\right\rangle/H_{ \lambda}[-1;q^{-1},t^{-1}]\\
&= \left\langle H_\lambda\left[(1-t\sigma^{-1})X^\bullet; q^{-1},t^{-1} \right], H_\mu\left[ (1-q\sigma^{-1})X^\bullet \right]\right\rangle/H_{ \lambda}[-1;q^{-1},t^{-1}]\\
&= 0,
\end{align*}
due to the multi-Schur functions that appear in the definition of $H_\lambda$.
The case $\lambda\not\leq\mu$ is similar.
Finally, $\langle H_\lambda^\dagger,H_\lambda\rangle_{q,t}'\not=0$ because $\langle-,-\rangle_{q,t}'$ is nondegenerate and $\{H_\lambda^\dagger\}$ and $\{H_\lambda\}$ are bases. 
\end{proof}

\begin{rem}
Our pairing differs from the one in \cite{OSWreath} by a plethystic minus.
In \textit{loc. cit.}, the authors compute the value of the pairing $N_\lambda \coloneq \langle H_\lambda^\dagger,H_\lambda\rangle_{q,t}'$.
\end{rem}

\subsubsection{Cauchy kernel}
As usual, let $\Omega[X]$ denote the following infinite sum in $\Lambda$:
\[
\Omega[X]=\exp\left( \sum_{k>0} \frac{p_k[X]}{k} \right)=\sum_{n\ge 0}h_n[X].
\]
The \textit{Cauchy kernel} for the Hall pairing on $\Lambda$ is then given by
\[
\Omega[XY]=\exp\left( \sum_{k>1}\frac{p_k[X]p_k[Y]}{k} \right).
\]
We then have the corresponding kernel for the $\iota$-twisted Hall pairing $\langle-,-\rangle^*$ on $\Lambda^{\otimes r}$:
\begin{align*}
\Omega[X^{-\bullet} Y^{\bullet}]& \coloneq \Omega\left[X^{(0)}Y^{(0)}+X^{(-1)}Y^{(1)}+\cdots +X^{(-r+1)}Y^{(r-1)}\right]\\
&= \Omega [X^\bullet Y^{-\bullet}].
\end{align*}
Note that in this notation, we only take products of variables with negating colors.
When applying plethysm to this kernel, we do so on each summand above.
The Cauchy kernel for $\langle -,-\rangle_{q,t}'$ is then
\begin{align*}
\Omega_{q,t}[X^{-\bullet} Y^{\bullet}]& \coloneq \Omega\left[ X^{-\bullet}\left( \frac{Y^{\bullet}}{(1-q\sigma^{-1})(t\sigma-1)} \right) \right]
\end{align*}
Here, the denominators denote the application of the inverse of $(1-q\sigma^{-1})(1-t\sigma)$ to the given family of variables.
The following is a standard result from the theory of such kernels.

\begin{prop}\label{CauchyProp}
Set
\[
N_\lambda \coloneq \langle H^\dagger_\lambda,H_\lambda\rangle_{q,t}'.
\]
For a given $r$-core $\alpha$, we have
\[
\Omega_{q,t}[X^{-\bullet} Y^{\bullet}]=\sum_{\substack{\lambda\\ \core(\lambda)=\alpha}}\frac{H_\lambda^\dagger[X^\bullet]H_\lambda[Y^\bullet]}{N_\lambda}
=
\sum_{\substack{\lambda\\ \core(\lambda)=\alpha}}\frac{H_\lambda[X^\bullet]H_\lambda^\dagger[Y^\bullet]}{N_\lambda}.
\]
Moreover, $\Omega_{q,t}[X^{-\bullet}Y^\bullet]=\Omega_{q,t}[Y^{-\bullet}X^\bullet]$.
\end{prop}

\section{Quantum toroidal and shuffle algebras}\label{TorShuffSec}
Now, we fix $r>2$.

\subsection{The quantum toroidal algebra}
We will introduce two parameters $(\qqq,\ddd)$ such that
\begin{equation}
\begin{aligned}
q&= \qqq\ddd,&
t&= \qqq\ddd^{-1}
\end{aligned}
\label{qtqd}
\end{equation}
Let $\FF \coloneq \CC(\qqq^{\frac{1}{2}},\ddd^{\frac{1}{2}})$.

\subsubsection{The definition of $\UTor$}
For $i,j\in\ZZ/r\ZZ$, define $a_{i,j}$ and $m_{i,j}$ to be
\begin{equation*}
\begin{aligned}
a_{i,i}&=2,& a_{i,i\pm1}&=-1,& m_{i,i\pm 1}&=\mp 1,\hbox{ and }&a_{i,j}&=m_{i,j}=0\hbox{ otherwise.}
\end{aligned}
\end{equation*}
We then define
\[g_{i,j}(z) \coloneq \frac{\qqq^{a_{i,j}}z-1}{z-\qqq^{a_{i,j}}}.\]
The \textit{quantum toroidal algebra} $\UTor$ is a unital associative $\FF$-algebra generated by 
\[\{e_{i,k},f_{i,k},\psi_{i,k},\psi_{i,0}^{-1},\gamma^{\pm\frac{1}{2}}, \qqq^{\pm d_1},\qqq^{\pm d_2}\}_{i\in\ZZ/r\ZZ}^{k\in\ZZ}.\]
Its relations are described in terms of currents: Let
\begin{align*}
e_i(z)& \coloneq \sum_{k\in\ZZ}e_{i,k}z^{-k},\\
f_i(z)& \coloneq \sum_{k\in\ZZ}f_{i,k}z^{-k}, \text{ and }\\
\psi_i^\pm(z)& \coloneq \psi_{i,0}^{\pm 1}+\sum_{k>0}\psi_{i,\pm k}z^{\mp k}.
\end{align*}
The relations may then be written as follows:
\begin{gather*}
[\psi_i^\pm(z),\psi_j^\pm(w)]=0,~\,\gamma^{\pm\frac{1}{2}}\hbox{ are central},\\
\psi_{i,0}^{\pm1}\psi_{i,0}^{\mp1}=\gamma^{\pm\frac{1}{2}}\gamma^{\mp\frac{1}{2}}=\qqq^{\pm d_1}\qqq^{\mp d_1}=\qqq^{\pm d_2}\qqq^{\mp d_2}=1,\\
\qqq^{d_1}e_i(z)\qqq^{-d_1}=e_i(\qqq z),~\, \qqq^{d_1}f_i(z)\qqq^{-d_1}=f_i(\qqq z),~\, \qqq^{d_1}\psi_i^\pm(z)\qqq^{-d_1}=\psi_i^\pm(\qqq z),\\
\qqq^{d_2}e_i(z)\qqq^{-d_2}=\qqq e_i(z),~\, \qqq^{d_2}f_i(z)\qqq^{-d_2}=\qqq^{-1} f_i(z),~\, \qqq^{d_2}\psi_i^\pm(z)\qqq^{-d_2}=\psi_i^\pm( z),\\
g_{i,j}(\gamma^{-1}\ddd^{m_{i,j}}z/w)\psi_i^{+}(z)\psi_j^{-}(w)=g_{i,j}(\gamma\ddd^{m_{i,j}} z/w)\psi_j^{-}(w)\psi_i^{+}(z),\\
e_i(z)e_j(w)=g_{i,j}(\ddd^{m_{i,j}}z/w)e_j(w)e_i(z),\\
f_i(z)f_j(w)=g_{i,j}(\ddd^{m_{i,j}}z/w)^{-1}f_j(w)f_i(z),\\
(\qqq-\qqq^{-1})[e_i(z),f_j(w)]=\delta_{i,j}\left(\delta(\gamma w/z)\psi_i^+(\gamma^{\frac{1}{2}}w)-\delta(\gamma z/w)\psi_i^-(\gamma^\frac{1}{2}z)\right),\\
\psi_i^\pm(z)e_j(w)=g_{i,j}(\gamma^{\pm\frac{1}{2}}\ddd^{m_{i,j}}z/w)e_j(w)\psi_i^\pm(z),\\
\psi_i^\pm(z)f_j(w)=g_{i,j}(\gamma^{\mp\frac{1}{2}}\ddd^{m_{i,j}}z/w)^{-1}f_j(w)\psi_i^\pm(z),\\
\Sym_{z_1,z_2}[e_i(z_1),[e_i(z_2),e_{i\pm1}(w)]_\qqq]_{\qqq^{-1}}=0,~\,[e_i(z),e_j(w)]=0\hbox{ for }j\not=i,i\pm1,\\
\Sym_{z_1,z_2}[f_i(z_1),[f_i(z_2),f_{i\pm1}(w)]_\qqq]_{\qqq^{-1}}=0,~\,[f_i(z),f_j(w)]=0\hbox{ for }j\not=i,i\pm1.
\end{gather*}
Here, $\delta(z)$ denotes the delta function
\[\delta(z)=\sum_{k\in\ZZ}z^k\]
and $[a,b]_\qqq=ab-\qqq ba$ is the $\qqq$-commutator. 

\subsubsection{Horizontal and vertical subalgebras}
$\UTor$ contains two subalgebras isomorphic to the quantum affine algebra $U_\qqq(\dot{\mathfrak{sl}}_r)$.
Omitting the grading element $D$ of $U_\qqq(\dot{\mathfrak{sl}}_r)$, they are:
\begin{itemize}
\item the subalgebra generated by the constant terms $\{e_{i,0}, f_{i,0}, \psi_{i,0}^{\pm 1}\}_{i\in\ZZ/r\ZZ}$, called the \textit{horizontal subalgebra} and denoted $\dot{U}^h$;
\item the subalgebra generated by the currents indexed by $i=1,\ldots, r-1$, called the \textit{vertical subalgebra} and denoted $\dot{U}^v$.
\end{itemize}
Let $'\ddot{U}'$ denote the subalgebra obtained by omitting the elements $\{\qqq^{d_1},\qqq^{d_2}\}$.
The following beautiful construction of Miki gives an automorphism exchanging the ``two loops'' of the quantum toroidal algebra:
\begin{thm}[\cite{Miki}]\label{MikiAut}
There exists an algebra automorphism $\varpi$ of $'\ddot{U}'$ such that $\varpi(\dot{U}^h)=\dot{U}^v$.
\end{thm}
\noindent Explicit formulas for the images of elements under $\varpi$ are generally quite complicated.
Thanks to existing work, $\varpi$ will only appear conceptually in this paper.

To end, let us note that the currents $\{\psi_{i,0}^{\mp 1}\psi_i^\pm(z)\}_{i\in\ZZ/r\ZZ}$ together with $\gamma^{\frac{1}{2}}$ generate a rank $r$ Heisenberg subalgebra $\mathcal{H}$.
Specifically, if we define logarithmic generators $\{b_{i,n}\}$ by
\[
\psi_{i}^\pm (z)=\psi_{i,0}^{\pm 1}\exp\left( \pm(\qqq-\qqq^{-1})\sum_{n>0}b_{i, \pm n}z^{\mp n} \right),
\]
then the $\{b_{i,n}\}$ satisfy the relation
\begin{equation}
\left[ b_{i,n},b_{j,n'} \right]= \delta_{n,-n'}\frac{\left( \gamma^n-\gamma^{-n} \right)\ddd^{-nm_{i,j}}\left[ na_{i,j} \right]_\qqq}{(\qqq-\qqq^{-1})n},
\label{HeisRel}
\end{equation}
where $[n]_v$ denotes the quantum number
\[
[n]_v=\frac{v^n-v^{-n}}{v-v^{-1}}.
\]
We call $\mathcal{H}$ the \textit{vertical Heisenberg subalgebra}.
There are two suitable notions of \textit{horizontal} Heisenberg subalgebras, namely $\varpi^{-1}(\mathcal{H})$ and $\varpi(\mathcal{H})$.
They seem to be different, and both will appear in this paper.
To relate the two, consider the $\CC(\qqq)$-linear algebra anti-involution $\eta$ on $'\ddot{U}'$ such that
\begin{equation*}
\begin{gathered}
    \eta(\ddd)=\ddd^{-1},\\
\eta(e_{i,k})=e_{i,-k},~\, \eta(f_{i,k})=f_{i,-k},~\, \eta(h_{i,k})=-h_{i,-k},\\
\eta(\psi_{i,0})=\psi_{i,0}^{-1},\, \text{ and }~ \eta(\gamma^{\frac{1}{2}})=\gamma^{\frac{1}{2}}.
\end{gathered}
\end{equation*}
\begin{prop}\cite{Miki}
The Miki automorphism $\varpi$ satisfies $\eta\varpi^{-1}=\varpi\eta$.
\end{prop}

\subsection{Representations}
We will use two representations of $\UTor$ to study wreath Macdonald polynomials.

\subsubsection{Vertex representation}
Recall that we denote the $A_{r-1}$ root lattice by $Q$; denote by $Q^\vee$ the coroot lattice and $(-,-): Q^\vee\times Q\rightarrow\ZZ$ the pairing given by the Cartan matrix, so that
\[
(h_i,\alpha_j)=a_{i,j}, 
\]
where $\alpha_1,\ldots, \alpha_r\in Q$ are the simple roots and $h_1,\ldots, h_r\in Q^\vee$ are the simple coroots.
Additionally, we set
\begin{equation*}
\begin{aligned}
\alpha_0 \coloneq -\sum_{i=1}^{r-1}\alpha_i&& \text{ and }&&
h_0 \coloneq -\sum_{i=1}^{r-1}h_i.
\end{aligned}
\end{equation*}
Finally, let $P$ denote the weight lattice with fundamental weights $\Lambda_1,\ldots, \Lambda_{r-1}$.
The set $\{\alpha_2,\ldots, \alpha_{r-1},\Lambda_{r-1}\}$ is a basis of $P$.

The \textit{twisted group algebra} $\FF\{P\}$ is the $\FF$-algebra generated by $\{e^{\alpha_j}\}_{j=2}^{r-1}\cup\{e^{\Lambda_{r-1}}\}$ satisfying the relations
\begin{align*}
e^{\alpha_i}e^{\alpha_j}= (-1)^{(h_i,\alpha_j)}e^{\alpha_j}e^{\alpha_i},&& \text{ and } &&
e^{\alpha_i}e^{\Lambda_{r-1}}= (-1)^{\delta_{i,r-1}}e^{\Lambda_{r-1}}e^{\alpha_i}.
\end{align*}
For $\alpha\in P$, we write $\alpha=\sum_{j=2}^{r-1}m_j\alpha_j +m_r\Lambda_{r-1}$ and set
\[
e^{\alpha}=e^{m_2\alpha_2}\cdots e^{m_{r-1}\alpha_{r-1}}e^{m_r\Lambda_{r-1}}.
\]
We then set $\FF\{Q\}$ to be the subalgebra generated by $e^{\alpha_1},\ldots, e^{\alpha_{r-1}}$, which has
$\{e^\alpha\}_{\alpha\in Q}$ as a basis.

The vertex representation is defined on the space $\mathcal{W} \coloneq \Lambda^{\otimes r}\otimes\FF\{Q\}$.
To define it, we need to introduce the following additional operators:
\begin{align*}
\TT\left[\pm X^{(i)}z \right]& \coloneq \exp\left(\pm\sum_{k>0}p_k^\perp[X^{(i)}]\frac{z^k}{k}  \right),\\
\partial_{\alpha_i}(e^\alpha)& \coloneq  (h_i,\alpha) e^\alpha, \text{ and }\\
z^{H_{i,0}}(e^{\alpha})& \coloneq z^{(h_i, \alpha)}\ddd^{\frac{1}{2}\sum_{j=1}^{r-1}(h_i,m_j\alpha_j)m_{i,j}}e^{\alpha},
\end{align*}
where $\alpha=\sum_{j=1}^{r-1}m_j\alpha_j$.
Because of the exponential in $\TT[X^{(i)}z]$, we can consolidate products of such operators using plethystic notation:
\[
\TT\left[X^{(i)}z\right]\TT\left[X^{(j)}w\right] \eqcolon\TT\left[ X^{(i)}z+X^{(j)}w \right].
\]
As usual with plethystic notation, a minus sign will not denote a substitution $z\mapsto -1$.

\begin{thm}[\cite{Saito}]
Let $\vec{\ccc}=(\ccc_0,\ldots, \ccc_{r-1})\in(\FF^\times)^r$.
The following define a $'\ddot{U}'$-action on $\mathcal{W}$:
\begin{align*}
\rho_{\vec{\ccc}}(e_i(z))&=\ccc_i\Omega\left[\qqq^{-\frac{1}{2}}X^{(i)}z\right]\TT\left[ -\qqq^{-\frac{1}{2}}\left( -\ddd^{-1}X^{(i-1)}+(\qqq+\qqq^{-1})X^{(i)}-\ddd X^{(i+1)} \right)z^{-1} \right]e^{\alpha_i} z^{1+H_{i,0}},\\
\rho_{\vec{\ccc}}(f_i(z))&=\frac{(-1)^{r\delta_{i,0}}}{\ccc_i}\Omega\left[-\qqq^{\frac{1}{2}} X^{(i)}z \right]
\TT\left[ \qqq^{\frac{1}{2}}\left( -\ddd^{-1}X^{(i-1)}+(\qqq+\qqq^{-1})X^{(i)}-\ddd X^{(i+1)} \right)z^{-1} \right]e^{-\alpha_i} z^{1-H_{i,0}},\\
\rho_{\vec{\ccc}}(\psi_i^+(z))&=\exp\left( (\qqq-\qqq^{-1})
\sum_{k>0}
\frac{[k]_\qqq}{k}\left(-\ddd^{-k}p_k^{\perp}[X^{(i-1)}]+(\qqq^k+\qqq^{-k})p_k^\perp[X^{(i)}]-\ddd^{k}p_k^\perp[X^{(i+1)}] \right)z^{-k} \right)
\qqq^{\partial_{\alpha_i}},\\
\rho_{\vec{\ccc}}(\psi_i^-(z))&=\exp\left( -(\qqq-\qqq^{-1})\sum_{k>0}\frac{[k]_\qqq}{k} p_{k}[X^{(i)}]z^k \right)\qqq^{-\partial_{\alpha_i}}, \text{ and }\\
\rho_{\vec{\ccc}}(\gamma^{\frac{1}{2}})&=\qqq^{\frac{1}{2}}.
\end{align*}
\end{thm}

\begin{rem}\label{VertexRem}
To reconcile our formulas above with the usual way the vertex representation is presented, we note
\begin{equation}
\rho_{\vec{\ccc}}(b_{i,k})=
\begin{cases}
\displaystyle\frac{[k]_\qqq}{k} p_{|k|}[X^{(i)}]& k<0\\
\vspace{-1em} \\
\displaystyle\frac{[k]_\qqq}{k}\left(-\ddd^{-k}p_k^{\perp}[X^{(i-1)}]+(\qqq^k+\qqq^{-k})p_k^\perp[X^{(i)}]-\ddd^{k}p_k^\perp[X^{(i+1)}] \right)& k>0
\end{cases}.
\label{VerSym}
\end{equation}
Together with $\rho_{\vec{\ccc}}(\gamma^{\frac{1}{2}})$, these indeed satisfy the Heisenberg relations (\ref{HeisRel}).
\end{rem}

\subsubsection{Fock representation}
The second representation is commonly called the Fock representation.
Let $\mathcal{F}$ be the $\FF$-vector space with basis indexed by partitions $\{|\lambda\rangle\}$.
The corresponding dual basis vectors will be denoted $\{\langle\lambda|\}$.
We recall that for a given partition $\lambda$,
\begin{itemize}
\item $d_i(\lambda)$ is the number of boxes in $\lambda$ with color $i$; 
\item $A_i(\lambda)$ and $R_i(\lambda)$ denote the addable and removable boxes of $\lambda$ with color $i$, respectively; and
\item $\bar{c}_\square$ is the color of $\square \in \lambda$.
\end{itemize}

\begin{thm}[\cite{FJMMRep,WreathEigen}]
Let $v\in\FF^\times$.
The following matrix elements define an action $\tau_v$ of $'\ddot{U}'$ on $\mathcal{F}$:
 \begin{gather*}
\begin{aligned}
\langle\lambda | \tau_v\left(e_i(z)\right)|\lambda+\square\rangle&=\delta_{\bar{c}_\square,i}(-\ddd)^{d_{i+1}(\lambda)}\delta\left( \frac{ z}{\chi_\square v}\right)
\frac{\displaystyle\prod_{\blacksquare\in R_{i}(\lambda)}\left( \chi_\square-\qqq^2\chi_\blacksquare \right)}
{\displaystyle\prod_{\substack{\blacksquare\in A_{i}(\lambda)\\\blacksquare\not=\square}}\left(\chi_\square-\chi_\blacksquare\right)},\\
\langle\lambda+\square |\tau_v\left(f_i(z)\right)|\lambda\rangle&=\delta_{\bar{c}_\square,i}(-\ddd)^{-d_{i+1}(\lambda)}\delta\left(\frac{ z}{\chi_\square v}\right)
\frac{\displaystyle\prod_{\substack{\blacksquare\in A_{i}(\lambda)\\\blacksquare\not=\square}}\left( \qqq\chi_\square-\qqq^{-1}\chi_\blacksquare \right)}
{\displaystyle\prod_{\blacksquare\in R_{i}(\lambda)}\qqq\left( \chi_\square-\chi_\blacksquare \right)},\\
\langle\lambda|\tau_v\left(\psi_i^\pm(z)\right)|\lambda\rangle&=
\prod_{\blacksquare\in A_{i}(\lambda)}\frac{\left(\qqq z-\qqq^{-1}\chi_{\blacksquare}v\right)}{\left( z-\chi_\blacksquare v\right)}
\prod_{\blacksquare\in R_{i}(\lambda)}\frac{\left(\qqq^{-1} z-\qqq\chi_\blacksquare v\right)}{\left( z-\chi_\blacksquare v\right)}, \text{ and }\\
\langle\lambda|\tau_v(\gamma^{\frac{1}{2}})|\lambda\rangle&=1.
\end{aligned}
\end{gather*}
\end{thm}

It will be useful to compute the following:

\begin{lem}\label{BEigen}
For $k>0$,the Heisenberg generators $\{b_{i,\pm k}\}$ act on $\mathcal{F}$ as
\begin{equation*}
\begin{aligned}
\langle\lambda|\tau_v(b_{i,k})|\lambda\rangle 
&= \frac{v^{k}[ k]_{\qqq}}{\qqq^{ k}k}\left( \sum_{\blacksquare\in A_{i}(\lambda)}\chi_\blacksquare^{k} - \sum_{\blacksquare\in R_{i}(\lambda)}(qt\chi_\blacksquare)^{k}\right)
=\frac{v^{k}[ k]_{\qqq}}{\qqq^{ k}k}p_k[-D_\lambda^{(i)}] \text{ and }\\
\langle\lambda|\tau_v(b_{i,-k})|\lambda\rangle 
&= \frac{\qqq^{k}[ k]_{\qqq}}{v^{k}k}\left( \sum_{\blacksquare\in A_{i}(\lambda)}\chi_\blacksquare^{-k} - \sum_{\blacksquare\in R_{i}(\lambda)}(qt\chi_\blacksquare)^{-k}\right)
=\frac{\qqq^{k}[ k]_{\qqq}}{v^{k}k}p_k\left[-(D_\lambda^{(i)})_*\right].
\end{aligned}
\end{equation*}
\end{lem}

\subsubsection{Tsymbaliuk isomorphism}
The representation $\tau_v$ on $\mathcal{F}$ is generated as a module by $|\varnothing\rangle$, whereas on the other hand, $\rho_{\vec{\ccc}}$ is generated by $1\otimes 1\in \mathcal{W}$.
Both can be considered vacuum vectors for their respective representations.
The following theorem was proved by Tsymbaliuk.
\begin{thm}[\cite{Tsym,WreathEigen}]\label{TsymIso}
Let
\begin{equation*}
v=\frac{(-1)^{\frac{(r-2)(r-3)}{2}}\qqq}{\ddd^{\frac{r}{2}}\ccc_0\cdots \ccc_{r-1}}.
\label{FockVertexPar}
\end{equation*}
The vacuum-to-vacuum map
\begin{equation}
\mathcal{F}\ni|\varnothing\rangle\mapsto 1\otimes 1\in \mathcal{W}
\label{VacToVac}
\end{equation}
induces an isomorphism $\mathrm{T}:\mathcal{F}\rightarrow \mathcal{W}$ between the 
representation $\tau_v$ and the $\varpi$-twisted representation $\rho_{\vec{\ccc}}\circ\varpi$.
\end{thm}

Recall that in Proposition \ref{CoreQuotDec}, we give a bijection between $r$-cores and $Q$.
We will abuse notation and write $\core(\lambda)\in Q$.
The following result provides more details on the Tsymbaliuk isomorphisms:
\begin{thm}[\cite{WreathEigen}]\label{Eigenstates}
For a nonzero scalar $c_\lambda$, we have:
\[\mathrm{T}\left(|\lambda\rangle\right)=c_\lambda H_\lambda\otimes e^{\core(\lambda)}.\]
Therefore, $\{H_\lambda\otimes e^{\core(\lambda)}\}$ is a diagonal basis for $(\rho_{\vec{\ccc}}\circ\varpi)(\mathcal{H})$.
\end{thm}

\subsubsection{Plethysm for $\mathcal{T}$-operators}
Note that
\begin{equation}
\mathcal{T}\left[ \pm X^{(i)}z \right]\left(f[X^{(j)}]\right)=
\begin{cases}
f[X^{(i)}\pm z] & i=j\\
f[X^{(j)}] & i\not=j
\end{cases},
\label{TOp}
\end{equation}
which can be checked on power sums.
We can extend the matrix plethysm notation \ref{MatPleth} to $\mathcal{T}$-operators:
\[
\mathcal{T}\left[ A X^{(j)} \right] \coloneq \mathcal{T}\left[ \sum_{i\in\ZZ/r\ZZ}A_{ij}X^{(i)} \right].
\]
Below is a lemma that makes this plethystic notation more viable to work with.
\begin{lem}\label{TOLem}
For two $r\times r$ matrices $A$ and $B$ with entries in $\CC(q,t)$ and $f\in \Lambda_{q,t}^{\otimes r}$, we have
\[
\mathcal{T}\left[ AX^{(j)}z \right]\Omega\left[ BX^{(k)}w \right]=\Omega\left[\left(A^TB\right)_{jk}zw\right]\Omega\left[ BX^{(k)}w \right]\mathcal{T}\left[ AX^{(j)}z \right]
\]
\end{lem}

\begin{proof}
Writing things out, we have
\begin{align*}
\mathcal{T}\left[ AX^{(j)}z \right]= \mathcal{T}\left[ \sum_{i\in\ZZ/r\ZZ}A_{ij}X^{(i)}z \right] && \text{ and } &&
\Omega\left[ BX^{(k)}w \right]= \Omega\left[ \sum_{i\in\ZZ/r\ZZ}B_{ik}X^{(i)}w \right].
\end{align*}
Thus, applying (\ref{TOp}), we have
\begin{align*}
\mathcal{T}\left[ AX^{(j)}z \right] \Omega\left[BX^{(k)}w\right] &= \Omega\left[\sum_{i\in\ZZ/r\ZZ}B_{ik}\left(X^{(i)}+A_{ij}z\right)w\right] \mathcal{T}\left[ AX^{(j)}z \right]\\
&= \Omega\left[ \sum_{i\in\ZZ/r\ZZ}B_{ik}A_{ij}zw \right]\Omega\left[ BX^{(k)}w \right]\mathcal{T}\left[ AX^{(j)}z \right]\\
&= \Omega\left[ (A^TB)_{jk}zw \right]\Omega\left[ BX^{(k)}w \right]\mathcal{T}\left[ AX^{(j)}z \right].\qedhere
\end{align*}
\end{proof}

\subsubsection{Rescaled vertex operators}
Let us rescale $\mathcal{W}$ by
\[
p_k[X^{(i)}]\mapsto \qqq^{-\frac{k}{2}}p_k[X^{(i)}].
\]
This yields an isomorphic representation $\rho_{\vec{\ccc}}^+$.
We will only be concerned with the action of the $e$-currents:
\begin{align*}
E_i(z)& \coloneq \rho_{\vec{\ccc}}^+(e_i(z))\\
&=
\ccc_i\Omega\left[X^{(i)}z \right]
\TT\left[-\left( -q^{-1}X^{(i-1)}+(1+q^{-1}t^{-1})X^{(i)}-t^{-1} X^{(i+1)} \right)z^{-1} \right]e^{\alpha_i} z^{1+H_{i,0}}\\
&=
 \ccc_i\Omega\left[X^{(i)}z \right]
\TT\left[-(1-q^{-1}\sigma^{-1})(1-t^{-1}\sigma)X^{(i)}z^{-1} \right]e^{\alpha_i} z^{1+H_{i,0}}.
\end{align*}
Likewise, we define $\rho_{\vec{\ccc}}^-$ by rescaling
\[
p_k[X^{(i)}]\mapsto \qqq^{\frac{k}{2}}p_k[X^{(i)}].
\]
We will only be interested in the action of the $f$-currents:
\begin{align*}
F_i(z)& \coloneq \rho_{\vec{\ccc}}^-(f_i(z))\\
&=
\frac{(-1)^{r\delta_{i,0}}}{\ccc_i}\Omega\left[-X^{(i)}z \right]
\TT\left[ \left( -tX^{(i-1)}+(1+qt)X^{(i)}-q X^{(i+1)} \right)z^{-1} \right]e^{-\alpha_i} z^{1-H_{i,0}}\\
&= 
\frac{(-1)^{r\delta_{i,0}}}{\ccc_i}\Omega\left[-X^{(i)}z \right]
\TT\left[ (1-q\sigma)(1-t\sigma^{-1})X^{(i)}z^{-1} \right]e^{-\alpha_i} z^{1-H_{i,0}}.
\end{align*}
Since we have only rescaled homogeneous elements, in particular $\{H_\lambda\}$, the diagonalization statement of Theorem \ref{Eigenstates} still holds for $\rho_{\vec{\ccc}}^\pm$.

With this rescaling, the components of $\rho_{\vec{\ccc}}^-(f_i(z))$ have nice adjoints under $\langle -, -\rangle'_{q,t}$.
\begin{lem}\label{FAdjoint1}
For $f,g\in\Lambda^{\otimes r}$, we have the following adjunction relation:
\[
\left\langle f, \Omega\left[-X^{(i)}z \right]g \right\rangle_{q,t}'
=
\left\langle \TT\left[ (1-q\sigma)(1-t\sigma^{-1})X^{(-i)}z \right]f,g\right\rangle_{q,t}'.
\]
Here, we view both sides as power series in the variable $z$.
\end{lem}

\begin{proof}
It suffices to show
\[
\left\langle f, -p_n[X^{(i)}]g\right\rangle_{q,t}'
=
\left\langle \left(-t^np_n^\perp[X^{(-i-1)}]+(1+q^nt^n)p_n^\perp[X^{(-i)}]-q^np_n^\perp[X^{(-i+1)}]\right)f, g\right\rangle_{q,t}'.
\]
This in turn follows from
\[
-p_n[\iota(1-q\sigma^{-1})(t\sigma-1)X^{(i)}]=-t^np_n[X^{(-i-1)}]+(1+q^nt^n)p_n[X^{(-i)}]-q^np_n[X^{(-i+1)}].
\qedhere
\]
\end{proof}

\subsubsection{Extending the pairing}\label{PairCores}
We will extend the wreath Macdonald pairing $\langle-,-\rangle_{q,t}'$ $\FF$-linearly from $\Lambda_{q,t}^{\otimes r}$ to $\mathcal{W}$ as follows.
Viewing $\alpha\in Q$ as $\alpha=(a_0,\ldots,a_{r-1} )\in\ZZ^r$, define its \textit{reverse} $w_0\alpha$ by:
\[
w_0\alpha \coloneq  (a_{r-1},\ldots, a_{0})
\]
Recall that in \ref{Reverse}, we have also defined reverse map $w_0\lambda$ on partitions.
The two notions of reverse agree when we regard $\core(\lambda)\in Q$ via its vector of charges $\core(\lambda)=(c_0,\ldots, c_{r-1})$.

We now set
\[
\langle f\otimes e^\alpha, g\otimes e^{\beta}\rangle_{q,t}'=\delta_{\alpha, w_0\beta}\langle f,g\rangle_{q,t}'.
\]
Thus, $\{H_\lambda\otimes e^{\core(\lambda)}\}$ is a basis of $\mathcal{W}$ with dual basis $\{H_\lambda^\dagger \otimes e^{\core(w_0\lambda)}\}$.
Recall that in Lemma \ref{FAdjoint1}, the $\Omega$- and $\TT$-components of $F_i(z)$ go to those of $F_{-i}(z^{-1})$ under adjunction.
Our choice of extension does something similar for the $z^{H_{i,0}}$ component:
\begin{lem}\label{FAdjoint2}
We have
\[
\left\langle f\otimes e^\alpha, \prod_{i=0}^{r-1}z_i^{H_{i,0}}(g\otimes e^\beta)\right\rangle_{q,t}'=
\left\langle \prod_{i=0}^{r-1}z_i^{-H_{-i,0}}(f\otimes e^\alpha), g\otimes e^\beta\right\rangle_{q,t}'.
\]
\end{lem}

\begin{proof}
For $\beta=(b_0,\ldots, b_{r-1})$, first observe that
\[
z_0^{H_{0,0}}\cdots z_{r-1}^{H_{r-1,0}}(g\otimes e^\beta)=z_0^{b_{r-1}-b_0}z_1^{b_{0}-b_1}\cdots z_{r-1}^{b_{r-2}- b_{r-1}}g\otimes e^\beta.
\]
The only nontrivial part of this calculation is to check that the powers of $\ddd$ in the action of $H_{i,0}$ cancel.
This is because we have exactly one term per color: if $\alpha=\sum_j m_j\alpha_j$, then
\[
\sum_{i}\langle h_i, m_j\alpha_j\rangle m_{i,j}=-\langle\alpha_{j-1},m_j\alpha_j\rangle+\langle\alpha_{j+1},m_j\alpha_{j}\rangle=0.
\]
We likewise then have
\[
z_0^{-H_{0,0}}z_1^{-H_{r-1,0}}\cdots z_{r-1}^{-H_{1,0}}(f\otimes e^{w_0\beta})=z_{0}^{b_{r-1}-b_0}z_1^{b_0-b_1}\cdots z_{r-1}^{b_{r-2}-b_{r-1}}f\otimes e^{w_0\beta}.\qedhere
\]
\end{proof}

\begin{rem}
We do not obtain that each $F_i(z)$ is adjoint to $F_{-i}(z^{-1})$.
For this to hold, we would need to twist the pairing to invert $\ddd$.
At best, we can make such a statement for products of $\{F_i(z)\}$ that have an equal number of each color.
All the operators we consider in this paper are like this. 
\end{rem}

\subsection{Shuffle algebra}
The shuffle algebra gives a different way to view $\UTor$ and its actions on $\mathcal{W}$ and $\mathcal{F}$.
In particular, we will use it to work with elements of the horizontal Heisenberg subalgebras $\varpi(\mathcal{H})$ and $\varpi^{-1}(\mathcal{H})$, which are otherwise inaccessible via generators and relations.

\subsubsection{Definition of $\Sss$}
For $\vec{n}=(n_0,\ldots, n_{r-1})\in\ZZ_{\ge 0}^r$, consider the space of rational functions
\[
\FF\left( z_{i,a} \right)_{i\in\ZZ/r\ZZ}^{1\le a\le n_i}.
\]
We will call the index $i$ of $z_{i,a}$ the \textit{color} of the variable $x_{i,a}$.
Let $\Sigma_{\vec{n}} \coloneq \prod_{i=0}^{r-1}\Sigma_{n_i}$ and consider its action on the above ring whereby the factor $\Sigma_{n_i}$ only permutes the variables of color $i$.
We then set
\begin{align*}
\mathbb{S}_{\vec{n}} \coloneq \left[\FF\left( z_{i,a} \right)_{i\in\ZZ/r\ZZ}^{1\le a\le n_i}\right]^{\Sigma_{\vec{n}}}&& \text{ and } &&
\mathbb{S} \coloneq  \bigoplus_{\vec{n}\in\ZZ_{\ge 0}^r}\mathbb{S}_{\vec{n}}.
\end{align*}
Thus, $\mathbb{S}_{\vec{n}}$ is a space of \textit{color-symmetric} rational functions.

We endow the direct sum $\mathbb{S}$ with the \textit{shuffle product}.
To define this product, we need the \textit{mixing terms}: for $i,j\in\ZZ/r\ZZ$, set
\[
\omega_{i,j}(z,w) \coloneq 
\begin{cases}
\left(z-\qqq^{2}w\right)^{-1}\left(z-w\right)^{-1} & i=j,\\
\left(\qqq w-\ddd^{-1}z\right) &i+1=j,\\
 \left(z-\qqq\ddd^{-1} w\right) &i-1=j, \text{ and }\\
1 &\hbox{otherwise.}
\end{cases}
\]
For $F\in \mathbb{S}_{\vec{n}}$ and $G\in \mathbb{S}_{\vec{m}}$, we define the shuffle product $F\star G\in \mathbb{S}_{\vec{n}+\vec{m}}$ to be
\[
F\star G \coloneq 
\mathrm{Sym}_{\vec{n}+\vec{m}}\left( 
F\bigg( \{z_{i,a}\}_{a=1}^{n_i} \bigg)
G\bigg( \{z_{j,b}\}_{b=n_j+1}^{n_j+m_j} \bigg)
\prod_{i,j\in\ZZ/r\ZZ}
\prod_{\substack{1\le a\le n_i\\ n_j+1\le b\le n_j+m_j}}
\omega_{i,j}(z_{i,a}, z_{j,b})
\right)
\]
Here, $\mathrm{Sym}_{\vec{n}+\vec{m}}$ denotes the symmetrization for the $\Sigma_{\vec{n}+\vec{m}}$ action.

Finally, we consider the subspace $\Sss_{\vec{n}}\subset \mathbb{S}_{\vec{n}}$ of functions $F$ satisfying the following two conditions:
\begin{enumerate}
\item \textit{Pole conditions:} $F$ is of the form
\begin{equation}
F=\frac{f(\{z_{i,a}\})}{\displaystyle \prod_{i\in\ZZ/r\ZZ}\,\prod_{\substack{1\le a, b\le n_i\\a\not= b}}(z_{i,a}-\qqq^2z_{i,b})}
\label{PoleCond}
\end{equation}
for a color-symmetric Laurent polynomial $f$.
\item \textit{Wheel conditions:} $F$ has a well-defined finite limit when
\[\frac{z_{i,a_1}}{z_{i+\epsilon,b}}\rightarrow\qqq\ddd^\epsilon\hbox{ and }\frac{z_{i+\epsilon,b}}{z_{i,a_2}}\rightarrow\qqq\ddd^{-\epsilon}\]
for any choice of $i$, $a_1$, $a_2$, $b$, and $\epsilon$, where $\epsilon\in\{\pm 1\}$. 
This is equivalent to specifying that the Laurent polynomial $f$ in the formula (\ref{PoleCond}) evaluates to zero.
\end{enumerate}
Define $\Sss$ to be the direct sum
\[\Sss \coloneq \bigoplus_{\vec{n}\in(\ZZ_{\ge 0})^{r}}\Sss_{\vec{n}}.\]

\begin{prop}
The product $\star$ is associative and $\Sss$ is closed under $\star$.
\end{prop}

We call $\Sss$ the \textit{shuffle algebra of type $\hat{A}_{r-1}$}.
The relation between $\Sss$ and $\UTor$ is given by the following result of Negu\cb{t}:

\begin{thm}[\cite{NegutTor}]\label{NegutShuffIso}
Let 
\begin{itemize}
\item $\ddot{U}^+\subset\UTor$ denote the subalgebra generated by $\{e_{i,k}\}_{i\in\ZZ/r\ZZ}^{k\in \ZZ}$ and
\item $\ddot{U}^-\subset\UTor$ denote the subalgebra generated by $\{f_{i,k}\}_{i\in\ZZ/r\ZZ}^{k\in\ZZ}$.
\end{itemize}
$\Sss$ is generated by $\{z_{i,1}^k\}_{i\in\ZZ/\ell\ZZ}^{k\in\ZZ}$ and the maps
\begin{align*}
\Psi_+\left(z_{i,1}^k\right)=e_{i,k}&& \text{ and } &&
\Psi_-\left(z_{i,1}^k\right)=f_{i,k}
\end{align*}
define algebra isomorphisms $\Psi_+:\Sss\rightarrow\ddot{U}^+$ and $\Psi_-:\Sss^{\mathrm{op}}\rightarrow\ddot{U}^-$.
\end{thm}

\begin{prop}[\cite{OSW}]\label{EtaShuff}
For $F\in\mathcal{S}_{\vec{n}}$, we have
\begin{equation*}
\Psi_+^{-1}\eta\Psi_+(F)=\Psi_-^{-1}\eta\Psi_-(F)= 
 \left.F\left(\{z_{i,a}^{-1}\}\right)\prod_{i\in \ZZ/r\ZZ}\prod_{a=1}^{n_i}(-\ddd)^{n_{i+1}n_i} z_{i,a}^{n_{i+1}+n_{i-1}-2(n_i-1)}\right|_{\ddd\mapsto\ddd^{-1}}.
\end{equation*}
\end{prop}

\subsubsection{Action on vertex representation}
For a Laurent series $f$, denote by $\{f\}_0$ its constant term.
Moreover, for a set of noncommuting operators $a_1,\ldots, a_n$, we define the ordered products
\begin{equation*}
\begin{aligned}
\overset{\curvearrowright}{\prod_{i=1}^n} \coloneq  a_1\cdots a_n,&& \text{ and } &&
\overset{\curvearrowleft}{\prod_{i=1}^n} \coloneq  a_n\cdots a_1.
\end{aligned}
\end{equation*}
The following will allow us to compute the action of shuffle elements on $\mathcal{W}$.

\begin{prop}[\cite{OSW}]\label{ShuffVert}
For $F\in\Sss_{\vec{n}}$ and $f\otimes e^\alpha\in\mathcal{W}$, we have
\begin{equation*}
(\rho^+_{\vec{c}}\circ\Psi_+)(F)(f\otimes e^\alpha)
=\left\{\frac{\displaystyle F\times\overset{\curvearrowright}{\prod_{i=0}^{r-1}}\,\overset{\curvearrowright}{\prod_{a=1}^{n_i}}E_i(z_{i,a})(f\otimes e^\alpha)}
{\displaystyle\left(\prod_{i\in I}\prod_{1\le a<a'\le n_i}\omega_{i,i}(z_{i,a},z_{i,a'})\right)
\left(\prod_{0\le i< j\le r-1}\,\prod_{\substack{1\le a\le n_i\\1\le b\le n_j}}\omega_{i,j}(z_{i,a},z_{j,b})\right)}\right\}_0,
\end{equation*}
where the rational function on the right hand side is expanded into a Laurent series assuming
\begin{equation*}
\begin{aligned}
|z_{i,a}|= 1,&&
|\qqq|>1,&& \text{ and }&&
|\ddd|=1.
\end{aligned}
\end{equation*}
On the other hand, we have
\begin{equation*}
(\rho_{\vec{\ccc}}^-\circ\Psi_-)(F)(f\otimes e^\alpha)
=\left\{\frac{\displaystyle F\times\overset{\curvearrowleft}{\prod_{i=0}^{r-1}}\,\overset{\curvearrowleft}{\prod_{a=1}^{n_i}}F_i(z_{i,a})(f\otimes e^\alpha)}
{\displaystyle\left(\prod_{i\in \ZZ/r\ZZ}\prod_{1\le a<a'\le n_i}\omega_{i,i}(z_{i,a},z_{i,a'})\right)
\left(\prod_{0\le i< j\le r-1}\,\prod_{\substack{1\le a\le n_i\\1\le b\le n_j}}\omega_{i,j}(z_{i,a},z_{j,b})\right)}\right\}_0,
\end{equation*}
where the rational function on the right hand side is expanded into a Laurent series assuming
\begin{equation*}
\begin{aligned}
|z_{i,a}|=1,&&
|\qqq|<1, && \text{ and } &&
|\ddd|=1.
\end{aligned}
\end{equation*}
\end{prop}

\subsection{Feigin--Tsymbaliuk elements}\label{FTSec}
One set of elements in $\UTor$ to which we will apply the shuffle algebra machinery are those found by Feigin--Tsymbaliuk in \cite{FeiTsym, Tsym} as \textit{vacuum expectations of $L$-operators}.

\subsubsection{Dual bosons}
For a fixed $n$, the matrix 
\[ \left(
\frac{\ddd^{-nm_{i,j}}[na_{i,j}]_\qqq}{(\qqq-\qqq^{-1})n} \right)_{i,j}
\]
is invertible.
Thus, by (\ref{HeisRel}), we can define \textit{dual bosons} $\{b_{i,n}^*\}$ satisfying
\begin{equation}
\left[ b_{i,n}^*, b_{j,n'}  \right]=\delta_{n,-n'}\delta_{i,j}\left( \gamma^{n}-\gamma^{-n} \right).
\label{DualBoson}
\end{equation}

\begin{lem}
For $n>0$, we have
\begin{align}
\label{DualForm}
b_{i,n}^*&=
\frac{\qqq^{n}(\qqq-\qqq^{-1})n}{(1-q^{nr})(1-t^{nr})[n]_\qqq}
\sum_{j,k=0}^{r-1} q^{nj}t^{nk}b_{i+j-k, n},
\text{ and }\\
\nonumber
b_{i,-n}^*&=
\frac{\qqq^{n}(\qqq-\qqq^{-1})n}{(1-q^{nr})(1-t^{nr})[n]_\qqq}
\sum_{j,k=0}^{r-1} q^{nj}t^{nk}b_{i-j+k,-n}
\end{align}
\end{lem}

\begin{proof}
This is by direct calculation, recalling that $q=\qqq\ddd$ and $t=\qqq\ddd^{-1}$.
\end{proof}

\begin{cor}\label{DualSym}
In terms of symmetric functions, the dual bosons act as follows: for $n>0$,
\begin{align*}
\rho_{\vec{\ccc}}(b^*_{i,n})&= (\qqq-\qqq^{-1})p_n^\perp[X^{(i)}], \text{ and }\\
\rho_{\vec{\ccc}}(b^*_{i,-n})&= \qqq^n(\qqq-\qqq^{-1})p_n\left[ \frac{X^{(i)}}{(1-q\sigma^{-1})(1-t\sigma)} \right].
\end{align*}
\end{cor}

\begin{proof}
The first equation comes from (\ref{VerSym}) and (\ref{DualBoson}).
Equation (\ref{VerSym}) also yields the second equation.
\end{proof}

\subsubsection{The elements $F_{p,\pm1}$}
For $p\in\ZZ/r\ZZ$, let
\begin{align*}
F_{p,1} \coloneq 
\frac{z_{p,1}}{z_{0,1}}\prod_{i\in\ZZ/r\ZZ}z_{i,1} && \text{ and } &&
F_{p,-1} \coloneq 
\frac{z_{0,1}}{z_{p,1}}\prod_{i\in\ZZ/r\ZZ}z_{i,1}.
\end{align*}
We note that both elements are in $\Sss_{\delta}$ where $\delta=(1,\ldots, 1)$ is the diagonal vector.

\begin{lem}[\cite{Tsym, WreathEigen}]\label{FTShuff}
We have
\begin{align}
\label{FT1}
-(\qqq-\qqq^{-1})^{r-1}\Psi_+(F_{p,-1})&= \frac{1}{(\qqq-\qqq^{-1})}\varpi(b_{p,-1}^*),\text{ and }\\
\label{FT2}
\ddd^{-r}\qqq(\qqq^2-1)^{r-1}
\Psi_-\left( F_{p,1}\right)&= \frac{1}{(\qqq-\qqq^{-1})}\varpi(b_{p,1}^*).
\end{align}
\end{lem}

\begin{proof}
By \cite[B.2]{WreathEigen}\footnote{Note that our $b_{i,n}^*$ is equal to $-b_{i,n}^\perp$ from \textit{loc. cit.}}, we have
\[
(-1)^{r-1}\ddd^{-r}(\qqq-\qqq^{-1})^{r-1}\Psi_+(F_{p,1})=\frac{1}{(\qqq-\qqq^{-1})} \varpi^{-1}(b_{p,1}^*).
\]
We then obtain \eqref{FT1} by applying Proposition \ref{EtaShuff}.
Likewise, we obtain
\[
(-1)^{r-1}\qqq(\qqq^2-1)^{r-1}
\Psi_-\left( F_{p,-1}\right)= \frac{1}{(\qqq-\qqq^{-1})}\varpi^{-1}(b_{p,-1}^*)
\]
by combining Lemma 4.11 and Theorem 4.12 from \cite{WreathEigen}.
Equation \eqref{FT2} then follows by applying Proposition \ref{EtaShuff}.
\end{proof}

\subsubsection{The elements $\tilde{p}_1^{(i)}$}\label{PTilde}
For $i\in\ZZ/r\ZZ$, let
\[
\tilde{p}_1^{(i)} \coloneq p_1\left[ \frac{X^{(i)}}{(1-q\sigma^{-1})(1-t\sigma)} \right].
\]
These elements will play a key role later on.
Note that by Corollary \ref{DualSym},
\begin{align*}
\tilde{p}_1^{(i)}=\frac{\qqq^{-1}}{(\qqq-\qqq^{-1})}\rho_{\vec{\ccc}}(b^*_{i,-1}) && \text{ and } &&
p_1^\perp[X^{(i)}]= \frac{1}{(\qqq-\qqq^{-1})}\rho_{\vec{\ccc}}(b_{i,1}^*).
\end{align*}
Thus, they are related to the formulas in Lemma \ref{FTShuff} via the Miki automorphism $\varpi$.

\subsubsection{The elements $F_{0,n}$}
Notice that $F_{0,1}=F_{0,-1}$.
To access nabla operators, we will also make use of higher degree analogues of $F_{0,1}$.
Namely, consider the following element of $\mathcal{S}_{n\delta}$:
\[
F_{0,n} \coloneq \prod_{i=0}^r\prod_{a=1}^n z_{i,a}.
\]
Setting $F_{0,0}=1$, 
\begin{lem}\label{F0nLem}
We have
\begin{align*}
\sum_{n=0}^\infty\frac{(-1)^{\frac{rn(n-1)}{2}+n}(\qqq^2-1)^{nr}}{\qqq^n\ddd^{\frac{rn(n+1)}{2}}}\Psi_-(F_{0,n})z^n&= \exp\left( \sum_{k>0}\frac{\qqq^{-2k}-1}{(\qqq-\qqq^{-1})k}\varpi(b_{0,k}^*)z^k \right).
\end{align*}
\end{lem}

\begin{proof}
The following was proved in \cite{Tsym} (cf. Lemma 4.11 and Corollary 4.13 of \cite{WreathEigen}, noting that $b_{0,n}^*=-b_{0,n}^\perp$):
\begin{align*}
\sum_{n=0}^\infty\frac{(-1)^{\frac{rn(n+1)}{2}+n}(\qqq^2-1)^{nr}}{\qqq^n\ddd^{\frac{rn(n-1)}{2}}}\Psi_-(F_{0,n})z^n&= \exp\left(-\sum_{k>0}\frac{\qqq^{-2k}-1}{(\qqq-\qqq^{-1})k}\varpi^{-1}(b_{0,-k}^*)z^k \right).
\end{align*}
We obtain the second equality by applying Proposition \ref{EtaShuff}.
\end{proof}

\subsubsection{Eigenvalues}
The bosonic expressions in Lemma \ref{FTShuff} are in fact quite natural when one considers their action in the Fock representation.
\begin{lem}
On the Fock representation $\mathcal{F}$, we have
\begin{align}
\label{FTEigen}
\left\langle\lambda\left|\frac{1}{(\qqq-\qqq^{-1})v}\tau_v(b_{p,1}^*) \right|\lambda\right\rangle
&= \left(\frac{-D_\lambda^\bullet}{(1-q)(1-t)}\right)^{(p)}, \text{ and }\\
\nonumber
\left\langle\lambda\left|\frac{v}{(\qqq-\qqq^{-1})}\tau_v(b_{p,-1}^*) \right|\lambda\right\rangle
&= \left(\frac{-D_\lambda^\bullet}{(1-q)(1-t)}\right)^{(p)}_*.
\end{align}
Similarly, we have
\begin{align*}
\left\langle\lambda\left|\exp\left( \sum_{k>0}\frac{\qqq^{-2k}-1}{(\qqq-\qqq^{-1})k}\tau_v(b_{0,k}^*)\frac{z^k}{v^k} \right) \right|\lambda\right\rangle
&= \Omega\left[ -\left(\frac{(q^{-1}t^{-1}-1)D_\lambda^\bullet}{(1-q)(1-t)}\right)^{(0)}z \right].
\end{align*}
\end{lem}

\begin{proof}
We will only prove (\ref{FTEigen}).
Lemma \ref{BEigen} gives us
\begin{align*}
\nonumber
\langle\lambda|\tau_v(b_{i, 1})|\lambda\rangle&=-v\qqq^{-1}D_\lambda^{(i)}.
\end{align*}
Applying (\ref{DualForm}), consider
\begin{align*}
\left\langle\lambda\left|\frac{1}{(\qqq-\qqq^{-1})v}\tau_v(b^*_{p,1})\right|\lambda\right\rangle
&= \frac{1}{(1-q^{r})(1-t^{r})}\sum_{j,k=0}^{r-1}q^{j}t^{k}\left( -D_{\lambda}^{(p+j-k)}\right)\\
&= \left( \frac{-D_\lambda}{(1-q)(1-t)} \right)^{(p)}.\qedhere
\end{align*}
\end{proof}

\begin{cor}
In $\mathcal{W}$, we have the following equalities:
\begin{align*}
\frac{1}{(\qqq-\qqq^{-1})v}(\rho_{\vec{\ccc}}^-\circ\varpi)(b_{p,1}^*)(H_\lambda\otimes e^{\core(\lambda)})
&=\left(\frac{-D_\lambda^\bullet}{(1-q)(1-t)} \right)^{(p)}H_\lambda\otimes e^{\core(\lambda)},\\
\frac{v}{(\qqq-\qqq^{-1})}(\rho_{\vec{\ccc}}^+\circ\varpi)(b_{p,-1}^*)(H_\lambda\otimes e^{\core(\lambda)})
&=\left(\frac{-D_\lambda^\bullet}{(1-q)(1-t)} \right)^{(p)}_*H_\lambda\otimes e^{\core(\lambda)},
\\
\exp\left( \sum_{k>0}\frac{\qqq^{-2k}-1}{(\qqq-\qqq^{-1})k}(\rho_{\ccc}^-\circ\varpi)(b_{0,-k}^*)z^n \right)(H_\lambda\otimes e^{\core(\lambda)})
&=
\Omega\left[ -\left(\frac{(q^{-1}t^{-1}-1)D_\lambda^\bullet}{(1-q)(1-t)}\right)^{(0)}z \right]H_\lambda\otimes e^{\core(\lambda)}.
\end{align*}
Here, $v$ and $\vec{\ccc}$ are related as in Theorem \ref{TsymIso}.
\end{cor}

\subsubsection{The series $\mathbb{D}$ and $\mathbb{D}^*$}\label{DD*}
Define
\begin{align*}
\mathbb{D}& \coloneq 
\frac{(1-qt)^{r-1}}{\prod_{i\in\ZZ/r\ZZ}(1-qz_{i+1}/z_{i})(1-tz_{i}/z_{i+1})}
\Omega\left[ -\sum_{i\in\ZZ/r\ZZ}X^{(i)}z_{i} \right]
\\
& \hspace{5em}\quad\times
\mathcal{T}\left[ (1-q\sigma)(1-t\sigma^{-1})\sum_{i\in\ZZ/r\ZZ}X^{(i)}z_{i}^{-1} \right]
\prod_{i\in\ZZ/r\ZZ} z_{i}^{-H_{i,0}}, \text{ and }
\\
\mathbb{D}^*& \coloneq \frac{(1-q^{-1}t^{-1})^{r-1}}{\prod_{i\in\ZZ/r\ZZ}(1-t^{-1}z_{i+1}/z_{i})(1-q^{-1}z_i/z_{i+1})}
\Omega\left[ \sum_{i\in\ZZ/r\ZZ}X^{(i)}z_i \right]\\
& \hspace{5em}\quad\times
\mathcal{T}\left[-(1-q^{-1}\sigma^{-1})(1-t^{-1}\sigma)\sum_{i\in\ZZ/r\ZZ}X^{(i)}z_i^{-1} \right]
\prod_{i\in\ZZ/r\ZZ} z_{i}^{H_{i,0}}.
\end{align*}
For $\mathbb{D}$, we expand the rational functions as Laurent series assuming
\begin{equation*}
\begin{aligned}
|z_{i}|=1&& \text{ and } &&
|q|,|t|<1;
\end{aligned}
\end{equation*}
whereas for $\mathbb{D}^*$, we expand assuming
\begin{equation*}
\begin{aligned}
|z_{i}|=1&& \text{ and } &&
|q^{-1}|,|t^{-1}|<1.
\end{aligned}
\end{equation*}
Upon acting on $f\otimes e^\alpha\in\mathcal{W}$, we obtain Laurent series in the variables $z_{0},\cdots, z_{r-1}$.
Given a vector $\vec{k}=(k_0,\ldots,k_{r-1})\in(\ZZ)^{r}$, we define the operators $\mathbb{D}_{\vec{k}}$ and $\mathbb{D}^*_{\vec{k}}$ by setting
\begin{align*}
\mathbb{D}_{\vec{k}}(f\otimes e^\alpha)& \coloneq \left\{z_{0}^{k_0}\cdots z_{r-1}^{k_{r-1}}\mathbb{D}(f\otimes e^\alpha)\right\}_0 \text{ and }\\
\mathbb{D}^*_{\vec{k}}(f\otimes e^\alpha)& \coloneq \left\{z_{0}^{k_0}\cdots z_{r-1}^{k_{r-1}}\mathbb{D}^*(f\otimes e^\alpha)\right\}_0.
\end{align*}
Let $\epsilon_0,\ldots, \epsilon_{r-1}$ be the coordinate vectors.

\begin{lem}
For $p\in\ZZ/r\ZZ$, we have
\begin{align*}
\frac{1}{(\qqq-\qqq^{-1})v}(\rho_{\vec{\ccc}}^-\circ\varpi)(b_{p,1}^*)
=\mathbb{D}_{\epsilon_p-\epsilon_0} && \text{ and } &&
\frac{v}{(\qqq-\qqq^{-1})}(\rho_{\vec{\ccc}}^+\circ\varpi)(b_{p,-1}^*)
= \mathbb{D}^*_{\epsilon_0-\epsilon_p}.
\end{align*}
\end{lem}

\begin{proof}
We will only prove the first equality.
Combining Lemma \ref{FTShuff} and Proposition \ref{ShuffVert}, we get
\begin{align}
\nonumber
\frac{1}{(\qqq-\qqq^{-1})v}&(\rho_{\vec{\ccc}}^-\circ\varpi)(b_{p,1}^*)(f\otimes e^\alpha)\\
\nonumber
&=
\left\{ \frac{\displaystyle \ddd^{-r}\qqq(\qqq^2-1)^{(r-1)}\frac{z_{p,1}}{z_{0,1}}\left(\prod_{i\in\ZZ/r\ZZ}z_{i,1}\right)\overset{\curvearrowleft}{\prod_{i=0}^{r-1}}F_i(z_{i,1})(f\otimes e^\alpha)}
{\displaystyle (-1)^{\frac{(r-2)(r-3)}{2}}\frac{\qqq\ddd^{-\frac{r}{2}}}{\ccc_0\cdots\ccc_{r-1}}
\prod_{0\le i< j\le r-1}\omega_{i,j}(z_{i,1},z_{j,1})}\right\}_0\\
\label{DLem}
&=
\left\{ \frac{\displaystyle (1-qt)^{(r-1)}\frac{z_{p,1}z_{r-1,1}}{z_{0,1}^2}\overset{\curvearrowleft}{\prod_{i=0}^{r-1}}F_i(z_{i,1})(f\otimes e^\alpha)}
{\displaystyle (-1)^{\frac{(r-2)(r-3)}{2}}\frac{\ddd^{-\frac{r}{2}+1}}{\ccc_0\cdots\ccc_{r-1}}
\left( 1-tz_{r-1,1}/z_{0} \right)\prod_{i=0}^{r-2}\left(1- qz_{i+1}/z_{i} \right)}\right\}_0,
\end{align}
and
\begin{align*}
\overset{\curvearrowleft}{\prod_{i=0}^{r-1}}F_i(z_{i,1})
&=\frac{(-1)^{\frac{(r-2)(r-3)}{2}}}{\ddd^{\frac{r}{2}-1}\ccc_0\cdots\ccc_{r-1}}\\ 
&\quad\times \frac{z_{0,1}/z_{r-1,1}}{\displaystyle\left( 1-qz_{0,1}/z_{r-1} \right)\prod_{i=0}^{r-2}\left( 1-tz_{i,1}/z_{i+1} \right)}\Omega\left[-\sum_{i\in\ZZ/r\ZZ}X^{(i)}z_{i,1} \right]\\
&\quad\times
\mathcal{T}\left[ (1-q\sigma)(1-t\sigma^{-1})\sum_{i\in\ZZ/r\ZZ} X^{(i)}z_{i,1}^{-1} \right]\prod_{i\in\ZZ/r\ZZ}z_{i,1}^{-H_{i,0}}.
\end{align*}
Inserting this into (\ref{DLem}), we do indeed obtain $\mathbb{D}_{\epsilon_p-\epsilon_{0}}(f\otimes e^\alpha)$.
\end{proof}

\begin{cor}\label{DDiag}
The operators $\mathbb{D}_{\epsilon_{p}-\epsilon_{0}}$ and $\mathbb{D}^*_{\epsilon_{0}-\epsilon_{p}}$ act diagonally on $H_\lambda\otimes e^{\core(\lambda)}$ with eigenvalues
\begin{align*}
\mathbb{D}_{\epsilon_{p}-\epsilon_{0}}(H_\lambda\otimes e^{\core(\lambda)})
&=\left( \frac{-D_\lambda^\bullet}{(1-q)(1-t)} \right)^{(p)}H_\lambda\otimes e^{\core(\lambda)} \text{ and }\\
\mathbb{D}_{\epsilon_{0}-\epsilon_{p}}^*(H_\lambda\otimes e^{\core(\lambda)})
&=\left( \frac{-D_\lambda^\bullet}{(1-q)(1-t)} \right)^{(p)}_*H_\lambda\otimes e^{\core(\lambda)}.
\end{align*}
\end{cor}

\subsection{Wreath Macdonald operators}\label{WreathOp}
We have already seen some operators that act diagonally on $H_\lambda$.
Here, we make a deeper study of eigenoperators.

\subsubsection{Delta operators}\label{DeltaMap}
We wish to define a ring homomorphism $\Delta:\Lambda_{q,t}^{\otimes r}\rightarrow \mathrm{End}_{\FF}(\mathcal{W})$.
First, we define the embedding $\epsilon^+:\Lambda_{q,t}^{\otimes r}\rightarrow\varpi(\mathcal{H})$ by setting
\[
p_k[X^{(i)}]\mapsto \frac{\qqq^kk}{v^k[k]_\qqq}\varpi(b_{i,k}).
\]
We then define $\Delta \coloneq \rho_{\vec{\ccc}}^-\circ\epsilon^+$ and denote its inputs using square brackets.
By Theorem \ref{Eigenstates}, the image of $\Delta$ acts diagonally on $\{H_\lambda\otimes e^{\core(\lambda)}\}$.
The following is a corollary of Lemma \ref{BEigen}:
\begin{cor}\label{DeltaCor}
For any $f\in\Lambda_{q,t}^{\otimes r}$, we have
\[
\Delta[f](H_\lambda\otimes e^{\core(\lambda)})=f[-D_\lambda^\bullet] H_\lambda\otimes e^{\core(\lambda)}.
\]
\end{cor}

Using (\ref{DualForm}), Corollary \ref{DDiag} implies 

\begin{prop}\label{DeltaP1}
We have
\[
\Delta\left[ p_1\left[ \frac{X^{(i)}}{(1-q\sigma)(1-t\sigma^{-1})} \right] \right]=
\Delta\left[ \tilde{p}_1^{(-i)}[\iota X^\bullet] \right]
=\mathbb{D}_{\epsilon_{i}-\epsilon_{0}}.
\]
\end{prop}

\subsubsection{Column-type operators}
We now consider a modified elementary symmetric function
\[
\hat{e}_n^{(p)}=e_n\left[ \frac{X^{(p)}}{1-t^{-1}\sigma^{-1}} \right]=t^ne_n\left[\frac{-X^{(p+1)}}{1-t\sigma}  \right],
\]
for which we have
\[
\hat{e}_n^{(-p-1)}[-\iota X^\bullet]=
t^ne_n\left[- \iota\left(\frac{ -X^{((-p-1)+1)}}{1-t\sigma}  \right)\right]
=t^ne_n\left[ \frac{X^{(p)}}{1-t\sigma^{-1}} \right].
\]

\begin{thm}[\cite{OSW}]\label{MacCol}
We have
\begin{equation}
\begin{aligned}
&\Delta\left[ \hat{e}_n^{(-p-1)}[-\iota X^\bullet] \right](f\otimes e^\alpha)\\
&\hspace{3em} = 
\frac{t^n(1-qt)^{nr}}{\prod_{k=1}^n(1-q^kt^k)} 
\left\{\rule{0cm}{2.2em}\right.
\prod_{1\le a<b\le n}\left[\rule{0cm}{2.2em}\right.
\frac{\left(1-z_{p,a}/z_{p,b}\right)\left(1-qtz_{p,a}/z_{p,b}\right)}{\left(1-t^{-1}z_{p+1,a}/z_{p,b}\right)\left(1-tz_{p-1,a}/z_{p,b}\right)}
\\
&\hspace{3em}\quad\times\left.
\prod_{i\in \ZZ/r\ZZ\backslash\{p\}}
\frac{\left(1-z_{i,a}/z_{i,b}\right)\left(1-qtz_{i,a}/z_{i,b}\right)}{\left(1-qz_{i+1,a}/z_{i,b}\right)\left(1-tz_{i-1,a}/z_{i,b}\right)}\right]\\
&\hspace{3em}\quad\times
\prod_{a=1}^n\left[
\prod_{i\in \ZZ/r\ZZ\setminus\{p\}}
\left(\frac{1}{1-tz_{i,a}/z_{i+1,a}}\right)
\left( \frac{1}{1-qz_{i+1,a}/z_{i,a}} \right)
\right.\\
&\hspace{3em}\quad\times 
\left(\frac{z_{p,a}}{z_{0,a}}\right)
\left(\frac{1}{1-tz_{p,a}/z_{p+1,a}}\right)
\left]\rule{0cm}{2.2em}\right.
\Omega\left[-\sum_{i\in\ZZ/r\ZZ}X^{(i)}\left( \sum_{a=1}^n z_{i,a}\right) \right]\\
&\hspace{3em} \left.\quad\times
\TT\left[ (1-q\sigma)(1-t\sigma^{-1})\sum_{i\in\ZZ/r\ZZ}  X^{(i)}\left( \sum_{a=1}^nz_{i,a}^{-1} \right)  \right]
f\otimes \prod_{i\in\ZZ/r\ZZ}\prod_{a=1}^n z_{i,a}^{-H_{i,0}}e^\alpha\right\}_0.
\end{aligned}
\label{DeltaE}
\end{equation}
In particular,
\[
\Delta\left[ \hat{e}_n^{(-p-1)}[-\iota X^\bullet] \right](H_\lambda\otimes e^{\core(\lambda)})= t^ne_n\left[ \left(\frac{-D_\lambda^\bullet}{1-t}\right)^{(p)} \right]H_\lambda\otimes e^{\core(\lambda)}.
\]
\end{thm}

\begin{proof}
The expression (\ref{DeltaE}) is gotten from \cite{OSW}.
Namely, we wish to apply $\rho_{\vec{\ccc}}^-$ to the coefficient of $z^n$ in
\[
\varpi\exp\left[-\sum_{k>0}\left(\frac{\sum_{i=0}^{r-1}t^{k(i+1)}b_{p-i,k}}{(1-t^{kr})}\right)\frac{\qqq^k}{v^k[k]_\qqq}(-z)^k \right].
\]
This is exactly the expression on the right-hand side of the second equation of \cite[Lemma 4.4]{OSW}\footnote{Note that our $\varpi$ is denoted by $\varsigma$ and our $b_{i,k}$ is denoted by $h_{i,k}$.}.
The aforementioned lemma computes $\Psi_-^{-1}$ of this coefficient.
We can then use Proposition \ref{ShuffVert} to compute the action of this shuffle element on $\mathcal{W}$.
The end result is given by the second equation in \cite[Lemma 5.2]{OSW}.
Note that \textit{loc. cit.} considers the finite-variable situation and their vertex representation differs from ours; thus, the $\Omega$- and $\TT$-terms need to be adjusted.
For the eigenvalue, we use Corollary \ref{DeltaCor}:
\begin{align*}
\frac{X^{(p)}}{1-t\sigma^{-1}} \bigg|_{X^{(i)}\mapsto -D_\lambda^{(i)}}
&= \frac{\sum_{i=0}^{r-1}t^{i}X^{(p-i)}}{1-t^{r}}\bigg|_{X^{(i)}\mapsto -D_{\lambda}^{(i)}}\\
&= \frac{-\sum_{i=0}^{r-1}t^{i}D_{\lambda}^{(p-i)}}{1-t^{r}}\\
&= \left( \frac{-D_\lambda^\bullet}{1-t} \right)^{(p)}.\qedhere
\end{align*}
\end{proof}

\subsubsection{Row-type operators}
We make a similar study of a modified complete homogeneous symmetric function
\[
\hat{h}_n^{(p)}=h_n\left[ \frac{X^{(p)}}{1-q\sigma^{-1}} \right].
\]
Notice that we have
\[
\hat{h}_n^{(-p)}[-\iota X^\bullet]=h_n\left[ \frac{-X^{(p)}}{1-q\sigma} \right].
\]

\begin{thm}[\cite{OSW}]
We have
\begin{align*}
&\Delta\left[ \hat{h}_n^{(-p)}[-\iota X^\bullet] \right](f\otimes e^\alpha)\\
&=
\frac{(-1)^n\left(1- qt \right)^{nr}}{\prod_{k=1}^n\left(1- q^{k}t^{k} \right)}
\left\{\rule{0cm}{2.2em}\right.\prod_{1\le a<b\le n}\left[\rule{0cm}{2.2em}\right.
\frac{\left(1-z_{p+1,a}/z_{p+1,b}\right)\left(1-qtz_{p+1,a}/z_{p+1,b}\right)}{\left(1-q^{-1}z_{p,a}/z_{p+1,b}\right)\left(1-qz_{p-2,a}/z_{p+1,b}\right)} \\
&\quad\times\left.
\prod_{i\in \ZZ/r\ZZ\backslash\{p+1\}}
\frac{\left(1-z_{i,a}/z_{i,b}\right)\left(1-qtz_{i,a}/z_{i,b}\right)}{\left(1-qz_{i+1,a}/z_{i,b}\right)\left(1-tz_{i-1,a}/z_{i,b}\right)}\right]\\
&\quad\times \prod_{a=1}^n
\left[
\prod_{i\in \ZZ/r\ZZ\setminus\{p+1\}}
 \left(\frac{1}{1-qz_{i,a}/z_{i-1,a}}\right)\left(\frac{1}{1-tz_{i-1,a}/z_{i,a}}\right)
\right.\\
&\quad \times
\left(\frac{z_{p+1,a}}{z_{0,a}}\right)
\left(\frac{1}
{1-qz_{p+1,a}/z_{p,a}}\right) \left]\rule{0cm}{2.2em}\right.
\Omega\left[ -\sum_{i\in\ZZ/r\ZZ}X^{(i)}\left( \sum_{a=1}^n z_{i,a} \right) \right]
\\
&\quad\left.\times 
\TT\left[ (1-q\sigma)(1-t\sigma^{-1})\sum_{i\in\ZZ/r\ZZ}  X^{(i)}\left( \sum_{a=1}^nz_{i,a}^{-1} \right)  \right]
f\otimes \prod_{i\in\ZZ/r\ZZ}\prod_{a=1}^n z_{i,a}^{-H_{i,0}}e^\alpha\right\}_0
\end{align*}
and
\begin{align*}
\Delta\left[ \hat{h}_n^{(-p)}[-\iota X^\bullet] \right](H_\lambda\otimes e^{\core(\lambda)})
&= h_n\left[ \left( \frac{D_\lambda^\bullet}{1-q} \right)^{(p)} \right]H_\lambda\otimes e^{\core(\lambda)}.
\end{align*}
Here, we expand the character on the right hand side assuming $|q|<1$.
\end{thm}

\begin{proof}
The proof is similar to that of Theorem \ref{MacCol}.
Namely, the end of \cite[A.2]{OSW} gives shuffle elements corresponding to the coefficients of
\[
\varpi\exp\left[ -\sum_{k>0}\left( \frac{\sum_{i=0}^{r-1}q^{ki}b_{p+i,k}}{1-q^{kr}} \right)\frac{\qqq^k}{v^k[k]_\qqq}z^k \right].
\] 
We apply Proposition \ref{ShuffVert}, the result of which is similar to \cite[A.4]{OSW}.
\end{proof}

\subsubsection{Operators from $F_{0,n}$}
Let
\[
e_{F,n}^{(0)} \coloneq e_n\left[ \frac{(q^{-1}t^{-1}-1)X^{(0)}}{(1-q\sigma)(t\sigma^{-1}-1)}\right].
\]
We will compute one more family of Delta operators.
\begin{lem}
We have
\begin{equation}
\begin{aligned}
\Delta\left[ e_{F,n}^{(0)}\right]& (f\otimes e^\alpha)\\
&= 
\frac{(1-qt)^{nr}}{\prod_{k=1}^n(1-q^kt^k)} 
\left\{
\prod_{1\le a\not= b\le n}
\prod_{i\in \ZZ/r\ZZ}
\frac{\left(1-z_{i,a}/z_{i,b}\right)\left(1-qtz_{i,a}/z_{i,b}\right)}{\left(1-qz_{i+1,a}/z_{i,b}\right)\left(1-tz_{i-1,a}/z_{i,b}\right)}
\right.
\\
&\quad\times
\prod_{a=1}^n
\prod_{i\in \ZZ/r\ZZ}
\left(\frac{1}{1-tz_{i,a}/z_{i+1,a}}\right)
\left( \frac{1}{1-qz_{i+1,a}/z_{i,a}} \right)
\Omega\left[-\sum_{i\in\ZZ/r\ZZ}X^{(i)}\left( \sum_{a=1}^n z_{i,a}\right) \right]\\
&\left.\quad\times
\TT\left[ (1-q\sigma)(1-t\sigma^{-1})\sum_{i\in\ZZ/r\ZZ}  X^{(i)}\left( \sum_{a=1}^nz_{i,a}^{-1} \right)  \right]
f\otimes \prod_{i\in\ZZ/r\ZZ}\prod_{a=1}^n z_{i,a}^{-H_{i,0}}e^\alpha\right\}_0
\end{aligned}
\label{F0Calc}
\end{equation}
and
\[
\Delta\left[ e_{F,n}^{(0)} \right]
(H_\lambda\otimes e^{\core(\lambda)})
=
e_n\left[ (q^{-1}t^{-1}-1)\left( B_\lambda^\bullet-\frac{1}{(1-q)(1-t)} \right)^{(0)} \right]H_\lambda\otimes e^{\core(\lambda)}.
\]
\end{lem}

\begin{proof}
First observe that by (\ref{DualForm}), 
\begin{align*}
\epsilon^+\left( p_k\left[ \frac{(q^{-1}t^{-1}-1)X^{(0)}}{(1-q\sigma)(t\sigma^{-1}-1)}\right] \right)
&=-\frac{(\qqq^{-2k}-1)\qqq^kk}{v^k[k]_\qqq}\frac{\sum_{j,k=0}^{r-1}q^{nj}t^{nk}b_{j-k,k}}{(1-q^{nr})(1-t^{nr})}\\
&=-\frac{(\qqq^{-2k}-1)}{(\qqq-\qqq^{-1})v^k}b_{0,k}^*.
\end{align*}
Thus, by Lemma \ref{F0nLem},
\[
\Delta\left[ e_{F,n}^{(0)} \right]
=\frac{(-1)^{\frac{rn(n-1)}{2}+n}(\qqq^2-1)^{nr}}{\qqq^n\ddd^{\frac{rn(n+1)}{2}}v^n}
\left(\rho_{\vec{\ccc}}^-\circ\Psi_-\right)(F_{0,n}).
\]
We obtain (\ref{F0Calc}) by applying Proposition \ref{ShuffVert} and \cite[(3.12)]{OSW}.
For the eigenvalue, note that
\[
\frac{-D_\lambda^\bullet}{(1-q)(t-1)}=B_\lambda^\bullet -\frac{1}{(1-q)(1-t)}.
\qedhere
\]
\end{proof}

\subsubsection{Nabla revisited}
Recall the nabla operator $\nabla_\alpha$, which depends on an $r$-core $\alpha.$
We extend its definiton to $\Lambda_{q,t}\otimes\CC[Q]$ by setting
\[
\nabla\left( f\otimes e^\alpha \right)=\nabla_\alpha(f)\otimes e^\alpha.
\]

\begin{lem}\label{NablaShuff}
For a fixed degree $N$, there exists an expression $\mathcal{E}_N$ in terms of $\{e_{F,n}^{(0)}\}$ such that for $f\in\Lambda_{q,t}^{\otimes r}$ of degree $N$,
\[
\nabla\left(f\otimes e^{\alpha}\right) = \frac{(-1)^N\Delta\left[ \mathcal{E}_N \right](f\otimes e^\alpha)}{\displaystyle\prod_{\substack{\square\in\core(\lambda)\\ \bar{c}_\square=0}}\chi_\square}.
\]
We emphasize that $\mathcal{E}_N$ is independent of $\alpha$.
\end{lem}

\begin{proof}
We will prove the lemma by considering the eigenvalues on $H_\lambda$ with $|\quot(\lambda)|=N$.
Notice that
\[
\frac{q^{-1}t^{-1} - 1}{(1-q)(1-t)}=\frac{q^{-1}t^{-1}}{1-t}+\frac{ t^{-1}}{1-q},
\]
and since $q^{-1}t^{-1}$ has color $0$, 
\begin{align*}
(q^{-1}t^{-1} - 1)
\left(\frac{1}{(1-q)(1-t)}\right)^{(0)} =\left(\frac{q^{-1}t^{-1} -1 }{(1-q)(1-t)}\right)^{(0)}.
\end{align*}
Thus,
\[
(q^{-1}t^{-1}-1)\left( B_\lambda^\bullet-\frac{1}{(1-q)(1-t)} \right)^{(0)}=(q^{-1}t^{-1}-1)B_\lambda^{(0)}-\frac{q^{-1}t^{-1}}{1-t^r}-\frac{q^{r-1}t^{-1}}{1-q^r}.
\]
We can then write $e_N[B_{\lambda}^{(0)}]$ in terms of
\begin{equation*}
\begin{aligned}
e_n\left[(q^{-1}t^{-1}-1)B_\lambda^{(0)}-\frac{q^{-1}t^{-1}}{1-t^r}-\frac{q^{r-1}t^{-1}}{1-q^r}\right],&
&e_n\left[ -\frac{q^{-1}t^{-1}}{1-t^r} \right],&& \text{ and }
&&e_n\left[- \frac{q^{r-1}t^{-1}}{1-q^r} \right].
\end{aligned}
\end{equation*}
The latter two are rational functions in $\CC(q,t)$, and thus we obtain $\mathcal{E}_N$.
\end{proof}

\subsubsection{Adjoints}
Recall how in \ref{PairCores}, we extended the pairing $\langle -,-\rangle_{q,t}'$ to $\mathcal{W}$ in a somewhat peculiar way.
Let $\Delta^\dagger$ denote the composition of $\Delta$ with the adjunction.
Thus, we have
\begin{equation}
\Delta^\dagger[f](H_\lambda^\dagger\otimes e^{w_0\core(\lambda)})=f[-D_\lambda^\bullet]H_\lambda^\dagger\otimes e^{w_0\core(\lambda)}.
\label{AdjointEq}
\end{equation}
More generally, for an operator $M$, we denote its adjoint by $M^\dagger$.
To explicitly compute adjoint operators when the currents $\{F_i(z)\}$ are involved, Lemmas \ref{FAdjoint1} and \ref{FAdjoint2} tell us to send the variables $z_{i,a}\mapsto z_{-i,a}^{-1}$.
When working with $n$ currents of each color, we will in fact perform $z_{i,a}\mapsto z_{-i, n+1-a}^{-1}$.

\begin{lem}\label{AdjointLem}
Taking adjoints gives us the following families of operators:
\begin{enumerate}
\item For $\vec{k}=(k_0,\ldots, k_{r-1})\in\ZZ^n$, let $\iota\vec{k}=(k_0, k_{r-1}, \ldots, k_1)$ and $|\vec{k}|=\sum k_i$.
Then
\begin{equation*}
\begin{aligned}
\mathbb{D}_{\vec{k}}^\dagger = \mathbb{D}_{-\iota\vec{k}}&& \text{ and } &&
\left(\mathbb{D}_{\vec{k}}^*\right)^\dagger= (qt)^{-|\vec{k}|}\mathbb{D}_{-\iota\vec{k}}^*.
\end{aligned}
\end{equation*}
In particular, from $\mathbb{D}_{\epsilon_{p}-\epsilon_{0}}$ and $\mathbb{D}_{\epsilon_{0}-\epsilon_{p}}^*$, we obtain
\begin{align*}
\mathbb{D}_{\epsilon_{0}-\epsilon_{p}}(H_\lambda^\dagger\otimes e^{w_0\core(\lambda)})&=\left( \frac{-D_\lambda^\bullet}{(1-q)(1-t)} \right)^{(-p)}H_\lambda^\dagger\otimes e^{w_0\core(\lambda)} \text{ and }\\
\mathbb{D}_{\epsilon_{p}-\epsilon_{0}}^*(H_\lambda^\dagger\otimes e^{w_0\core(\lambda)})&=\left( \frac{-D_\lambda^\bullet}{(1-q)(1-t)} \right)^{(-p)}_*H_\lambda^\dagger\otimes e^{w_0\core(\lambda)}.
\end{align*}
\item The column-type operators become 
\begin{equation}
\begin{aligned}
\Delta^\dagger&\left[ \hat{e}_n^{(p)}[-\iota X^\bullet] \right](f\otimes e^\alpha)\\
&= 
\frac{t^n(1-qt)^{nr}}{\prod_{k=1}^n(1-q^kt^k)} 
\left\{
\prod_{1\le a<b\le n}\left[\rule{0cm}{2.2em}\right.
\frac{\left(1-z_{p,a}/z_{p,b}\right)\left(1-qtz_{p,a}/z_{p,b}\right)}{\left(1-t^{-1}z_{p+1,a}/z_{p,b}\right)\left(1-tz_{p-1,a}/z_{p,b}\right)}
\right.
\\
&\quad\times 
\prod_{i\in \ZZ/r\ZZ\backslash\{p\}}
\frac{\left(1-z_{i,a}/z_{i,b}\right)\left(1-qtz_{i,a}/z_{i,b}\right)}{\left(1-qz_{i+1,a}/z_{i,b}\right)\left(1-tz_{i-1,a}/z_{i,b}\right)}\left]\rule{0cm}{2.2em}\right.\\
&\quad\times
\prod_{a=1}^n \left[\rule{0cm}{2.2em}\right.
\prod_{i\in \ZZ/r\ZZ\setminus\{p\}}
\left(\frac{1}{1-tz_{i,a}/z_{i+1,a}}\right)
\left( \frac{1}{1-qz_{i+1,a}/z_{i,a}} \right)
\\
&\quad\times 
\left(\frac{z_{0,a}}{z_{p+1,a}}\right)
\left(\frac{1}{1-tz_{p,a}/z_{p+1,a}}\right)
\left]\rule{0cm}{2.2em}\right.
\Omega\left[-\sum_{i\in\ZZ/r\ZZ}X^{(i)}\left( \sum_{a=1}^n z_{i,a}\right) \right]\\
&\left.\quad\times
\TT\left[ (1-q\sigma)(1-t\sigma^{-1})\sum_{i\in\ZZ/r\ZZ}  X^{(i)}\left( \sum_{a=1}^nz_{i,a}^{-1} \right)  \right]
f\otimes \prod_{i\in\ZZ/r\ZZ}\prod_{a=1}^n z_{i,a}^{-H_{i,0}}e^\alpha\right\}_0.
\end{aligned}
\label{DeltaEForm}
\end{equation}
\item The row-type operators become
\begin{equation}
\begin{aligned}
\Delta^\dagger & \left[ \hat{h}_n^{(p)}[-\iota X^\bullet] \right](f\otimes e^\alpha)\\
&=
\frac{(-1)^n\left(1- qt \right)^{nr}}{\prod_{k=1}^n\left(1- q^{k}t^{k} \right)}
\left\{\prod_{1\le a<b\le n}\left[\rule{0cm}{2.2em}\right.
\frac{\left(1-z_{p,a}/z_{p,b}\right)\left(1-qtz_{p,a}/z_{p,b}\right)}{\left(1-q^{-1}z_{p-1,a}/z_{p,b}\right)\left(1-qz_{p+1,a}/z_{p,b}\right)}\right.\\
&\quad\times 
\prod_{i\in \ZZ/r\ZZ\backslash\{p\}}
\frac{\left(1-z_{i,a}/z_{i,b}\right)\left(1-qtz_{i,a}/z_{i,b}\right)}{\left(1-qz_{i+1,a}/z_{i,b}\right)\left(1-tz_{i-1,a}/z_{i,b}\right)}\left]\rule{0cm}{2.2em}\right.\\
&\quad\times \prod_{a=1}^n
\left[\rule{0cm}{2.2em}\right.
\prod_{i\in \ZZ/r\ZZ\setminus\{p\}}
 \left(\frac{1}{1-qz_{i,a}/z_{i-1,a}}\right)\left(\frac{1}{1-tz_{i-1,a}/z_{i,a}}\right)
\\
&\quad \times
\left(\frac{z_{0,a}}{z_{p,a}}\right)
\left(\frac{1}
{1-qz_{p,a}/z_{p-1,a}}\right) \left]\rule{0cm}{2.2em}\right.
\Omega\left[ -\sum_{i\in\ZZ/r\ZZ}X^{(i)}\left( \sum_{a=1}^n z_{i,a} \right) \right]
\\
&\quad\left.\times 
\TT\left[ (1-q\sigma)(1-t\sigma^{-1})\sum_{i\in\ZZ/r\ZZ}  X^{(i)}\left( \sum_{a=1}^nz_{i,a}^{-1} \right)  \right]
f\otimes \prod_{i\in\ZZ/r\ZZ}\prod_{a=1}^n z_{i,a}^{-H_{i,0}}e^\alpha\right\}_0.
\end{aligned}
\label{DeltaHForm}
\end{equation}
\item $\Delta\left[ e_{F,n}^{(0)} \right]$ is self-adjoint.
\end{enumerate}
\end{lem}

\begin{cor}\label{NablaAdjoint}
The adjoint of $\nabla_{\alpha}$ is $\nabla_{w_0\alpha}$.
In other words, $\nabla$ is self-adjoint.
\end{cor}

\begin{proof}
Because $\Delta\left[ e_{F,n}^{(0)} \right]$ is self-adjoint, the expression $\Delta[\mathcal{E}_N]$ from Lemma \ref{NablaShuff} is self-adjoint as well.
Proposition \ref{RevProp} then tells us
\[
\prod_{\substack{\square\in\alpha\\ \bar{c}_\square=0}}\chi_\square=\prod_{\substack{\square\in w_0\alpha\\ \bar{c}_\square=0}}\chi_\square
.\qedhere
\]
\end{proof}

\section{Tesler identity}\label{TeslerSec}

\subsection{Setup}
Here, we gather some ingredients necessary for our proofs.

\subsubsection{Constant term lemma}
We will make use of the following lemma for computing constant terms:
\begin{lem}\label{ConstTermLem}
Suppose we are taking the constant term of an expression
\[
\frac{zF(z)}{(z-P_1)\cdots(z-P_k)}
\]
with respect to the variable $z$, where
\begin{itemize}
\item $F(z)$ is a series in nonnegative powers of $z$;
\item each pole is simple and nonzero; and
\item each pole is expanded as
\[
\frac{1}{z-P_i}=z^{-1}\sum_{n\ge 0}\frac{P_i^n}{z^n}.
\]
\end{itemize}
Then the constant term is a sum over evaluations:
\[
\left\{\frac{zF(z)}{(z-P_1)\cdots(z-P_k)}\right\}_0=\sum_{i=1}^k \frac{F(P_i)}{\displaystyle\prod_{j\not=i}(P_i-P_j)}.
\]
\end{lem}

\begin{proof}
Applying the partial fraction decomposition to $\prod_{i=1}^k(z-P_i)^{-1}$, we have
\[
\frac{z}{(z-P_1)\cdots(z-P_k)}=\sum_{i=1}^k \frac{1}{\displaystyle\left(1-\frac{P_i}{z}\right)\prod_{j\not=i}(P_i-P_j)}.
\]
Furthermore, if $F(z) = \sum_{k \geq 0 } F_k z^k$, then
\[
\left\{ \frac{\sum_{k  } z^k F_k}{1-P_i/z} \right\}_0
= {\sum_{k \geq 0} P_i^k F_k} = F(P_i).
\qedhere
\]
\end{proof}

\subsubsection{Down arrow}
To pass between $\mathbb{D}$ and $\mathbb{D}^*$, we will make use of an analogue of the $\downarrow$ map from \cite{GHT}.
This was first presented in \cite{OSWreath} on $\Lambda_{q,t}^{\otimes r}$, and we make the following extension to $\Lambda_{q,t}^{\otimes r}\otimes\CC[Q]$:
\[
\downarrow \left( f\otimes e^\alpha \right) \coloneq  f\left[-\iota X^\bullet; q^{-1}, t^{-1}\right]\otimes e^{w_0\alpha}.
\]
Thus, $\downarrow$ is only $\CC$-linear as it inverts $(q,t)$.
Our choice of extension to the $\CC[Q]$ factor is motivated by the following, which is straightforward:
\begin{prop}\label{DownArrowProp}
The involution $\downarrow$ satisfies:
\begin{enumerate}
\item $\downarrow\left( H_\lambda\otimes e^{\core(\lambda)} \right)=H_\lambda^\dagger\otimes e^{w_0\core(\lambda)}$,
\item $\downarrow\nabla_\alpha\downarrow= \nabla^{-1}_{w_0\alpha}$,
\item For any $\vec{k}=(k_0,k_1,\ldots, k_{r-1})\in \ZZ^{r}$, let $\iota\vec{k}=(k_0, k_{r-1},\ldots, k_1)$.
Then $\downarrow\mathbb{D}_{\vec{k}}\downarrow=\mathbb{D}^*_{\iota\vec{k}}$.
\end{enumerate}
\end{prop}

\subsubsection{Containment lemma}
We note that the proof of the following result uses the quantum toroidal algebra:

\begin{lem}\label{P1Supp}
Let $f\in\Lambda_{q,t}^{\otimes r}$ be homogeneous of degree $n$.
Then
\begin{align}
\label{SuppM}
fH_\lambda&=\sum_{\substack{\mu\\\core(\mu) = \core(\lambda)}}c_{\mu\lambda}H_\mu~ \text{ and }\\
\label{SuppMDag}
fH_\lambda^\dagger&=\sum_{\substack{\mu\\\core(\mu) = \core(\lambda)}}c_{\mu\lambda}^\dagger H^\dagger_\mu,
\end{align}
where the $\mu$ that appear in either sum are such that $\mu\supset\lambda$ and $\mu\backslash\lambda$ contains exactly of $n$ boxes of each color.
Similarly,
\begin{equation}
\begin{aligned}
f^\perp H_\lambda&=\sum_{\substack{\mu\\\core(\mu) = \core(\lambda)}}d_{\mu\lambda}H_\mu ~ \text{ and }\\
f^\perp H_\lambda^\dagger&=\sum_{\substack{\mu\\\core(\mu) = \core(\lambda)}}d_{\mu\lambda}^\dagger H^\dagger_\mu,
\end{aligned}
\label{SuppS}
\end{equation}
where the $\mu$ that appear in either sum are such that $\mu\subset\lambda$ and $\lambda\backslash\mu$ contains exactly $n$ boxes of each color.
\end{lem}

\begin{proof}
Let us embed $\Lambda_{q,t}^{\otimes r}$ into $\UTor$ compatibly with (\ref{VerSym}):
\[
p_n[X^{(i)}]= \frac{n}{[n]_\qqq}b_{i,-n}
\]
Recall the elements $\{\hat{e}_n^{(p)}\}$ from \ref{ModEH}.
Combining Proposition 4.23 and Theorem 4.31 from \cite{WreathEigen}, we have that in the Fock representation $\mathcal{F}$,
\[
\left\langle\mu\left|(\tau_v\circ\varpi^{-1})(\hat{e}_n^{(p)})\right|\lambda\right\rangle
\] 
is nonzero only if $\mu\supset\lambda$ and $\mu\backslash\lambda$ contains exactly $n$ boxes of each color.
Applying Theorems \ref{TsymIso} and \ref{Eigenstates}, we obtain (\ref{SuppM}) when $f=\hat{e}_n^{(p)}$.
Since $\{\hat{e}_n^{(p)} \}$ generates $\Lambda_{q,t}^{\otimes r}$ algebraically, we obtain (\ref{SuppM}) for general $f$.
By Proposition \ref{DownArrowProp}(1), (\ref{SuppMDag}) follows from applying $\downarrow$ to (\ref{SuppM}).
Finally, by Lemma \ref{FAdjoint1}, we have
\[
\left\langle g, f^\perp h \right\rangle_{q,t}'=\left\langle f\left[ \frac{\iota X^\bullet}{(1-q\sigma^{-1})(t\sigma-1)} \right]g, h\right\rangle_{q,t}'
\]
Thus, (\ref{SuppS}) follows from (\ref{SuppM}) and (\ref{SuppMDag}).
\end{proof}

\subsubsection{Relations}\label{Relations}
Here, we record some commutation relations involving $\mathbb{D}$ and $\mathbb{D}^*$.
These are completely analogous to those catalogued in \cite{GHT}.

\begin{lem}\label{ODRel}
The following relations hold:
\begin{align}
\label{Rel1}
\Omega\left[ \frac{X^{(i)}w}{(1-q\sigma^{-1})(1-t\sigma)}\right]\mathbb{D}&= \left(1-\frac{w}{z_i}\right)\mathbb{D}\Omega\left[ \frac{X^{(i)}}{(1-q\sigma^{-1})(1-t\sigma)}w \right],\\
\label{Rel2}
\TT\left[ X^{(i)}w \right]\mathbb{D}&= \left( 1-z_iw \right)\mathbb{D}\TT\left[ X^{(i)} w\right],\\
\Omega\left[ \frac{X^{(i)}w}{(1-q^{-1}\sigma)(t^{-1}\sigma^{-1}-1)}\right]\mathbb{D}^*&= \left(1-\frac{w}{z_i}\right)\mathbb{D}^*\Omega\left[ \frac{X^{(i)}}{(1-q^{-1}\sigma)(t^{-1}\sigma^{-1}-1)}w \right],\\
\TT\left[-X^{(i)}w \right]\mathbb{D}^*&= \left( 1-z_iw \right)\mathbb{D}^*\TT\left[-X^{(i)}w \right].
\end{align}
\end{lem}

\begin{proof}
These all follow from an application of Lemma \ref{TOLem}.
We will only go through the proof of (\ref{Rel1}).
The $\Omega$-term in the equation will only interact nontrivially with the $\mathcal{T}$-term of $\mathbb{D}$, which is
\[
\TT\left[ (1-q\sigma)(1-t\sigma^{-1})\sum_{j\in\ZZ/r\ZZ} X^{(j)}z_j^{-1} \right].
\]
By Proposition \ref{SigProp}(2), we have
\[
\left[ (1-q\sigma)(1-t\sigma^{-1}) \right]^T=(1-q\sigma^{-1})(1-t\sigma).
\]
Thus, Lemma \ref{TOLem} gives
\begin{align*}
\Omega&\left[ \frac{X^{(i)}w}{(1-q\sigma^{-1})(1-t\sigma)} \right]\TT\left[ (1-q\sigma)(1-t\sigma^{-1})\sum_{j\in\ZZ/r\ZZ}X^{(j)} z_j^{-1}\right]\\
&= \Omega\left[ -\sum_{j\in\ZZ/r\ZZ}\delta_{i,j}\frac{w}{z_{j}} \right]\TT\left[ (1-q\sigma)(1-t\sigma^{-1})\sum_{j\in\ZZ/r\ZZ}X^{(j)} z_j^{-1}\right]\Omega\left[ \frac{X^{(i)}w}{(1-q\sigma^{-1})(1-t\sigma)} \right]\\
&= \left( 1-\frac{w}{z_i} \right)
\TT\left[ (1-q\sigma)(1-t\sigma^{-1})\sum_{j\in\ZZ/r\ZZ}X^{(j)} z_j^{-1}\right]\Omega\left[ \frac{X^{(i)}w}{(1-q\sigma^{-1})(1-t\sigma)} \right].\qedhere
\end{align*}
\end{proof}

\begin{cor}
Recall the elements $\tilde{p}_1^{(i)}$ from \ref{PTilde}.
\begin{enumerate}
\item For any vector $\vec{k}=(k_0,\ldots, k_{r-1})\in \ZZ^r$, the following hold:
\begin{align}
\label{MultD}
\left[\mathbb{D}_{\vec{k}}, \tilde{p}_1^{(i)}\right]&=\mathbb{D}_{\vec{k}-\epsilon_{i}},\\
\label{SkewD}
\left[\mathbb{D}_{\vec{k}}, p_1^\perp[X^{(i)}]\right]&= \mathbb{D}_{\vec{k}+\epsilon_{i}},\\
\label{MultD*}
\left[\mathbb{D}_{\vec{k}}^*, \tilde{p}_1^{(i)}\right]&= -qt\mathbb{D}_{\vec{k}-\epsilon_{i}}^*,\\
\left[\mathbb{D}_{\vec{k}}^*, p_1^\perp[X^{(i)}]\right]&= -\mathbb{D}_{\vec{k}+\epsilon_{i}}^*.
\end{align}
\item We also have the following relations:
\begin{align}
\label{NablaP1}
\nabla \tilde{p}_1^{(i)}\nabla^{-1}&= \mathbb{D}_{-\epsilon_i}\\
\label{NablaPerp}
\nabla^{-1} p_1^\perp[X^{(i)}]\nabla&=-\mathbb{D}_{\epsilon_i},\\
\nabla^{-1} \tilde{p}_1^{(i)}\nabla&= -qt\mathbb{D}^*_{-\epsilon_i},\\
\label{NablaPerp*}
\nabla p_1^\perp[X^{(i)}]\nabla^{-1}&=\mathbb{D}_{\epsilon_i}^*.
\end{align}
\end{enumerate}
\end{cor}

\begin{proof}
Part (1) follows from Lemma \ref{ODRel}.
For instance, for (\ref{MultD}), we take the coefficient of $wz^{-k_0}\cdots z^{-k_{r-1}}$ in (\ref{Rel1}).
For (\ref{MultD*}), we note that:
\[
\tilde{p}_1^{(i)}=p_1\left[ \frac{X^{(i)}}{(1-q\sigma^{-1})(1-t\sigma)} \right]= -q^{-1}t^{-1}p_1\left[ \frac{X^{(i)}}{(1-q^{-1}\sigma)(t^{-1}\sigma^{-1}-1)} \right]
\]

Part (2) follows from combining part (1) with Lemma \ref{P1Supp} and Corollary \ref{DDiag}.
We will only prove (\ref{NablaP1}).
First note that by (\ref{MultD}), 
\[
\left[\mathbb{D}_{\vec{0}},\tilde{p}_1^{(i)}\right]=\mathbb{D}_{-\epsilon_i}.
\]
Consider now the action of the left-hand-side on $H_\lambda\otimes e^{\core(\lambda)}$.
Lemma \ref{P1Supp} says that
\[
\tilde{p}_1^{(i)}H_\lambda=\sum_\mu c_{\lambda\mu}H_\mu
\]
where the $\mu$ are such that $\mu\supset\lambda$ and $\mu\backslash\lambda$ contains exactly one box of each color.
On the other hand, Corollary \ref{DDiag} tells us that
\begin{align*}
\mathbb{D}_{\vec{0}}\left(H_\nu\otimes e^{\core(\nu)}\right)
&= \left(\frac{-D_\lambda^{\bullet}}{(1-q)(1-t)}\right)^{(0)}\left(H_\nu\otimes e^{\core(\nu)}\right)\\
&=\left( \left( \frac{1}{(1-q)(1-t)} \right)^{(0)}-B_\nu^{(0)} \right)H_\nu\otimes e^{\core(\nu)}.
\end{align*}
Since only one color $0$ box $\square_\mu$ is in $\mu\backslash\lambda$, we have
\begin{align*}
\left[ \mathbb{D}_{\vec{0}}, \tilde{p}_1^{(i)} \right]\left( H_\lambda\otimes e^{\core(\lambda)} \right)&=\sum_{\mu}c_{\lambda\mu}\left( -B_\mu^{(0)}+B_\lambda^{(0)}  \right)\left(H_\mu\otimes e^{\core(\lambda)}\right)\\
&=\sum_{\mu}c_{\lambda\mu}\left( -\chi_{\square_\mu}  \right)\left(H_\mu\otimes e^{\core(\lambda)}\right)\\
&=\nabla\tilde{p}_1^{(i)}\nabla^{-1}\left(H_\lambda\otimes e^{\core(\lambda)}\right).
\qedhere
\end{align*}
\end{proof}

\subsubsection{Induction}
We will need an analogue of \cite[Theorem 2.1]{GHT}.
First, we prove the following:
\begin{thm}\label{InductThm}
Every element $f\otimes e^\alpha \in \Lambda_{q,t}^{\otimes r}\otimes\CC[Q]$ may be written as 
\begin{align}
\label{InductD}
 f\otimes e^\alpha&=P\left(\mathbb{D}_{-\epsilon_i}, \tilde{p}_1^{(i)}\right)(1\otimes e^\alpha)
\end{align}
for some noncommutative polynomial $P$ (with coefficients in $\CC(q,t)$) in the operators $\mathbb{D}_{-\epsilon_0},\ldots, \mathbb{D}_{-\epsilon_{r-1}}$ and $ \tilde{p}_1^{(0)},\ldots, \tilde{p}_1^{(r-1)}$.
Likewise, we can write
\begin{equation}
  f\otimes e^\alpha=P^*\left(\mathbb{D}^*_{-\epsilon_i}, \tilde{p}_1^{(i)}\right)(1\otimes e^\alpha)    
\label{InductD*}
\end{equation}
for some noncommutative polynomial $P^*$ in the operators $\mathbb{D}^*_{-\epsilon_0},\ldots, \mathbb{D}^*_{-\epsilon_{r-1}}$ and $ \tilde{p}_1^{(0)},\ldots, \tilde{p}_1^{(r-1)}$.
\end{thm}

\begin{proof}
By Proposition \ref{DownArrowProp}, it suffices to prove (\ref{InductD}).
Let 
\begin{align*}
 \tilde{h}_n^{(p)} = h_n\left[\frac{X^{(p)}}{(1-q\sigma^{-1})(1-t\sigma)}\right]
 && \text{ and }
 &&\tilde{h}_{\vec{\lambda}}= h_{\vec{\lambda}}\left[\frac{X^{\bullet}}{(1-q\sigma^{-1})(1-t\sigma)}\right].
\end{align*}
We first study the basis given by $\{\tilde{h}_{\vec{\lambda}}\otimes e^\alpha\}$.
Let us set
\[R_{a_1,\dots,a_k}  \coloneq 
\Omega\left[  \sum_{j=1}^k w_j \frac{X^{(a_j)}}{(1-q \sigma^{-1})(1-t \sigma)} \right] (1\otimes e^\alpha)
= \sum_{m_1,\dots,m_k \geq 0} \prod_{j=1}^k
w_j^{m_j} \tilde{h}^{(a_j)}_{m_j}\otimes e^\alpha .
\]
Any complete homogeneous symmetric function $\tilde{h}_{\vec{\lambda}}$ can be realized as a coefficient of some $R_{a_1,\ldots, a_k}$ by choosing (possibly repeating) suitable colors $a_i$.

First note that for given colors $a_1,\dots,a_k$, by (\ref{Rel1}),
\begin{align}
\label{Induct1}
\prod_{j=1}^k (1-w_j/z_{a_j})~ & \mathbb{D} \Omega\left[  \sum_{j=1}^k w_j \frac{X^{(a_j)}}{(1-q \sigma^{-1})(1-t \sigma)} \right] (1\otimes e^\alpha)\\
\label{Induct2}
&= 
\Omega\left[  \sum_{j=1}^k w_j \frac{ X^{(a_j)}}{(1-q \sigma^{-1})(1-t \sigma)  } \right] \mathbb{D} (1\otimes e^\alpha).   
\end{align}
We will compare the constant terms in the $z$-variables of both sides.
By Corollary \ref{DDiag}, $\mathbb{D}_{\vec{0}}(1\otimes e^\alpha) = C_\alpha(1\otimes e^\alpha)$, where
\[
C_\alpha=\left(\frac{-D_\alpha^\bullet}{(1-q)(1-t)}\right)^{(0)}=-B_\alpha^{(0)}+\left(\frac{1}{(1-q)(1-t)}\right)^{(0)}\not=0.
\]
Thus, taking the constant term in the $z$-variables in (\ref{Induct2}) gives us
\begin{align}
\label{Induct3}
    \left\{\Omega\left[  \sum_{j=1}^k w_j \frac{ X^{(a_j)}}{(1-q \sigma^{-1})(1-t \sigma)  } \right] \mathbb{D} (1\otimes e^\alpha) \right\}_0
    =C_\alpha R_{a_1,\ldots, a_k}.
\end{align}

On the other hand, since $\left\{ (z_{b_1} \cdots z_{b_l})^{-1} \mathbb{D}  \right\}_0 = \mathbb{D}_{-\epsilon_{b_1}- \cdots - \epsilon_{b_l}}$, the constant term of the $z$-variables in (\ref{Induct1}) can be written as:
\begin{align}
\nonumber
&
\left\{\prod_{j=1}^k (1-w_j/z_{a_j}) ~ \mathbb{D} \Omega\left[  \sum_{j} w_j \frac{X^{(a_j)}}{(1-q\sigma^{-1})(1-t\sigma)} \right] (1\otimes e^\alpha) \right\}_{0}
\\
\label{Induct4}
& \hspace{5em}
= \mathbb{D}_{\vec{0}}\left(R_{a_1,\dots,a_k}\right)+ \sum_{ \substack {\varnothing \not = \{i_1,\dots,i_l\} \\
\hspace{2em} \subseteq \{1,\dots,k\}} } (-1)^l w_{i_1} \cdots w_{i_l}
\mathbb{D}_{-\epsilon_{a_{i_1}} - \cdots - \epsilon_{a_{i_l}}} R_{a_1,\dots, a_k}.
\end{align}
We will instead write $\{b_1,\dots, b_l\} \subseteq \{a_1,\dots, a_k\}$ to mean that $b_j = a_{i_j} $ for $\{i_1,\dots, i_l\} \subseteq \{1,\dots, k\}$. 
Extracting the coefficients of $w_1^{m_1}\cdots w_k^{m_k}$ from (\ref{Induct3}) and (\ref{Induct4}) and equating them, we get
\begin{align*}
C_\alpha\tilde{h}^{(a_1)}_{m_1} \cdots \tilde{h}^{(a_k)}_{m_k}\otimes e^\alpha  &= 
\mathbb{D}_{\vec{0}}\left( \tilde{h}^{(a_1)}_{m_1} \cdots \tilde{h}^{(a_k)}_{m_k}\otimes e^\alpha \right)\\
&\quad+
\sum_{ \substack {\varnothing\not=\{b_1,\dots,b_l\} \\
\hspace{2em} \subseteq \{a_1,\dots,a_k\}} } (-1)^l 
\mathbb{D}_{-\epsilon_{b_1} - \cdots - \epsilon_{b_l}} \left( 
\tilde{h}^{(a_1)}_{m_1-u_1} \cdots \tilde{h}^{(a_k)}_{m_k-u_k} 
\otimes e^\alpha
\right),
\end{align*}
or rather
\[ (C_\alpha- \mathbb{D}_{\vec{0}}) \left(\tilde{h}^{(a_1)}_{m_1} \cdots \tilde{h}^{(a_k)}_{m_k}\otimes e^\alpha\right)  = 
\sum_{ \substack {\varnothing\not= \{b_1,\dots,b_l\} \\
\hspace{2em} \subseteq \{a_1,\dots,a_k\}} } (-1)^l 
\mathbb{D}_{-\epsilon_{b_1} - \cdots - \epsilon_{b_l}} \left( 
\tilde{h}^{(a_1)}_{m_1-u_1} \cdots \tilde{h}^{(a_k)}_{m_k-u_k} \otimes e^\alpha
\right),
\]
where in the last sum, $u_i =\chi(a_i \in \{b_1,\dots, b_l\})$ indicates if $a_i$ is in the set of $b_j$'s.

By Equation (\ref{MultD}), we can write
\begin{align}
\nonumber
\mathbb{D}_{-\epsilon_{b_1} - \cdots - \epsilon_{b_k}}
& = \left[\mathbb{D}_{-\epsilon_{b_1} - \cdots - \epsilon_{b_{k-1}}}, \tilde{p}_1^{(b_k)} \right] \\
\nonumber
& = \left[ \left[\mathbb{D}_{-\epsilon_{b_1} - \cdots - \epsilon_{b_{k-2}}}, \tilde{p}_1^{(b_{k-1})} \right], \tilde{p}_1^{(b_k)} \right]  \\
\label{InductCommute}
& = \left[ \ldots \left[ \mathbb{D}_{-\epsilon_{b_1}}, \tilde{p}_1^{(b_2)}\right], \dots, \tilde{p}_1^{(b_k)}\right],
\end{align}
and by induction on degree, we can now assume that there is a noncommutative polynomial $P_{m_1,\dots,m_k}$ for which
\[
(C_\alpha- \mathbb{D}_{\vec{0}})\left(\tilde{h}^{(a_1)}_{m_1} \cdots \tilde{h}^{(a_k)}_{m_k}\otimes e^\alpha\right)
= P_{m_1,\dots,m_k}\left( \mathbb{D}_{-\epsilon_i}, \tilde{p}_1^{(i)} \right)(1\otimes e^\alpha).
\]
This implies that for the wreath Macdonald basis, for $\lambda$ with $\core(\lambda)=\alpha$,
\[
(C_\alpha-\mathbb{D}_{\vec{0}}) \left( H_\lambda\otimes e^\alpha\right) = P_{H_\lambda}\left( \mathbb{D}_{-\epsilon_i}, \tilde{p}_1^{(i)} \right)(1\otimes e^\alpha)
\]
for some noncommutative polynomial $P_{H_\lambda}$.
But 
\[(C_\alpha-\mathbb{D}_{\vec{0}}) \left( H_\lambda\otimes e^\alpha\right) =  \left(B_\lambda^{(0)}-B_\alpha^{(0)}\right)H_\lambda\otimes e^\alpha,\]
which, outside of the trivial case of $H_\lambda=H_\alpha=1$, is a nonzero scalar multiple.
The result then follows for $H_\lambda$  by dividing $P_{H_\lambda}$ by this scalar. 
Since $\{H_\lambda\,|\, \core(\lambda)=\alpha\}$ is a basis of $\Lambda_{q,t}^{\otimes r}$, the result holds for any $f$.
\end{proof}

\begin{cor}\label{D1P1Ind}
For any $f\in\Lambda_{q,t}^{\otimes r}$ of degree $k$, we can write
\begin{align*}
 f &= \sum_{i=0}^{r-1} \tilde{p}_1^{(i)} A_{i} 
+ \mathbb{D}_{-\epsilon_i} B_i =\sum_{i=0}^{r-1} \tilde{p}_1^{(i)} A^*_{i} 
+ \mathbb{D}^*_{-\epsilon_i} B_i^*
\end{align*}
for some $A_i,B_i, A^*_i,B^*_i \in \Lambda_{q,t}^{\otimes r}$ of degree $k-1$.
\end{cor}

\begin{proof}
Note that $\tilde{p}_1^{(i)}$, $\mathbb{D}_{-\epsilon_i}$, and $\mathbb{D}_{-\epsilon_i}^*$ increase degree by $1$.
\end{proof}

\begin{cor}\label{D0Gen}
Every element $f\otimes e^\alpha \in \Lambda_{q,t}^{\otimes r}\otimes\CC[Q]$ may be written as 
\begin{align*}
 f\otimes e^\alpha&=P\left(\mathbb{D}_{\vec{0}}, \tilde{p}_1^{(0)},\ldots, \tilde{p}_1^{(r-1)}\right)(1\otimes e^\alpha)   \\
 &=P^*\left(\mathbb{D}^*_{\vec{0}}, \tilde{p}_1^{(0)},\ldots, \tilde{p}_1^{(r-1)}\right)(1\otimes e^\alpha)   
\end{align*}
for some noncommutative polynomial $P$ in the operators $\{\mathbb{D}_{\vec{0}}, \tilde{p}_1^{(0)},\ldots, \tilde{p}_1^{(r-1)}\}$ and some noncommutative polynomial $P^*$ in $\{\mathbb{D}_{\vec{0}}^*, \tilde{p}_1^{(0)},\ldots, \tilde{p}_1^{(r-1)}\}$.
\end{cor}

\begin{proof}
    Combine Theorem \ref{InductThm} with (\ref{MultD}) and (\ref{MultD*}), utilizing
    \begin{equation*}
        \begin{aligned}
         \mathbb{D}_{-\epsilon_i}=\left[\mathbb{D}_{\vec{0}}, \tilde{p}_1^{(i)}\right] && \text{ and } &&
         \mathbb{D}^*_{-\epsilon_i}=-q^{-1}t^{-1}\left[\mathbb{D}_{\vec{0}}^*, \tilde{p}_1^{(i)}\right].
        \end{aligned}    
    \end{equation*}
     \qedhere       
\end{proof}

\subsection{Delta functions}\label{DeltaSec}
We now introduce the following ``delta functions'':
\begin{align*}
\mathbb{E}_\lambda& \coloneq \Omega\left[\sum_{i\in\ZZ/r\ZZ} X^{(i)}\left( \frac{D_\lambda^\bullet}{(1-q)(t-1)} \right)^{(i)} \right]\\
&= \Omega\left[\sum_{i\in\ZZ/r\ZZ} X^{(i)}\left( \left(\frac{1}{(1-q)(1-t)} \right)^{(i)}-B_\lambda^{(i)} \right)\right]
\end{align*}
It is an element in the completion of $\Lambda_{q,t}^{\otimes r}$ by degree.

\begin{prop}\label{DeltaExpand}
For any $r$-core $\alpha$, $\mathbb{E}_\lambda$ can be written as
\[
\mathbb{E}_\lambda=\sum_{\substack{\mu\\ \core(\mu)=\alpha}}\frac{H_\mu^\dagger[X^\bullet]H_\mu[\iota D_\lambda^\bullet]}{N_\mu}
=
\sum_{\substack{\mu\\ \core(\mu)=\alpha}}\frac{H_\mu[X^\bullet]H_\mu^\dagger[\iota D_\lambda^\bullet]}{N_\mu}.
\]
\end{prop}

\begin{proof}
We obtain $\mathbb{E}_\lambda[X^\bullet]$ from $\Omega_{q,t}[X^{-\bullet}Y^\bullet]$ by (1) applying $\iota$ to the $Y^\bullet$-variables and then (2) specializing them to $D_\lambda^\bullet$:
\begin{align*}
\Omega_{q,t}[X^{-\bullet}Y^\bullet]
&= \Omega\left[ X^{(0)}\left( \frac{Y^{(0)}}{(1-q\sigma^{-1})(t\sigma-1)}\right)+\cdots +X^{(r-1)}\left( \frac{Y^{(1)}}{(1-q\sigma^{-1})(t\sigma-1)} \right) \right]\\
&\overset{(1)}{\mapsto}\Omega\left[ X^{(0)}\left( \frac{Y^{(0)}}{(1-q\sigma)(t\sigma^{-1}-1)}\right)+\cdots +X^{(r-1)}\left( \frac{Y^{(r-1)}}{(1-q\sigma)(t\sigma^{-1}-1)} \right) \right]\\
&\overset{(2)}{\mapsto} \Omega\left[ X^{(0)}\left( \frac{D_\lambda^{\bullet}}{(1-q)(t-1)}\right)^{(0)}+\cdots +X^{(r-1)}\left( \frac{D_\lambda^\bullet}{(1-q)(t-1)} \right)^{(r-1)} \right].\
\end{align*}
The equations follow from Proposition \ref{CauchyProp}.
\end{proof}

\begin{cor}\label{DeltaFunc}
$\mathbb{E}_\lambda$ is a delta function for the $\iota$-twisted evaluation at $D_\lambda^\bullet$:
\[
\langle f, \mathbb{E}_\lambda\rangle_{q,t}'=f[\iota D_\lambda^\bullet].
\]
\end{cor}

\begin{cor}
The set $\{\mathbb{E}_\lambda\}$ is linearly independent in the completion of $\Lambda_{q,t}^{\otimes r}$.
\end{cor}

\subsubsection{Delta on delta}
To investigate how a Delta operator $\Delta^\dagger[f]$ acts on $\mathbb{E}_\lambda$, we will need to instead consider its action on $\mathbb{E}_\lambda\otimes e^{\core(\lambda)}$.

\begin{lem}\label{DeltaDelta}
Recall the elements $\hat{e}_n^{(p)}$ and $\hat{h}_n^{(p)}$ from \ref{ModEH}.
\begin{enumerate}
\item We have
\[
\Delta^\dagger\left[\hat{e}_n^{(p)}[-\iota X^\bullet]\right]\left( \mathbb{E}_\lambda\otimes e^{\core(\lambda)} \right)=\sum_{\mu}c_{\mu\lambda}\mathbb{E}_\mu\otimes e^{\core(\lambda)}
\]
where the $\mu$ that appear are such that:
\begin{itemize}
\item $\mu\supset\lambda$;
\item $\mu\backslash\lambda$ contains $n$ boxes of each color;
\item no color $p$ and color $(p+1)$ box in $\mu\backslash\lambda$ are horizontally adjacent.
\end{itemize}
 \item Similarly, we have
\[
\Delta^\dagger\left[\hat{h}_n^{(p)}[-\iota X^\bullet]\right]\left( \mathbb{E}_\lambda\otimes e^{\core(\lambda)} \right)=\sum_{\mu}d_{\mu\lambda}\mathbb{E}_\mu\otimes e^{\core(\lambda)}
\]
where the $\mu$ that appear are such that:
\begin{itemize}
\item $\mu\supset\lambda$;
\item $\mu\backslash\lambda$ contains $n$ boxes of each color;
\item no color $p$ and color $(p+1)$ box in $\mu\backslash\lambda$ are vertically adjacent. 
\end{itemize}
\end{enumerate}
\end{lem}

\begin{proof}
First consider the action of the translation components of both Delta operators (\ref{DeltaEForm}) and (\ref{DeltaHForm}) for a variable $z_{i,a}$:
\begin{align}
\nonumber
&\TT\left[ \left(-t X^{(i-1)}+(1+qt)X^{(i)}-qX^{(i+1)}\right)z_{i,a}^{-1} \right]\mathbb{E}_\lambda\\
\nonumber
&= \Omega\left[ \left( -t\left( \frac{D_\lambda^\bullet}{(1-q)(t-1)} \right)^{(i-1)}
+(1+qt)\left( \frac{D_\lambda^\bullet}{(1-q)(t-1)} \right)^{(i)}
-q\left( \frac{D_\lambda^\bullet}{(1-q)(t-1)} \right)^{(i+1)} \right)z_{i,a}^{-1} \right]\mathbb{E}_\lambda\\
\nonumber
&= \Omega\left[ \left(  \frac{(1-q)(1-t)D_\lambda^\bullet}{(1-q)(t-1)} \right)^{(i)}z_{i,a}^{-1} \right]\mathbb{E}_\lambda\\
\nonumber
&= \Omega\left[ -D_\lambda^{(i)}z_{i,a}^{-1} \right]\mathbb{E}_\lambda\\
\label{SkewDelta}
&= \frac{\displaystyle\prod_{\square\in R_i(\lambda)}\left( 1-qt\chi_\square/z_{i,a} \right)}{\displaystyle\prod_{\square\in A_i(\lambda)}\left( 1-\chi_\square/z_{i,a} \right)}\mathbb{E}_\lambda.
\end{align}
For the last equality, we have used (\ref{DLambda}).
Next, we look at $z_{i,a}^{-H_{i,0}}e^{\core(\lambda)}$.
As in the proof of Lemma \ref{FAdjoint2}, we can ignore the powers of $\ddd$ because we have an equal number of such operators per color.
Setting $\core(\lambda)=(c_0,\ldots, c_{r-1})$, equation (\ref{ChargeAR}) gives us
\[
z_{i,a}^{-H_{i,0}}e^{\core(\lambda)}=z_{i,a}^{\#R_i(\lambda)-\#A_i(\lambda)+\delta_{i,0}}e^{\core(\lambda)}
\]
We will use this balance out the binomials in (\ref{SkewDelta}).
Our final piece of preparation is to absorb the $\Omega$-components of the operators into $\mathbb{E}_\lambda$.
In total, after balancing all the remaining binomials, we have
\begin{equation}
\begin{aligned}
&\Delta^\dagger\left[ \hat{e}_n^{(p)}[-\iota X^\bullet] \right]\left(\mathbb{E}_\lambda\otimes e^{\core(\lambda)}\right)\\
&=\frac{t^n(1-qt)^{nr}}{\prod_{k=1}^n(1-q^kt^k)} 
\left\{
\prod_{1\le a<b\le n}\left[\rule{0cm}{2.2em}\right.
\frac{\left(z_{p,b}-z_{p,a}\right)\left(z_{p,b}-qtz_{p,a}\right)}{\left(z_{p,b}-t^{-1}z_{p+1,a}\right)\left(z_{p,b}-tz_{p-1,a}\right)} 
\right.
\\
&\quad\times 
\prod_{i\in \ZZ/r\ZZ\backslash\{p\}}
\frac{\left(z_{i,b}-z_{i,a}\right)\left(z_{i,b}-qtz_{i,a}\right)}{\left(z_{i,b}-qz_{i+1,a}\right)\left(z_{i,b}-tz_{i-1,a}\right)}\left]\rule{0cm}{2.2em}\right.\\
&\quad\times
\prod_{a=1}^n \left[\rule{0cm}{2.2em}\right.
\prod_{i\in \ZZ/r\ZZ\setminus\{p\}}
\left(\frac{z_{i+1,a}}{z_{i+1,a}-tz_{i,a}}\right)
\left( \frac{z_{i,a}}{z_{i,a}-qz_{i+1,a}} \right)
 \\
&\quad\times 
\left(\frac{z_{0,a}}{z_{p+1,a}}\right)
\left(\frac{z_{p+1,a}}{z_{p+1,a}-tz_{p,a}}\right)
\left]\rule{0cm}{2.2em}\right.
\prod_{a=1}^nz_{0,a}\prod_{i\in\ZZ/r\ZZ}\frac{\displaystyle\prod_{\square\in R_i(\lambda)}\left( z_{i,a}-qt\chi_\square \right)}
{\displaystyle\prod_{\square\in A_i(\lambda)}\left( z_{i,a}-\chi_\square \right)}\\
&\left.\quad\times
\Omega\left[ \sum_{i\in\ZZ/r\ZZ}X^{(i)}\left(\left(\frac{1}{(1-q)(1-t)}\right)^{(i)}-B_\lambda^{(i)}   -\sum_{a=1}^nz_{i,a} \right) \right]
\right\}_0\otimes e^{\core(\lambda)},
\end{aligned}
\label{DeltaELem}
\end{equation}
and
\begin{equation}
\begin{aligned}
&\Delta^\dagger\left[ \hat{h}_n^{(p)}[-\iota X^\bullet] \right]\left(\mathbb{E}_\lambda\otimes e^{\core(\lambda)}\right)\\
&=
\frac{(-1)^n\left(1- qt \right)^{nr}}{\prod_{k=1}^n\left(1- q^{k}t^{k} \right)}
\left\{\prod_{1\le a<b\le n}\left[\rule{0cm}{2.2em}\right.
\frac{\left(z_{p,b}-z_{p,a}\right)\left(z_{p,b}-qtz_{p,a}\right)}{\left(z_{p,b}-q^{-1}z_{p-1,a}\right)\left(z_{p,b}-qz_{p+1,a}\right)}\right. \\
&\quad\times 
\prod_{i\in \ZZ/r\ZZ\backslash\{p\}}
\frac{\left(z_{i,b}-z_{i,a}\right)\left(z_{i,b}-qtz_{i,a}\right)}{\left(z_{i,b}-qz_{i+1,a}\right)\left(z_{i,b}-tz_{i-1,a}\right)} \left]\rule{0cm}{2.2em}\right.\\
&\quad\times \prod_{a=1}^n
\left[\rule{0cm}{2.2em}\right.
\prod_{i\in \ZZ/r\ZZ\setminus\{p\}}
 \left(\frac{z_{i-1,a}}{z_{i-1,a}-qz_{i,a}}\right)\left(\frac{z_{i,a}}{z_{i,a}-tz_{i-1,a}}\right)
 \\
&\quad \times
\left(\frac{z_{0,a}}{z_{p,a}}\right)
\left(\frac{z_{p-1,a}}
{z_{p-1,a}-qz_{p,a}}\right) \left]\rule{0cm}{2.2em}\right.
\prod_{a=1}^nz_{0,a}\prod_{i\in\ZZ/r\ZZ}\frac{\displaystyle\prod_{\square\in R_i(\lambda)}\left( z_{i,a}-qt\chi_\square \right)}
{\displaystyle\prod_{\square\in A_i(\lambda)}\left( z_{i,a}-\chi_\square \right)}\\
&\left.\quad\times
\Omega\left[ \sum_{i\in\ZZ/r\ZZ}X^{(i)}\left(\left(\frac{1}{(1-q)(1-t)}\right)^{(i)}-B_\lambda^{(i)}   -\sum_{a=1}^nz_{i,a} \right) \right]
\right\}_0\otimes e^{\core(\lambda)}.
\end{aligned}
\label{DeltaHLem}
\end{equation}
In both formulas, the rational functions are Laurent series expanded assuming
\begin{equation*}
\begin{aligned}
|z_{i,a}|&=1,&
|q|,|t|&<1
\end{aligned}
\end{equation*}
For technical reasons, we will introduce an auxilliary variable $\ppp$.
The constant term (\ref{DeltaELem}) is equal to
\begin{align}
\label{DeltaELem1}
&\frac{t^n(1-qt)^{nr}}{\prod_{k=1}^n(1-q^kt^k)} 
\left\{
\prod_{1\le a<b\le n}\left[\rule{0cm}{2.2em}\right.
\frac{\left(z_{p,b}-z_{p,a}\right)\left(z_{p,b}-qtz_{p,a}\right)}{\left(z_{p,b}-\ppp^{-1}z_{p+1,a}\right)\left(z_{p,b}-\ppp z_{p-1,a}\right)}
\right.
\\
\label{DeltaELem2}
&\quad\times 
\prod_{i\in \ZZ/r\ZZ\backslash\{p\}}
\frac{\left(z_{i,b}-z_{i,a}\right)\left(z_{i,b}-qtz_{i,a}\right)}{\left(z_{i,b}-qz_{i+1,a}\right)\left(z_{i,b}-\ppp z_{i-1,a}\right)} \left]\rule{0cm}{2.2em}\right.\\
\label{DeltaELem3}
&\quad\times
\prod_{a=1}^n\left[\rule{0cm}{2.2em}\right.
\prod_{i\in \ZZ/r\ZZ\setminus\{p\}}
\left(\frac{z_{i+1,a}}{z_{i+1,a}-\ppp z_{i,a}}\right)
\left( \frac{z_{i,a}}{z_{i,a}-qz_{i+1,a}} \right)
 \\
\label{DeltaELem4}
&\quad\times 
\left(\frac{z_{0,a}}{z_{p+1,a}}\right)
\left(\frac{z_{p+1,a}}{z_{p+1,a}-\ppp z_{p,a}}\right)
\left]\rule{0cm}{2.2em}\right.
\prod_{a=1}^nz_{0,a}\prod_{i\in\ZZ/r\ZZ}\frac{\displaystyle\prod_{\square\in R_i(\lambda)}\left( z_{i,a}-qt\chi_\square \right)}
{\displaystyle\prod_{\square\in A_i(\lambda)}\left( z_{i,a}-\chi_\square \right)}\\
\label{DeltaELem5}
&\left.\quad\times
\Omega\left[ \sum_{i\in\ZZ/r\ZZ}X^{(i)}\left(\left(\frac{1}{(1-q)(1-t)}\right)^{(i)}-B_\lambda^{(i)}   -\sum_{a=1}^nz_{i,a} \right) \right]
\right\}_0\otimes e^{\core(\lambda)}\bigg|_{\ppp\mapsto t}
\end{align}
where the denominators are expanded assuming
\begin{equation*}
\begin{aligned}
|z_{i,a}|&=1, & |q|,|t|&<1, & |\ppp|<|q^at^b|\hbox{ for all }(a,b)\in\ZZ^2.
\end{aligned}
\end{equation*}
We will only present the proof of part (1) since part (2) is analyzed in a similar way; but we also note that after swapping $q\leftrightarrow t$ in (\ref{DeltaHLem}) and negating the color of variables, we obtain the formula for $\Delta^\dagger[\hat{e}_n^{(-p-1)}[-\iota X^\bullet]]$.
\\

\underline{The strategy}: We apply Lemma \ref{ConstTermLem}, starting with $z_{p,1}$ and proceeding \textit{downward} in cyclic order to $z_{p-1,1}$, then $z_{p-2,1}$, and so on. At each step, we will be evaluating $z_{i,1}$ at the character of a box; and in (\ref{DeltaELem5}) this corresponds to adding a character to $B_{\lambda}^{(i)}$. We will then proceed similarly with the remaining $z_{i,a}$'s, as we will describe.
In the end, we will arrive at some $B_\mu^{\bullet}$, and the goal is to show that the boxes which are added produce a partition $\mu$ which satisfies the conditions from the lemma. \\

Now, for $z_{p,1}$, we can always find a power of $z_{p,1}$ in the numerator of (\ref{DeltaELem3}); and
there are only two kinds of poles where $z_{p,1}$ is expanded in negative powers:
\begin{enumerate}
\item[($\chi$)]\textit{$\chi$-poles}: for some $\square\in A_p(\lambda)$, the pole $(z_{p,1}-\chi_\square)$ in (\ref{DeltaELem4});
\item[($t$)]\textit{$t$-pole}: the pole $(z_{p,1}-\ppp z_{p-1,1})$ in (\ref{DeltaELem3}).
\end{enumerate}
Thus, we have
\[
\frac{z_{p,1}F(z_{p,1})}{(z_{p,1}-\ppp z_{p-1,1})\displaystyle\prod_{\square\in A_p(\lambda)}(z_{p,1}-\chi_\square)},
\]
where $F(z_{p,1})$ is a power series in $z_{p,1}$.
Applying Lemma \ref{ConstTermLem}, we split into the cases where $z_{p,1}$ is evaluated at a $\chi$-pole or a $t$-pole. 
In the first case, in (\ref{DeltaELem5}), $B_\lambda^{(p)}$ is appended with an addable box and thus we still have a partition.

As we progress downward in cyclic order along the color index to $z_{i,1}$ with $i\not=p,p+1$, 
we note that there is a $z_{i-1,1}^2$ in the numerator of (\ref{DeltaELem3}); and
the poles where $z_{i,1}$ is expanded negatively include $t$- and $\chi$-poles.
But we also now have new kinds of poles we call a
\begin{enumerate}
\item[($q$)]\textit{$q$-pole}: the pole $(z_{i,1}-qz_{i+1,1})$ in (\ref{DeltaELem3}).
\end{enumerate}
If the previous variable $z_{i+1,1}$ was evaluated at the $t$-pole $(z_{i+1,1}-\ppp z_{i,1})$, then we have
\begin{equation}
\frac{z_{i,1}}{z_{i,1}-qz_{i+1}}\mapsto\frac{z_{i,1}}{z_{i,1}-q\ppp z_{i,1}}=(1-q\ppp)^{-1}
\label{TPoleCancel}
\end{equation}
in (\ref{DeltaELem3}), and so $z_{i,1}$ has no $q$-pole.
Also note that if $z_{i+1,1}$ was evaluated at the $t$-pole, then one must also consider poles in (\ref{DeltaELem4}) of the form
\[
\ppp^k z_{i,1}-\chi_\square
\]
for some $0<k<r$ and $\square\in A_{i+k}(\lambda)$.
However, because $|\ppp|<|q^at^b|$, these are not expanded in negative powers of $z_{i-1}$.
Thus, the evaluations occur in:
\begin{itemize}
\item \textit{horizontal chains} initiated at a $\chi$-pole evaluation followed by consecutive $q$-pole evaluations;
\item \textit{vertical chains} of consecutive $t$-pole evaluations that terminate at a $\chi$-pole evaluation;
\item a chain of consecutive $t$-pole evaluations leading up to $z_{p+1,1}$.
\end{itemize}
The first two append characters to $B_\lambda^{(i)}$ in (\ref{DeltaELem5}) upon the evaluation $\ppp\mapsto t$.
Notice that here, $z_{i,1}$ is evaluated at the character of a box of color $i$.
The only way these additions might fail to yield a partition is if for some $i$, $z_{i,1}=qt\chi_\square$ for some $\square\in R_i(\lambda)$.
This is prevented by the zero $(z_{i,1}-qt\chi_\square)$ in (\ref{DeltaELem4}), which can only be canceled away by the pole $(z_{i,1}-qz_{i+1})$ if $z_{i+1,1}$ is evaluated at $t\chi_\square$ (i.e. the box additions successfully complete the corner).

At $z_{p+1,1}$, two new considerations arise.
First, $z_{p+1,1}$ only has a single power in the numerator from (\ref{DeltaELem3}) and (\ref{DeltaELem4}).
As in (\ref{TPoleCancel}), if $z_{p+2,1}$ was evaluated at the $t$-pole $(z_{p+2,1}-\ppp z_{p+1,1})$, that power of $z_{p+1,1}$ may be canceled away.
However, in this case, the initiating variable of this chain of consecutive $t$-poles will have a leftover power remaining that is evaluated at $\ppp^kz_{p+1,}$ for some $k$.
Thus, we will always have a positive power of $z_{p+1,1}$ to implement Lemma \ref{ConstTermLem}. 
Second, there is a new kind of pole expanded in negative powers of $z_{p+1,1}$:
\begin{enumerate}
\item [($t^\uparrow$)]\textit{$t^\uparrow$-pole}: for $b>1$, $(z_{p,b}-\ppp^{-1}z_{p+1,1})$ in (\ref{DeltaELem1}).
\end{enumerate}
If we evaluate $z_{p+1,1}$ at such a pole, then we continue with $z_{p,b}$.
It only has the following poles expanded in negative powers of $z_{p,b}$:
\begin{itemize}
\item[($\chi$)]\textit{$\chi$-poles}: for some $\square\in A_p(\lambda)$, the pole $(z_{p,b}-\chi_\square)$ in (\ref{DeltaELem4});
\item[($t$)]\textit{$t$-poles}: for $a\le b$, $(z_{p,b}-\ppp z_{p-1,a})$ in (\ref{DeltaELem1}).
\end{itemize}
Only evaluation at the latter has the possibility of not evaluating at some $\chi_\square$; thus, this chain of $t$-pole evaluations continues.
The plan is to follow this chain of $t$-poles until it finally evaluates at a character.
Suppose that along this chain, we have evaluated $z_{i+1,c_1}$ at the pole $z_{i+1,c_1}-\ppp z_{i,c_2}$ for $c_2<c_1$, and consider now $z_{i,c_2}$.
Like for $z_{p,b}$, we have $\chi$- and $t$-poles, but we may also need to consider
\begin{enumerate}
\item[($q$)]\textit{$q$-poles}: for $c_3\le c_2$, $(z_{i,c_2}-qz_{i+1,c_3})$.
\end{enumerate}
However, since $c_1\ge c_3$, either $c_1=c_3$ and this pole has been canceled as in (\ref{TPoleCancel}) or $c_1<c_3$, in which case the result vanishes when $\ppp\mapsto t$ due to the factor $(z_{i+1,c_3}-qtz_{i+1,c_1})$ from (\ref{DeltaELem1}).
Finally, when $i=p+1$, there is also a $t^\uparrow$-pole that continues the chain.
When we finally evaluate at a character $\chi_\square$, the zeros $(z_{i,b}-z_{i,a})$ prevent duplicate evaluations while $(z_{i,a}-qt\chi_\square)$ for $\square\in R_i(\lambda)$ in (\ref{DeltaELem4}) and $(z_{i,b}-qtz_{i,1})$ in (\ref{DeltaELem1}) will ensure that the box additions to $B_\lambda^{(i)}$ in (\ref{DeltaELem5}) result in a partition.
Note that the removable box at the top of this vertical chain comes from evaluating a $z_{i,1}$ for some $i$.

After chasing down such a chain of $t$-poles resulting from evaluating $z_{p+1,1}$ at a $t^\uparrow$-pole, we continue onto $z_{p,2}$.
We may have variables that have already been evaluated in the vertical chain involving the $t^\uparrow$-pole, in which case we move down in cyclic order along the color index until we hit $p+1$, after which we move up in the second index.
The analysis is as before: we continue step by step and chase down any chain resulting from a $t^\uparrow$-pole.
Zeroes of the form $(z_{i,b}-z_{i,a})$ in (\ref{DeltaELem1}) prevent duplicate evaluations, while those of the form $(z_{i,a}-qt\chi_\square)$ for $\square\in R_i(\lambda)$ in (\ref{DeltaELem4}) and $(z_{i,b}-qtz_{i,a})$ in (\ref{DeltaELem1}) keep the characters in (\ref{DeltaELem5}) to those of a partition.

At the end of this process, the box additions in (\ref{DeltaELem5}) result in $B_\mu^\bullet$ for some partition $\mu$, and no color $p$ and $(p+1)$ boxes in $\mu\backslash\lambda$ are horizontally adjacent due to the lack of $q$-poles for $z_{p+1,a}$.
\end{proof}

\begin{cor}\label{ClosedCor}
The span of $\{\mathbb{E}_\lambda\otimes e^{\core(\lambda)}\}$ is closed under the action of $\Delta^\dagger\left[f\right]$ for any $f\in\Lambda_{q,t}^{\otimes r}$.
\end{cor}

\subsubsection{The map $\mathsf{V}$}\label{MapV}
By Corollary \ref{ClosedCor}, we can define the following important map.

\begin{defn}\label{VDef}
The linear map
\[
\mathsf{V}:\Lambda_{q,t}^{\otimes r}\otimes \CC[Q]\rightarrow \mathrm{span}\left\{ \mathbb{E}_\lambda\otimes e^{\core(\lambda)} \right\}
\]
is defined by setting two conditions:
\begin{itemize}
\item $\mathsf{V}(fg\otimes e^\alpha)=\Delta^\dagger\left[ f[-\iota X^\bullet] \right]\left(\mathsf{V}\left( g\otimes e^\alpha \right)\right)$,
\item $\mathsf{V}(1\otimes e^{\alpha})=\mathbb{E}_\alpha\otimes e^\alpha$.
\end{itemize}
\end{defn}

\begin{thm}\label{VThm}
We have
\[
\mathsf{V}(H_\lambda\otimes e^{\core(\lambda)})=H_\lambda[\iota D_{w_0\core(\lambda)}^\bullet]\mathbb{E}_\lambda\otimes e^{\core(\lambda)}
\]
\end{thm}

\begin{proof}
This theorem follows from the combinatorial results of Section 5 from \cite{WreathEigen}.
Specifically, we have $\Lambda_{q,t}^{\otimes r}$ acting on a vector space with a basis $\{\mathbb{E}_\lambda\otimes e^{\core(\lambda)}\}$ indexed by partitions, where $f\in \Lambda_{q,t}^{\otimes r}$ acts via the operator $\Delta^\dagger\left[ f[-\iota X^\bullet] \right]$.
For a partition $\lambda$, we consider the expansions of
\begin{align}
\label{VE}
\mathsf{V}\left( \hat{e}_{\quot(\lambda)}\otimes e^{\core(\lambda)} \right)&= \sum_{\mu}a_{\lambda\mu}\mathbb{E}_\mu\otimes e^{\core(\lambda)} ~ \text{ and }\\
\label{VH}
\mathsf{V}\left( \hat{h}_{\quot(\lambda)}\otimes e^{\core(\lambda)} \right)&= \sum_{\mu}b_{\lambda\mu}\mathbb{E}_\mu\otimes e^{\core(\lambda)}.
\end{align}
The actions of $\hat{e}_n^{(p)}$ and $\hat{h}_n^{(p)}$ add boxes to the partition index as prescribed in Lemma \ref{DeltaDelta}.
Thus, for any $\mu$ appearing on the right-hand-side of (\ref{VE}), each ordering $\mathcal{O}_c$ of the columns of $\quot(\lambda)$ endows $\mu$ with a certain kind of tableau, wherein each labeling records the boxes added from the action of a $\hat{e}_n^{(p)}$.
These are the \textit{row-strict $r$-tableaux} of content $(\lambda,\mathcal{O}_c)$ from \textit{loc. cit.}
Similarly, the $\mu$ in the right-hand-side of (\ref{VH}) possess a certain kind of tableau for each ordering $\mathcal{O}_r$ of the rows in $\quot(\lambda)$; these are called \textit{column-strict $r$-tableaux} of content $(\lambda,\mathcal{O}_r)$.

Additionally, these $\mu$ must satisfy a nested condition.
On the left-hand-side of (\ref{VE}), we can forget some columns of $\quot(\lambda)$ to obtain $\quot(\lambda')$ for some smaller partition where $\core(\lambda)=\core(\lambda')$; we take the smaller product $\hat{e}_{\quot(\lambda')}$.
Let $\mathcal{O}_c$ be an ordering of the columns of $\quot(\lambda)$ where the removed columns are at the end.
In terms of $\mu$, the column removal corresponds to taking a row-strict column tableau on $\mu$ of content $(\lambda,\mathcal{O}_c)$ and cutting off the boxes corresponding to the removed columns, yielding the smaller partition $\mu'$.
It should then be possible to produce such a $\mu'$ that has nonzero coefficient in the expansion for $\hat{e}_{\quot(\lambda')}$.

All-in-all, the existence of the row-strict $r$-tableaux and this nesting condition is summarized in \textit{loc. cit.} in the definition of \textit{strongly row-strict $\lambda$-tabloidizable}.
The $\mu$ appearing on the right-hand-side of (\ref{VE}) must be strongly row-strict $\lambda$-tabloidizable, and by Lemma 5.6 of \textit{loc. cit.}, this implies $\mu\le_r\lambda$.
Likewise, the $\mu$ appearing on the right-hand-side of (\ref{VH}) are \textit{strongly column-strict $\lambda$-tabloidizable}, which by the same lemma implies $\mu\ge_r\lambda$.
These together imply that given $\lambda$,
\begin{align*}
\mathsf{V}\left( \mathrm{span}\left\{ \hat{e}_{\quot(\mu)}\otimes e^{\core(\lambda)}\, \middle|\, \mu\le_r\lambda \right\} \right)&\subset
\mathrm{span}\left\{ \mathbb{E}_\mu\otimes e^{\core(\lambda)}\,\middle|\,\mu\le_r\lambda  \right\}~ \text{ and }\\
\mathsf{V}\left( \mathrm{span}\left\{ \hat{h}_{\quot(\mu)}\otimes e^{\core(\lambda)}\, \middle|\, \mu\ge_r\lambda \right\} \right)&\subset
\mathrm{span}\left\{ \mathbb{E}_\mu\otimes e^{\core(\lambda)}\,\middle|\,\mu\ge_r\lambda  \right\}.
\end{align*}
By Proposition \ref{IntDef}, it follows that $\mathsf{V}(H_\lambda\otimes e^{\core(\lambda)})$ is a scalar multiple of $\mathbb{E}_\lambda\otimes e^{\core(\lambda)}$.

To compute that scalar, first note that
\begin{equation}
\mathsf{V}\left( H_\lambda\otimes e^{\core(\lambda)} \right)=\Delta^\dagger\left[ H_\lambda[-\iota X^\bullet] \right]\left(\mathbb{E}_{\core(\lambda)}\otimes e^{\core(\lambda)}\right).
\label{VHLambda}
\end{equation}
It suffices to compare the constant term of the right-hand-side of (\ref{VHLambda}) with that of $\mathbb{E}_\lambda\otimes e^{\core(\lambda)}$.
To that end, we use Proposition \ref{DeltaExpand} to expand $\mathbb{E}_{\core(\lambda)}$ in the basis $\{H_{\mu}^\dagger\,|\, \core(\mu)=w_0\core(\lambda)\}$:
\[
\mathbb{E}_{\core(\lambda)}\otimes e^{\core(\lambda)}=\sum_{\substack{\mu \\ \core(\mu)=w_0\core(\lambda)}}\frac{H_{\mu}^\dagger[X^\bullet]H_{\mu}[\iota D_{\core(\lambda)}^\bullet]}{N_{\mu}}\otimes e^{\core(\lambda)}.
\]
By (\ref{AdjointEq}), the right-hand-side of (\ref{VHLambda}) is:
\begin{equation}
\Delta^\dagger\left[ H_\lambda[-\iota X^\bullet] \right]\left(\mathbb{E}_{\core(\lambda)}\otimes e^{\core(\lambda)}\right)
=
\sum_{\substack{\mu \\ \core(\mu)=w_0\core(\lambda)}}\frac{H_{\mu}^\dagger[X^\bullet]H_\lambda[\iota D_{\mu}^\bullet]H_{\mu}[\iota D_{\core(\lambda)}^\bullet]}{N_{\mu}}\otimes e^{\core(\lambda)}.
\label{DeltaHE}
\end{equation}
The constant term is the summand where $\mu=w_0\core(\lambda)$, in which case the right hand side of (\ref{DeltaHE}) gives us $H_\lambda[\iota D_{w_0\core(\lambda)}^\bullet]\otimes e^{\core(\lambda)}$.
On the other hand, from Proposition \ref{DeltaExpand}, the constant term of $\mathbb{E}_\lambda\otimes e^{\core(\lambda)}$ is $1\otimes e^{\core(\lambda)}$.
\end{proof}

Our main task will be to give an explicit formula for $\mathsf{V}$.
Note that despite lacking such a formula, we can already deduce reciprocity:

\begin{cor}\label{0Rec}
For partitions $\lambda$ and $\mu$ with $\core(\mu)=w_0\core(\lambda)$, the equality holds: 
\[
\frac{H_{\mu}\left[\iota D_\lambda^{\bullet}\right]}{H_{\mu}\left[\iota D^{\bullet}_{w_0\core(\mu)}\right]}=\frac{H_{\lambda}\left[\iota D^{\bullet}_{\mu}\right]}{H_\lambda\left[\iota D^{\bullet}_{w_0\core(\lambda)}\right]}.
\]
\end{cor}

\begin{proof}
Compare the coefficient of $H_{\mu}^\dagger[X^\bullet]$ in (\ref{DeltaHE}) with that of $H_\lambda[\iota D_{w_0\core(\lambda)}^\bullet]\mathbb{E}_\lambda\otimes e^{\core(\lambda)}$ given by using Proposition \ref{DeltaExpand}.
\end{proof}

\subsubsection{The main properties of $\mathsf{V}$} We reiterate here the main properties of $\mathsf{V}$.
In the next section, we will give an explicit expression for $\mathsf{V}$, as in Tesler's Identity, showing, in the end, that these identities hold for the explicit version. 
\begin{prop} \label{MainPropertiesofV}
    For any $f \in \Lambda_{q,t}^{\otimes r}$, let $\underline{f}$ denote the operation of multiplication by $f$. We have
    \begin{itemize}
        \item $\mathsf{V} \underline{f} = \Delta^\dagger\left[f[-\iota X^\bullet]\right] \mathsf{V}$,
        \item $f^{\perp} \mathsf{V} = \mathsf{V} \Delta\left[ f \left[ \frac{X^{\bullet}}{(1-q\sigma )(1-t\sigma^{-1})} \right] \right]$,
        \item $\mathsf{V}(1 \otimes e^{\alpha}) = \mathbb{E}_\alpha \otimes e^{\alpha},$
        \item $\mathsf{V} \left(H_{\lambda}[X^\bullet] \otimes e^{\core(\lambda)}\right) = H_\lambda[\iota D_{w_0\core(\lambda)}^\bullet]\mathbb{E}_\lambda \otimes e^{\core(\lambda)}.$
    \end{itemize}
\end{prop}
\begin{proof}
    The first and third properties are from the definition of $\mathsf{V}$. 
    The last property follows from Theorem \ref{VThm}.  
    The second condition follows from these. We will show that 
    \[
h_n^{\perp}[X^{(j)}] \mathsf{V} = \mathsf{V} \Delta\left[ h_n \left[ \frac{X^{(j)}}{(1-q\sigma^{-1})(1-t\sigma)}\right] \right]
    \]
by applying $\TT[ v X^{(j)}] = \sum_{n \geq 0 } v^n h_n^{\perp}[X^{(j)}]$ to  $\mathbb{E}_\lambda \otimes e^{\core(\lambda)}$:    
    \begin{align*}
     \TT[ v X^{(j)}] \mathbb{E}_\lambda \otimes e^{\core(\lambda)} & =
 \TT[ v X^{(j)}] \Omega\left[\sum_{i\in\ZZ/r\ZZ} X^{(i)}\left( \frac{D_\lambda^\bullet}{(1-q)(t-1)} \right)^{(i)} \right] \otimes e^{\core(\lambda)} \\
 & = \Omega\left[\sum_{i\in\ZZ/r\ZZ} \left(X^{(i)} + \delta_{i,j} v \right) \left( \frac{D_\lambda^\bullet}{(1-q)(t-1)} \right)^{(i)} \right]\otimes e^{\core(\lambda)} \\
 & = \Omega\left[  v \left( \frac{D_\lambda^\bullet}{(1-q)(t-1)} \right)^{(j)} \right]
 \Omega\left[\sum_{i\in\ZZ/r\ZZ} X^{(i)}  \left( \frac{D_\lambda^\bullet}{(1-q)(t-1)} \right)^{(i)} \right]\otimes e^{\core(\lambda)} \\
 & = \mathsf{V} \sum_{n \geq 0 }v^n h_n \left[\left( \frac{D_\lambda^\bullet}{(1-q)(t-1)} \right)^{(j)} \right]
  H_\lambda\otimes e^{\core(\lambda)}
  \\
  & = \mathsf{V} \sum_{n \geq 0 }v^n \Delta \left[ h_n \left[ \frac{X^{(j)}}{(1-q \sigma)(1-t \sigma^{-1})}  \right] \right]
  H_\lambda\otimes e^{\core(\lambda)}.
    \end{align*}
Comparing the coefficient of $v^n$ on both sides gives the desired property.
\end{proof}

\subsection{Proving the identity}
Here, we state and prove the main result of the paper:

\begin{thm}\label{Tesler}
The operator $\mathsf{V}$ from Definition \ref{VDef} is given by
\[
\mathsf{V}=\nabla\Omega\left[ \frac{X^{(0)}}{(1-q\sigma^{-1})(t\sigma-1)} \right]\TT\left[ X^{(0)} \right]\nabla
\]
Thus, when combined with Theorem \ref{VThm}, we obtain the Tesler formula:
\[
\mathbb{E}_\lambda\otimes e^{\core(\lambda)}=\nabla\Omega\left[ \frac{X^{(0)}}{(1-q\sigma^{-1})(t\sigma-1)} \right]\TT\left[ X^{(0)} \right]\nabla\left( \frac{H_\lambda}{H_\lambda[\iota D_{w_0\core(\lambda)}^\bullet]}\otimes e^{\core(\lambda)} \right)
\]
\end{thm}

\subsubsection{The eigenvalue of $\nabla$}
Before proceeding to the proof of Theorem \ref{Tesler}, we make a slight interlude and mention a corollary of the theorem.

\begin{cor}\label{NablaEv}
Recall the eigenvalue of $\nabla$ from (\ref{NableDef}).
The following equalities hold:
\begin{align*}
\nabla_{\core(\lambda)} H_\lambda= \prod_{\substack{\square\in\lambda\backslash\core(\lambda)\\\bar{c}_\square=0)}}(-\chi_\square)= H_\lambda[\iota D_{w_0\core(\lambda)}^\bullet].
\end{align*}
\end{cor}

\begin{proof}
By Theorem \ref{Tesler} and the defintion of $\mathsf{V}$, we have
\begin{equation}
\mathbb{E}_{w_0\core(\lambda)}\otimes e^{w_0\core(\lambda)}=\nabla\left(\Omega\left[ \frac{X^{(0)}}{(1-q\sigma^{-1})(t\sigma-1)} \right]\otimes e^{w_0\core(\lambda)}\right)
\label{NablaEv1}
\end{equation}
Consider the right-hand-side of (\ref{NablaEv1}).
First, note that by Proposition \ref{CauchyProp},
\begin{align*}
\Omega\left[ \frac{X^{(0)}}{(1-q\sigma^{-1})(t\sigma-1)} \right]
&= \Omega_{q,t}\left[ Y^{-\bullet} X^\bullet \right]\bigg|_{Y^{(i)}\mapsto \delta_{i,0}}\\
&= \sum_{\substack{\mu\\\core(\mu)=\core(\lambda)}}\frac{H_\mu[1]H_\mu^\dagger[X^\bullet]}{N_\mu}\\
&= \sum_{\substack{\mu\\\core(\mu)=\core(\lambda)}}\frac{H_\mu^\dagger[X^\bullet]}{N_\mu},
\end{align*}
because we have normalized $H_\mu$ such that $H_\mu[1]=1$.
Therefore, by Corollary \ref{NablaAdjoint}, the right-hand-side of (\ref{NablaEv1}) is
\begin{align}
\nonumber
\sum_{\substack{\mu\\\core(\mu)=\core(\lambda)}}&\frac{\nabla_{w_0\core(\lambda)}H_\mu^\dagger[X^\bullet]}{N_\mu}\otimes e^{w_0\core(\lambda)}\\
&=\sum_{\substack{\mu\\\core(\mu)=\core(\lambda)}}
\left(\prod_{\substack{\square\in\mu\backslash\core(\mu)\\\bar{c}_\square=0)}}(-\chi_\square)\right)
\frac{H_\mu^\dagger[X^\bullet]}{N_\mu}\otimes e^{w_0\core(\lambda)}.
\label{NablaEv2}
\end{align}
On the other hand, by Proposition \ref{DeltaExpand}, the left-hand-side of (\ref{NablaEv1}) is
\begin{equation}
\mathbb{E}_{w_0\core(\lambda)}\otimes e^{w_0\core(\lambda)}
=
\sum_{\substack{\mu\\ \core(\mu)=\core(\lambda)}}\frac{H_{\mu}^\dagger[X^\bullet]H_\mu[\iota D_{w_0\core(\lambda)}^\bullet]}{N_\mu}\otimes e^{w_0\core(\lambda)}.
\label{NablaEv3}
\end{equation}
The result follows from comparing the coefficients of $H_\lambda^\dagger[X^\bullet]$ in (\ref{NablaEv2}) and (\ref{NablaEv3}).
\end{proof}

\subsubsection{Reduction of the conditions}\label{RedProof}
Using Corollary \ref{D0Gen}, we can simplify the defining conditions for $\mathsf{V}$ even further:

\begin{prop}\label{VReduce}
$\mathsf{V}$ is determined by the following three conditions:
\begin{itemize}
\item \textit{Commuting with $\mathbb{D}_{\vec{0}}$}: 
\begin{equation}
\mathsf{V}\mathbb{D}_{\vec{0}}=p_1^\perp[X^{(0)}]\mathsf{V}.
\label{D0Rel}
\end{equation}
\item \textit{Commuting with $\tilde{p}_1^{(i)}$}: 
\begin{equation}
\mathsf{V}\tilde{p}_1^{(i)}=-\mathbb{D}_{\epsilon_0-\epsilon_i}\mathsf{V}.
\label{P1Rel}
\end{equation}
\item \textit{Base case}:
\[
\mathsf{V}(1\otimes e^\alpha)=\mathbb{E}_\alpha\otimes e^\alpha.
\]
\end{itemize}
\end{prop}
We first prove the first two commutation relations and prove the \textit{base case} in the subsequent sections.
\begin{proof}
Applying Corollary \ref{D0Gen}, we can reduce the defining condition
\[
\mathsf{V}(f\otimes e^\alpha)=\Delta^\dagger\left[ f[-\iota X^\bullet] \right]\left(\mathsf{V}\left( 1\otimes e^\alpha \right)\right)
\]
to commutation relations with $\mathbb{D}_{\vec{0}}$ and $\tilde{p}_1^{(i)}$.
These are given by \eqref{D0Rel} and \eqref{P1Rel}.
To see that, we apply Proposition \ref{MainPropertiesofV}, noting that
\begin{align*}
 \mathbb{D}_{\vec{0}}&=\Delta\left[p_1\left[\frac{X^{(0)}}{(1-q\sigma)(1-t\sigma^{-1})}\right]\right]\\
 \mathbb{D}_{\epsilon_0-\epsilon_i} &=\Delta^\dagger\left[\tilde{p}_1^{(i)}[\iota X^\bullet]\right].\qedhere
\end{align*}
\end{proof}

\subsubsection{Commutation relations}\label{Commute}
To prove the \textit{base case}, we must first show
\begin{prop} The following hold:
\begin{align}
\label{ComRel1}
\nabla\Omega\left[ \frac{X^{(0)}}{(1-q\sigma^{-1})(t\sigma-1)} \right]\TT\left[ X^{(0)} \right]\nabla \mathbb{D}_{\vec{0}}
& = p_1^\perp[X^{(0)}]\nabla\Omega\left[ \frac{X^{(0)}}{(1-q\sigma^{-1})(t\sigma-1)} \right]\TT\left[ X^{(0)} \right]\nabla,\\
\label{ComRel2}
\nabla\Omega\left[ \frac{X^{(0)}}{(1-q\sigma^{-1})(t\sigma-1)} \right]\TT\left[ X^{(0)} \right]\nabla \tilde{p}_1^{(i)} 
& = -\mathbb{D}_{\epsilon_0-\epsilon_i}\nabla\Omega\left[ \frac{X^{(0)}}{(1-q\sigma^{-1})(t\sigma-1)} \right]\TT\left[ X^{(0)} \right]\nabla.
\end{align}
\end{prop}
Both come from applying the relations from \ref{Relations}.
Let us begin with (\ref{ComRel1}).
$\mathbb{D}_{\vec{0}}$ commutes with $\nabla$ as they share the common eigenbasis $\{H_\lambda\otimes e^{\core(\lambda)}\}$.
Applying (\ref{Rel1}) and (\ref{Rel2}) yields
\begin{align*}
\nabla\Omega\left[ \frac{X^{(0)}}{(1-q\sigma^{-1})(t\sigma-1)} \right]\TT\left[ X^{(0)} \right]\nabla \mathbb{D}_{\vec{0}}
&=\nabla\Omega\left[ \frac{X^{(0)}}{(1-q\sigma^{-1})(t\sigma-1)} \right]\TT\left[ X^{(0)} \right]\mathbb{D}_{\vec{0}}\nabla\\
&= \left\{\left(\frac{1-z_0}{1-z_0^{-1}}\right)\nabla\mathbb{D}\Omega\left[ \frac{X^{(0)}}{(1-q\sigma^{-1})(t\sigma-1)} \right]\TT\left[ X^{(0)} \right]\nabla\right\}_0\\
&= -\nabla\mathbb{D}_{\epsilon_0}\Omega\left[ \frac{X^{(0)}}{(1-q\sigma^{-1})(t\sigma-1)} \right]\TT\left[ X^{(0)} \right]\nabla.
\end{align*}
Finally, we apply (\ref{NablaPerp}) to obtain the right-hand-side of (\ref{ComRel1}).

For (\ref{ComRel2}), we start with (\ref{NablaP1}):
\begin{align*}
\nabla\Omega\left[ \frac{X^{(0)}}{(1-q\sigma^{-1})(t\sigma-1)} \right]\TT\left[ X^{(0)} \right]\nabla \tilde{p}_1^{(i)}
&= \nabla\Omega\left[ \frac{X^{(0)}}{(1-q\sigma^{-1})(t\sigma-1)} \right]\TT\left[ X^{(0)} \right]\mathbb{D}_{-\epsilon_i}\nabla\\
&= \left\{\left( \frac{1-z_0}{1-z_0^{-1}} \right)\nabla\mathbb{D}\Omega\left[ \frac{X^{(0)}}{(1-q\sigma^{-1})(t\sigma-1)} \right]\TT\left[ X^{(0)} \right]\nabla\right\}_0\\
&= -\nabla\mathbb{D}_{\epsilon_0-\epsilon_i}\Omega\left[ \frac{X^{(0)}}{(1-q\sigma^{-1})(t\sigma-1)} \right]\TT\left[ X^{(0)} \right]\nabla,
\end{align*}
where, as in the previous paragraph, we applied (\ref{Rel1}) and (\ref{Rel2}).
We can commute $\mathbb{D}_{\epsilon_0-\epsilon_i}$ with $\nabla$ because by Lemma \ref{AdjointLem}(1) and Corollary \ref{NablaAdjoint}, both operators share the eigenbasis $\{H_\lambda^\dagger\otimes e^{w_0\core(\lambda)}\}$.

\subsubsection{Proof of the base case}\label{BaseCase}
Here, we want to show
\begin{equation}
\begin{aligned}
\mathbb{E}_\alpha\otimes e^\alpha&\overset{?}{=} 
\nabla\Omega\left[ \frac{X^{(0)}}{(1-q\sigma^{-1})(t\sigma-1)} \right]\TT\left[ X^{(0)} \right]\nabla (1\otimes e^\alpha)\\
&= \nabla\Omega\left[ \frac{X^{(0)}}{(1-q\sigma^{-1})(t\sigma-1)} \right](1\otimes e^\alpha).
\end{aligned}
\label{BaseCaseEq}
\end{equation}
Notice that because $H_\lambda[1]=1$, we have \begin{equation} \nabla\Omega\left[ \frac{X^{(0)}}{(1-q\sigma^{-1})(t\sigma-1)} \right](1\otimes e^\alpha)
=\sum_{\substack{\lambda\\ \core(\lambda)=w_0\alpha}}
\left(\prod_{\substack{\square\in\lambda\backslash\core(\lambda)\\\bar{c}_\square=0}}(-\chi_\square )\right)
\frac{H_\lambda^\dagger}{N_\lambda} \otimes e^\alpha.
\label{NablaP}
\end{equation}
In order to prove (\ref{BaseCaseEq}), by Corollary \ref{DeltaFunc}, it suffices to show that for any $f\in\Lambda_{q,t}^{\otimes r}$, we have
\[
\left\langle f\otimes e^{w_0\alpha}, \nabla\Omega\left[ \frac{X^{(0)}}{(1-q\sigma^{-1})(t\sigma-1)} \right](1\otimes e^\alpha)\right\rangle_{q,t}'
= f\left[ \iota D_\alpha^\bullet \right].
\]
By (\ref{NablaP}), this is obvious when $f=1$.
We will induct by degree.

If $\deg(f)=n$, by Corollary \ref{D1P1Ind}, we can write $f$ as
\[
f=\sum_{i=0}^{r-1} \tilde{p}_1^{(i)} A_i+\mathbb{D}^*_{-\epsilon_i}B_i,
\]
where $A_i,B_i\in\Lambda_{q,t}^{\otimes r}$ are of degree $n-1$.
It then suffices to show that for all $p\in\ZZ/r\ZZ$ and $g\in\Lambda_{q,t}^{\otimes r}$ with $\deg(g)=n-1$,
\begin{align}
\left\langle \tilde{p}_1^{(p)}g\otimes e^{w_0\alpha}, \nabla\Omega\left[ \frac{X^{(0)}}{(1-q\sigma^{-1})(t\sigma-1)} \right](1\otimes e^\alpha)\right\rangle_{q,t}'
&\overset{?}{=} \tilde{p}_1^{(p)}[\iota D_\alpha]g[\iota D_\alpha^\bullet]
\label{P1Induct},\\
\left\langle \mathbb{D}_{-\epsilon_p}^*\left(g\otimes e^{w_0\alpha}\right), \nabla\Omega\left[ \frac{X^{(0)}}{(1-q\sigma^{-1})(t\sigma-1)} \right](1\otimes e^\alpha)\right\rangle_{q,t}'
&\overset{?}{=} \mathbb{D}_{-\epsilon_p}^*(g\otimes e^{w_0\alpha})[\iota D_\alpha^\bullet].
\label{D1Induct}
\end{align}
Here, in the evaluation on the right-hand-side of (\ref{D1Induct}), we send $e^{w_0\alpha}\mapsto 1$.

First, consider (\ref{P1Induct}).
A consequence of Lemma \ref{FAdjoint1} is that
\begin{align*}
\left(\tilde{p}_1^{(p)}\right)^\dagger&= \left( p_1\left[ \frac{X^{(p)}}{(1-q\sigma^{-1})(1-t\sigma)} \right] \right)^\dagger\\
&=p_1^\perp\left[-(1-q\sigma)(1-t\sigma^{-1})\iota\left( \frac{X^{(p)}}{(1-q\sigma^{-1})(1-t\sigma)} \right) \right]\\
&=-p_1^\perp\left[ X^{(-p)} \right] .
\end{align*}
Thus, starting with the left-hand-side of (\ref{P1Induct}), taking adjoints and applying (\ref{NablaPerp}) and (\ref{Rel1}) yields
\begin{align*}
&\left\langle \tilde{p}_1^{(p)}g\otimes e^{w_0\alpha}, \nabla\Omega\left[ \frac{X^{(0)}}{(1-q\sigma^{-1})(t\sigma-1)} \right](1\otimes e^\alpha)\right\rangle_{q,t}'\\
& \hspace{2em}=\left\langle g\otimes e^{w_0\alpha}, -p_1^\perp[X^{(-p)}]\nabla\Omega\left[ \frac{X^{(0)}}{(1-q\sigma^{-1})(t\sigma-1)} \right](1\otimes e^\alpha)\right\rangle_{q,t}' \\
&\hspace{2em} = \left\langle g\otimes e^{w_0\alpha}, \nabla\mathbb{D}_{\epsilon_{-p}}\Omega\left[ \frac{X^{(0)}}{(1-q\sigma^{-1})(t\sigma-1)} \right](1\otimes e^\alpha)\right\rangle_{q,t}'\\
&\hspace{2em} = \left\langle g\otimes e^{w_0\alpha}, \nabla\Omega\left[ \frac{X^{(0)}}{(1-q\sigma^{-1})(t\sigma-1)} \right]\left(\mathbb{D}_{\epsilon_{-p}}-\mathbb{D}_{\epsilon_{-p}-\epsilon_0}\right)(1\otimes e^\alpha)\right\rangle_{q,t}'.
\end{align*}
Because $\mathbb{D}_{\epsilon_{-p}}$ decreases degree, it acts by zero on $1\otimes e^\alpha$.
Corollary \ref{DDiag} then implies
\begin{align*}
&\left\langle g\otimes e^{w_0\alpha}, \nabla\Omega\left[ \frac{X^{(0)}}{(1-q\sigma^{-1})(t\sigma-1)} \right]\left(\mathbb{D}_{\epsilon_{-p}}-\mathbb{D}_{\epsilon_{-p}-\epsilon_0}\right)(1\otimes e^\alpha)\right\rangle_{q,t}'\\
&\hspace{2em}= \left( \frac{D_\alpha^\bullet}{(1-q)(1-t)} \right)^{(-p)}
\left\langle g\otimes e^{w_0\alpha}, \nabla\Omega\left[ \frac{X^{(0)}}{(1-q\sigma^{-1})(t\sigma-1)} \right](1\otimes e^\alpha)\right\rangle_{q,t}'\\
&\hspace{2em}= \tilde{p}_1^{(p)}[\iota D_\alpha^\bullet]g[\iota D_\alpha^\bullet],
\end{align*}
where in the last step, we use the induction hypothesis.

We take a similar approach for (\ref{D1Induct}).
Starting from the left-hand-side, Lemma \ref{AdjointLem}(1) allows us to take the adjoint of $\mathbb{D}_{-\epsilon_{p}}^*$, after which we apply (\ref{NablaPerp*}):
\begin{align}
\nonumber
&\left\langle \mathbb{D}_{-\epsilon_p}^*\left(g\otimes e^{w_0\alpha}\right), \nabla\Omega\left[ \frac{X^{(0)}}{(1-q\sigma^{-1})(t\sigma-1)} \right](1\otimes e^\alpha)\right\rangle_{q,t}'\\
\nonumber
&\hspace{2em}=\left\langle g\otimes e^{w_0\alpha}, qt\mathbb{D}_{\epsilon_{-p}}^*\nabla\Omega\left[ \frac{X^{(0)}}{(1-q\sigma^{-1})(t\sigma-1)} \right](1\otimes e^\alpha)\right\rangle_{q,t}' \\
\label{D1Induct2}
&\hspace{2em}
=qt\left\langle g\otimes e^{w_0\alpha}, \nabla p_1^\perp[X^{(-p)}]\Omega\left[ \frac{X^{(0)}}{(1-q\sigma^{-1})(t\sigma-1)} \right](1\otimes e^\alpha)\right\rangle_{q,t}'.
\end{align}
To commute the $p_1^\perp[X^{(-p)}]$ further to the right, note that by Lemma \ref{TOLem},
\begin{equation}
\TT[X^{(-p)}u]\Omega\left[ \frac{X^{(0)}}{(1-q\sigma^{-1})(t\sigma-1)} \right]
=\Omega\left[-C^{(-p)}u \right]\Omega\left[ \frac{X^{(0)}}{(1-q\sigma^{-1})(t\sigma-1)} \right]\TT[X^{(-p)}u]
\label{TOComm}
\end{equation}
where
\[
C^{(-p)}=(1-q^r)^{-1}(1-t^r)^{-1}\sum_{\substack{0\le i,j\le r-1\\j-i=-p}}q^{i}t^j=\left( \frac{1}{(1-q)(1-t)} \right)^{(-p)}
\]
Taking the coefficient $u$ in (\ref{TOComm}), we get
\[
p_1^\perp[X^{(-p)}]\Omega\left[ \frac{X^{(0)}}{(1-q\sigma^{-1})(t\sigma-1)} \right]
=\Omega\left[ \frac{X^{(0)}}{(1-q\sigma^{-1})(t\sigma-1)} \right]\left(p_1^\perp[X^{(-p)}]- C^{(-p)}\right)
\]
Because $p_1^\perp[X^{(-p)}]$ annihilates $1\otimes e^\alpha$, (\ref{D1Induct2}) becomes
\begin{align*}
-qtC^{(-p)}
\left\langle g\otimes e^{w_0\alpha}, \nabla \Omega\left[ \frac{X^{(0)}}{(1-q\sigma^{-1})(t\sigma-1)} \right](1\otimes e^\alpha)\right\rangle_{q,t}'
=-qtC^{(-p)}g[\iota D_\alpha^\bullet],
\end{align*}
by induction.
It will be advantageous to rewrite the factor as
\begin{align*}
-qtC^{(p)}
&= -qt\left( \frac{1}{(1-q)(1-t)} \right)^{(-p)}\\
&=\frac{-qt}{(1-qt)}\left(\frac{1}{1-q}+\frac{t}{1-t}  \right)^{(-p)}\\
&= \frac{1}{(1-q^{-1}t^{-1})}\left( \frac{q^{p}}{1-q^r}+\frac{t^{r-p}}{1-t^r} \right)\\
&= \frac{-1}{(1-q^{-1}t^{-1})}\left( \frac{q^{p-r}}{1-q^{-r}}+\frac{t^{-p}}{1-t^{-r}} \right)
\end{align*}
for $0\le p\le r-1$.
All-in-all, we have reduced (\ref{D1Induct}) to showing
\begin{equation}
\mathbb{D}_{-\epsilon_p}^*(g\otimes e^{w_0\alpha})[\iota D_\alpha^\bullet]
\overset{?}{=}\frac{-1}{(1-q^{-1}t^{-1})}\left( \frac{q^{p-r}}{1-q^{-r}}+\frac{t^{-p}}{1-t^{-r}} \right)g[\iota D_\alpha^\bullet].
\label{D1InductFinal}
\end{equation}

We proceed by unpacking the left-hand-side of (\ref{D1InductFinal}).
To make the notation easier, it suffices to consider the case where $g$ can be factored according to its colors:
\[
g[X^\bullet]=\prod_{i\in\ZZ/r\ZZ}g_i[X^{(i)}].
\]
First, recall that in the definition of $\mathbb{D}^*$, all the rational functions are series expanded assuming
\begin{equation}
\begin{aligned}
|z_{i}|&=1,& |q|,|t|&>1.
\end{aligned}
\label{BaseConditions}
\end{equation}
On the other hand, when we evaluate the $\Omega$-terms, we have
\begin{equation}
\Omega\left[ D_\alpha^{(-i)}z_i \right]=\frac{\displaystyle\prod_{\square\in A_{-i}(\alpha)}(1-\chi_\square z_i)}{\displaystyle\prod_{\square\in R_{-i}(\alpha)}(1-qt\chi_\square z_i)}
\label{OmegaEv}
\end{equation}
For the above formula to hold, we need $|q^at^b|< |z_i|$, which is in conflict with (\ref{BaseConditions}).
This creates a conflict in describing the expansion region, and thus we introduce auxiliary variables $\ppp_q$ and $\ppp_t$ in place of $q$ and $t$ where expansions are taken with $|q|,|t|>1$.
Let us also note that the action of $z_i^{H_{i,0}}$ yields
\[
\left(\prod_{i\in\ZZ/r\ZZ }z_i^{H_{i,0}}\right)(e^{w_0\alpha})
=\left(\prod_{i\in\ZZ/r\ZZ}z_{i}^{c_{-i}-c_{-i-1}} \right)e^\alpha
=\left(z_0\prod_{i\in\ZZ/r\ZZ}z_i^{\#R_{-i}(\alpha)-\#A_{-i}(\alpha)}\right)e^\alpha,
\]
by Proposition \ref{CoreQuotDec}(4).
As in the proof of Lemma \ref{DeltaDelta}, these powers of $z_i$ can be used to balance the binomials in (\ref{OmegaEv}).
We obtain
\begin{align}
\label{BaseSetup}
\mathbb{D}_{-\epsilon_p}^*(g\otimes e^{w_0\alpha})[\iota D_\alpha^\bullet]
&= 
\left\{\left( \frac{z_0}{z_{p}} \right)\frac{(1-q^{-1}t^{-1})^{r-1}}{\prod_{i\in\ZZ/r\ZZ}(1-\ppp_t^{-1}z_{i+1}/z_{i})(1-\ppp_q^{-1}z_i/z_{i+1})}\right.\\
\nonumber
&\quad\times
\prod_{i\in\ZZ/r\ZZ}
\frac{\displaystyle\prod_{\square\in A_{-i}(\alpha)}(z_i^{-1}-\chi_\square)}{\displaystyle\prod_{\square\in R_{-i}(\alpha)}(z_i^{-1}-qt\chi_\square)}\\
\nonumber
& \quad\times
\left.\left.\prod_{i\in\ZZ/r\ZZ}g_i\left[ D_\alpha^{(-i)}+t^{-1}z_{i-1}^{-1}-(1+q^{-1}t^{-1})z_i^{-1}+q^{-1}z_{i+1}^{-1}  \right]\right\}_0\,\right|_{(\ppp_q,\ppp_t)\mapsto (q,t)}
\end{align}
where all rational functions are expanded assuming
\begin{equation}
\begin{aligned}
|z_{i}|&=1,& |q|,|t|&<1,& \text{ and }~ |\ppp_q|,|\ppp_t|&>1.
\end{aligned}
\label{BaseConditions2}
\end{equation}
It will also be convenient to make the variable switch $z_i\mapsto z_{-i}^{-1}$;
we also balance out the binomials in (\ref{BaseSetup}) and move things around:
\begin{align}
\label{Base1}
\mathbb{D}_{-\epsilon_p}^*(g\otimes e^{w_0\alpha})[\iota D_\alpha^\bullet]
&= 
\frac{1}{(1-q^{-1}t^{-1})}\left\{ \frac{z_{-p}}{z_{0}} \prod_{i\in\ZZ/r\ZZ}\frac{(1-q^{-1}t^{-1})z_i^2}{(z_i-\ppp_t^{-1}z_{i+1})(z_i-\ppp_q^{-1}z_{i-1})}\right.\\
\label{Base2}
&\quad\times
\prod_{i\in\ZZ/r\ZZ}
\frac{\displaystyle\prod_{\square\in A_{i}(\alpha)}(z_i-\chi_\square)}{\displaystyle\prod_{\square\in R_{i}(\alpha)}(z_i-qt\chi_\square)}\\
\nonumber
& \quad\times
\left.\left.\prod_{i\in\ZZ/r\ZZ}g_{i}\left[ D_\alpha^{(-i)}+t^{-1}z_{-i+1}-(1+q^{-1}t^{-1})z_{-i}+q^{-1}z_{-i-1}  \right]\right\}_0\,\right|_{(\ppp_q,\ppp_t)\mapsto(q,t)}
\end{align}
The functions are still expanded into series assuming (\ref{BaseConditions2}).
Finally, like in the proof of Lemma \ref{DeltaDelta}, we strengthen the conditions (\ref{BaseConditions2}) to include 
\begin{equation}
\begin{aligned}
|q^at^b\ppp_t|,|q^at^b\ppp_q|&>1
\hbox{ for all }(a,b)\in \NN^2.
\end{aligned}
\label{pqtcond}
\end{equation}
We are now ready to proceed with the computation.

\underline{The strategy}: We apply Lemma \ref{ConstTermLem} to find the constant term of $z_{r-1}$, then proceed down in cyclic order to $z_{r-2}$, then $z_{r-3}$, and so on until we reach $z_0$.
At each step, poles of $z_i$ will either give the evaluation of $z_i$ at a box $\square_i$ on the outer boundary of $\alpha$, or crucially, at the end, in terms of $z_0$.
We will show that the only surviving terms are of the latter kind, and the constant term is found by setting $z_0=0$.

We first note that there is a power of $z_{r-1}$ in the numerator; and the only poles expanded in negative powers of $z_{r-1}$ are given by the following terms:
\begin{enumerate}
\item[($t^{-1}$)]\textit{$t^{-1}$-pole}: the pole $(z_{r-1}-\ppp_t^{-1}z_{0})$ in (\ref{Base1});
\item[($q^{-1}$)]\textit{$q^{-1}$-pole}: the pole $(z_{r-1}-\ppp_q^{-1}z_{r-2})$ in (\ref{Base1});
\item[($\chi$)]\textit{$\chi$-poles}: for some $\square\in R_{r-1}(\alpha)$, the pole $(z_{r-1}-qt\chi_\square)$ in (\ref{Base2}).
\end{enumerate}
Once these are factored out, the remainder is a series in nonnegative powers of $z_{r-1}$.
Upon evaluating at the $t^{-1}$-pole, the factor
\[
\frac{(1-q^{-1}t^{-1})z_0}{z_0-\ppp_q^{-1}z_{r-1}}\mapsto\frac{1-q^{-1}t^{-1}}{1-\ppp_q^{-1}\qqq_t^{-1}}
\]
from (\ref{Base1}) becomes $1$ upon evaluation at $(\ppp_q,\ppp_t)\mapsto(q,t)$.
Likewise, upon evaluating at the $q^{-1}$ pole, the factor
\begin{equation}
\frac{(1-q^{-1}t^{-1})z_{r-2}}{z_{r-2}-\ppp_t^{-1}z_{r-1}}\mapsto\frac{1-q^{-1}t^{-1}}{1-\ppp_q^{-1}\qqq_t^{-1}}
\label{NoT}
\end{equation}
from (\ref{Base1}) becomes trivial as well.
Note that both of these cancellations remove a power of some variable in the numerator of (\ref{Base1}).

Our computation proceeds \textit{downward} in cyclic order.
For $z_i$ with $i\not=0$, we have the three poles as in the case of $z_{r-1}$.
If $z_{i+1}$ was evaluated at the $q^{-1}$-pole, then there are some additional considerations.
First, as in (\ref{NoT}), $z_i$ no longer has the $t$-pole.
Additionally, we encounter poles of the form $(\ppp_q^kz_{i}-qt\chi_\square)$ for some $\square\in R_{i+k}(\alpha)$ and $k>1$, but the stronger condition (\ref{pqtcond}) means that this pole is not expanded in negative powers of $z_i$.
Thus, by the time we reach $z_0$, the variables are broken up into contiguous groups, each evaluated at:
\begin{itemize}
\item a chain of $t^{-1}$-poles either \textit{starting} at $z_0$ or at a $\chi$-pole evaluation;
\item a chain of $q^{-1}$-poles either \textit{ending} at $z_0$ or at a $\chi$-pole evaluation.
\end{itemize}

At the final variable $z_0$, there are some slight differences, and one must consider how $z_1$ and $z_{r-1}$ were evaluated.
First, due to the $z_0$ in the denominator of (\ref{Base1}), one may be concerned that there is no leftover $z_0$ in the numerator to implement Lemma \ref{ConstTermLem} if $z_1$ is evaluated at a $q^{-1}$-pole or $z_{r-1}$ is evaluated at a $t^{-1}$-pole.
However, in the first case, the initial variable $z_i$ in the chain of $q^{-1}$-poles will have a leftover $z_i$ in the numerator that is evaluated to $\ppp_q^{-i}z_0$, and in the second case, the last variable $z_j$ in the chain of $t^{-1}$-poles will have a leftover $z_j$ that is evaluated at $\ppp_t^{-j+r}z_0$.

Thanks to the strengthened conditions (\ref{pqtcond}), we avoid poles from (\ref{Base2}) of the form
\begin{equation*}
\begin{aligned}
(\ppp_q^{-k}z_0-qt\chi_\square)&& \text{ or } &&
(\ppp_t^{-r+k}z_0-qt\chi_\square)
\end{aligned}
\end{equation*}
for $k>0$ and $\square\in R_{k}(\alpha)$.
Thus, the initial slate of options for $z_0$ is the same as the other variables, where as before, we need to take into account the cancellations that result when $z_1$ and $z_{r-1}$ are evaluated at $t^{-1}$- or $q^{-1}$-poles.
There are three cases where the $t^{-1}$- and $q^{-1}$-poles are both removed, which creates novelties:
\begin{itemize}
\item the variables all evaluated in one big chain of $t^{-1}$-poles starting at $z_0$;
\item the variables all evaluated in one big chain of $q^{-1}$-poles ending at $z_0$; or
\item $z_1$ is evaluated at the $q^{-1}$-pole and $z_{r-1}$ is evaluated at the $t^{-1}$-pole.
\end{itemize}
The first two cases will yield our desired result-- we address them at the end.
In the third case, if $z_0$ does not have a $\chi$-pole, then as remarked in the previous paragraph, there is a power of $z_0^2$ in the numerator.
After dividing by the $z_0$ in the denominator, we obtain $z_0F(z_0)$ for some power series $F$ and so the constant term is $0$.
Thus, this case results in either $0$ or a $\chi$-pole evaluation.

Outside of the cases where we have one big chain of $t^{-1}$-poles or $q^{-1}$-poles, our variables are split into contiguous segments consisting of a (possibly empty) chain of $q^{-1}$-poles ending at a $\chi$-pole evaluation, which then initiates a (possibly empty) chain of $t^{-1}$-poles.
Because of the $\chi$-pole evaluation, every variable $z_i$ in this segment specializes to a box of color $i$ when $(\ppp_q,\ppp_t)\mapsto(q,t)$.
The boxes form a sequence of cells on the boundary of $\alpha$ of shape \rotatebox[origin=c]{180}{L}-- we illustrate these segments in Figure \ref{fig:chichains}.
Our goal is to show that these terms vanish when $(\ppp_q,\ppp_t)\mapsto(q,t)$.
Vanishing occurs due to the zeros in (\ref{Base2}) of the form $(z_i-\chi_\square)$ for some $\chi\in A_i(\alpha)$, and here, the fact that $\alpha$ is a $r$-core plays a large role.

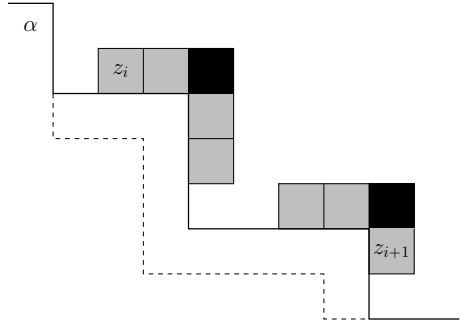
\begin{figure}[h]
\centering
\scalebox{.6}{
\begin{tikzpicture}
\draw[fill=black] (4,5)--(5,5)--(5,6)--(4,6)--(4,5);;
\draw[fill=lightgray] (4,5)--(2,5)--(2,6)--(4,6)--(4,5);;
\draw[fill=lightgray] (4,5)--(5,5)--(5,3)--(4,3)--(4,5);;
\draw (3,5)--(3,6);;
\draw (4,4)--(5,4);;
\draw[fill=black] (8,2)--(9,2)--(9,3)--(8,3)--(8,2);;
\draw[fill=lightgray] (8,2)--(8,3)--(6,3)--(6,2)--(8,2);;
\draw (7,2)--(7,3);;
\draw[fill=lightgray] (8,2)--(8,1)--(9,1)--(9,2);;
\draw (0.5,6.5) node {{\LARGE $\alpha$}};;
\draw (2.5,5.5) node {{\LARGE ${z_i}$}};;
\draw (8.5, 1.5) node {{\LARGE $z_{i+1}$}};;
\draw[dashed] (1,5)--(1,4)--(3,4)--(3,1)--(7,1)--(7,0)--(8,0);;
\draw[thick] (0,7)--(1,7)--(1,5)--(4,5)--(4,2)--(8,2)--(8,0)--(10,0);;
\end{tikzpicture}}
\caption{Two \rotatebox[origin=c]{180}{L}-shaped segments that avoid the addable boxes of $\alpha$.
The black squares are the $\chi$-poles while the chains of $t^{-1}$- and $q^{-1}$-poles are indicated by the gray squares.
If a segment starting at $z_i$ has $\chi$-pole northwest of the segment ending at $z_{i+1}$, then $\alpha$ has a removable ribbon of length $nr$ for some $n$, as illustrated by the dashed line.}
\label{fig:chichains}
\end{figure}

First, it is impossible that all the variables are evaluated in one \rotatebox[origin=c]{180}{L} shape that avoids the addable squares: 
If the northwesternmost square in the \rotatebox[origin=c]{180}{L} has color $i$, then the southeasternmost square has color $i+1$.
The $q^{-1}t^{-1}$-shifts of these two boxes are, respectively, the NW and SE endpoints of a removable ribbon of length $r$ in $\alpha$, contradicting $\alpha$ being an $r$-core.

If the variables are split into multiple \rotatebox[origin=c]{180}{L}-shapes, then we can find one shape $\hbox{\rotatebox[origin=c]{180}{L}}_1$ with northwestermost box $\square_i$ coming from evaluating some $z_i$ and another shape $\hbox{\rotatebox[origin=c]{180}{L}}_2$ with southeasternmost $\square_{i+1}$ box coming from evaluating $z_{i+1}$.
If both avoid addable squares, then we claim that $\hbox{\rotatebox[origin=c]{180}{L}}_1$ must be southeast of $\hbox{\rotatebox[origin=c]{180}{L}}_2$.
Otherwise, as before, the $q^{-1}t^{-1}$-shifts of $\square_i$ and $\square_{i+1}$ are the start and endpoints of a removable ribbon of $\alpha$ of length $kr$ for some $k$, in which case either the last $r$ cells or the first $r$ cells give a removable $r$-ribbon, contradicting $\alpha$ being an $r$-core.
We illustrate this in Figure \ref{fig:chichains}.
From this, we see that going down in cyclic order, later \rotatebox[origin=c]{180}{L}-shapes must be to the southeast, which leads to a contradiction once we go around the cycle of $\ZZ/r\ZZ$.

Finally, we are left with the two cases of a single long $t^{-1}$-chain and a single long $q^{-1}$-chain.
We will only address in detail the former.
Here we have evaluated $z_i\mapsto \ppp_t^{-r+i}z_0$ for $1<i\le r$.
This gives us
\begin{align}
\nonumber
&\frac{1}{(1-q^{-1}t^{-1})}\left\{ \left(\frac{\ppp_t^{-p}z_{0}}{z_{0}}\right) \frac{z_0}{(z_0-\ppp_t^{-r}z_{0})}\right.
\prod_{i=1}^r
\frac{\displaystyle\prod_{\square\in A_{i}(\alpha)}(\ppp_t^{-r+i}z_0-\chi_\square)}{\displaystyle\prod_{\square\in R_{i}(\alpha)}(\ppp_t^{-r+i}z_0-qt\chi_\square)}\\
\nonumber
& \quad\times
\left.\left.\prod_{i\in\ZZ/r\ZZ}g_{i}\left[ D_\alpha^{(-i)}+C_{(i)}z_0  \right]\right\}_0\,\right|_{(\ppp_q,\ppp_t)\mapsto(q,t)}\\
\label{AR}
&=\frac{t^{-p}}{(1-q^{-1}t^{-1})(1-t^{-r})}\left\{ 
\prod_{i=1}^r
\frac{\displaystyle\prod_{\square\in A_{i}(\alpha)}(\ppp_t^{-r+i}z_0-\chi_\square)}{\displaystyle\prod_{\square\in R_{i}(\alpha)}(\ppp_t^{-r+i}z_0-qt\chi_\square)}\right.\\
\nonumber
\nonumber
& \quad\times
\left.\left.\prod_{i\in\ZZ/r\ZZ}g_{i}\left[ D_\alpha^{(-i)}+C_{(i)}z_0  \right]\right\}_0\,\right|_{(\ppp_q,\ppp_t)\mapsto(q,t)} 
\end{align}
where $C_{(i)}\in\CC[q^{\pm 1},t^{\pm 1}, \ppp_t^{\pm 1}]$.
If there exists $\square\in R_0(\alpha)$, then $qt\chi_\square$ is the character of a box that is no more than $ r$ steps north of an addable box $\blacksquare$.
Otherwise, there would be a removable ribbon of length $r$ in $\alpha$.
Supposing that $\blacksquare$ has color $i$, we then have that if $z_0$ is evaluated at $qt\chi_\square$, then $\ppp_t^{-r+i}z_0-\chi_\blacksquare$ will evaluate to zero when $(\ppp_q,\ppp_t)\mapsto (q,t)$.

Thus, nothing is lost if we expand the poles $(z_0-qt\chi_\square)$ in positive powers of $z_0$.
We then have a power series in $z_0$.
Setting $z_0\mapsto0$ to find the constant term, we are left with
\begin{equation}
\frac{t^{-p}g[\iota D_\alpha^\bullet]}{(1-q^{-1}t^{-1})(1-t^{-r})}\prod_{i\in \ZZ/r\ZZ}\frac{\displaystyle\prod_{\square\in A_i(\alpha)}(-\chi_\square)}{\displaystyle\prod_{\square\in R_i(\alpha)}(-\chi_\square)}.
\label{TChar}
\end{equation}
The product of characters is $(-1)$:
every removable corner has exactly one addable corner in the same row, whereas every addable corner has exactly one removable corner in the same column. Therefore the $q$,$t$ powers in the product of character weights between addable and removable corners is the same. Lastly, there is one more addable corner than removable corners.

After this cancellation, (\ref{TChar}) becomes the $t$-summand of (\ref{D1InductFinal}).
The $q$-summand comes from the case of a single long $q^{-1}$-chain.\qed

\section{Variations}\label{VarSec}
Here, we catalogue some variants of Theorem \ref{Tesler}.

\subsection{Shifted versions}
The components of our operator $\mathsf{V}$ in Theorem \ref{Tesler} are very much centered on color $0$: the eigenvalues of $\nabla$ involve color $0$ boxes, the $\Omega$-term only involves the plethystic evaluation at $X^{(0)}$, and the $\TT$-term only acts on $X^{(0)}$.
One can try to shift all these colors simultaneously and ask if an analogous formula holds.
Here, we present the correct way of doing so.
Surprisingly, $\nabla$ does not change in the way outlined above. 
Recall that in \ref{Reverse}, we defined the action of $\sigma$ on partitions; we replace $\nabla_{\core(\lambda)}$ with $\nabla_{\sigma^{-k}\core(\lambda)}$.

\begin{thm}\label{ShiftedFourier}
Fix $k\in\ZZ/r\ZZ$ and set
\[
\mathbb{E}_\lambda^{(k)} \coloneq \Omega\left[\sum_{i\in\ZZ/r\ZZ}X^{(i)}\left( \frac{D_\lambda^\bullet}{(1-q)(t-1)} \right)^{(i+k)}  \right].
\]
We then have
\begin{equation}
\mathbb{E}_\lambda^{(k)}\otimes e^{\core(\sigma^{-k}\lambda)}
=
\nabla\Omega\left[ \frac{X^{(-k)}}{(1-q\sigma^{-1})(t\sigma-1)} \right]\TT\left[ X^{(-k)} \right]
\left(H_\lambda[\sigma^{-k}X^\bullet]\otimes e^{\sigma^{-k}\core(\lambda)}\right).
\label{TeslerK}
\end{equation}
\end{thm}
Theorem \ref{ShiftedFourier} follows from combining Theorems \ref{VMacK} and \ref{ShiftedTes} below.
As in Theorem \ref{Tesler}, the operator appearing on the right-hand-side of (\ref{TeslerK}) satisfies commutation properties analogous to those of $\mathsf{V}$.
We will need to introduce more notation to discuss this in detail.
The series $\mathbb{E}_\lambda^{(k)}$ also has a natural interpretation; the following is proven similarly to Proposition \ref{DeltaExpand} and Corollary \ref{DeltaFunc}.

\begin{prop}\label{ShiftDelta}
For any $r$-core $\alpha$, we have 
\[
\mathbb{E}_\lambda^{(k)}
=\sum_{\substack{\mu\\\core(\mu)=\alpha}}\frac{H_\mu[X^\bullet]H_\mu^\dagger[\sigma^k\iota D_\lambda^\bullet]}{N_\mu}
=\sum_{\substack{\mu\\\core(\mu)=\alpha}}\frac{H^\dagger_\mu[X^\bullet]H_\mu[\sigma^k\iota D_\lambda^\bullet]}{N_\mu}.
\]
Therefore, for any $f\in\Lambda_{q,t}^{\otimes r}$,
\[
\langle f, \mathbb{E}_\lambda^{(k)}\rangle_{q,t}'=f[\sigma^k\iota D_\lambda^\bullet].
\]
\end{prop}

\subsubsection{Extending the $\sigma$-action}
In our setup, we will need operators that act diagonally on $H_\lambda[\sigma^{-k}X^\bullet]$.
Obviously, we will want to twist $\Delta$ and $\Delta^\dagger$ by the plethysm $\sigma^{-k}$ on $\Lambda_{q,t}^{\otimes r}$, but recall that the action of a Delta operator also depends on the tensorand $e^\alpha$.
For $\alpha=(c_0,\ldots, c_{r-1})$, we define
\begin{equation*}
\begin{aligned}
\sigma\alpha=(c_{r-1},c_0, c_1, \ldots, c_{r-2}) && \text{ and }
\sigma(e^\alpha)= e^{\sigma\alpha}.
\end{aligned}
\end{equation*}
Recall from \ref{Reverse} that we have defined an action of $\sigma$ on partitions.
One can see that when we identify an $r$-core with its vector of charges, both actions of $\sigma$ agree.
Note as well that we have $\sigma w_0\alpha=w_0\sigma^{-1}\alpha$.

We have the following compatibility with respect to the pairing:
\begin{prop}\label{ShiftHProp}
Fix $k\in\ZZ/r\ZZ$.
The shifted elements $\{H_\lambda[\sigma^{-k}X^\bullet]\otimes e^{\sigma^{-k}\core(\lambda)}\}$ form a basis of $\Lambda_{q,t}^{\otimes r}\otimes \CC[Q]$.
With respect to the extended pairing $\langle - , - \rangle_{q,t}'$ (cf. \ref{PairCores}), they have a dual orthogonal basis given by $\{H_\lambda^\dagger[\sigma^kX^\bullet]\otimes e^{ w_0\sigma^{-k}\core(\lambda)}\}$, i.e.
\[
\left\langle H_\mu^\dagger[\sigma^kX^\bullet]\otimes e^{w_0\sigma^{-k}\core(\mu)},H_\lambda[\sigma^{-k}X^\bullet]\otimes e^{\sigma^{-k}\core(\lambda)}\right\rangle_{q,t}'=0
\]
if and only if $\lambda\not=\mu$.
\end{prop}

Finally, we have the following easy observation:
\begin{prop}\label{EShift}
$\mathbb{E}_\lambda^{(k)}$ is obtained from $\mathbb{E}_\lambda$ via the $\sigma$-action:
\[
\mathbb{E}_\lambda^{(k)}\otimes e^{\sigma^{-k}\core(\lambda)}= \mathbb{E}_\lambda[\sigma^{-k}X^\bullet]\otimes e^{\sigma^{-k}\core(\lambda)}.
\]
\end{prop}

\subsubsection{Shifted Delta operators}
In light of Proposition \ref{ShiftHProp}, for $k\in\ZZ/r\ZZ$ and $f\in\Lambda_{q,t}^{\otimes r}$, we set:
\begin{equation*}
\begin{aligned}
\Delta_{(k)}[f]& \coloneq \sigma^{-k}\Delta[f]\sigma^k,\\
\Delta_{(k)}^{\dagger}[f]& \coloneq \sigma^k\Delta^\dagger[f]\sigma^{-k}.
\end{aligned}
\end{equation*}
The following is by construction:
\begin{prop}\label{ShiftedDeltaOp}
For $k\in\ZZ/r\ZZ$ and $f\in\Lambda_{q,t}^{\otimes r}$, we have:
\begin{equation*}
\begin{aligned}
\Delta_{(k)}[f]\left( H_\lambda[\sigma^{-k}X^\bullet]\otimes e^{\sigma^{-k}\core(\lambda)} \right)
&= f[-D_\lambda^\bullet]H_\lambda[\sigma^{-k}X^\bullet]\otimes e^{\sigma^{-k}\core(\lambda)},\\
\Delta_{(k)}^\dagger[f]\left( H_\lambda^\dagger[\sigma^{k}X^\bullet]\otimes e^{w_0\sigma^{-k}\core(\lambda)} \right)
&= f[-D_\lambda^\bullet]H_\lambda[\sigma^{k}X^\bullet]^\dagger\otimes e^{w_0\sigma^{-k}\core(\lambda)}.
\end{aligned}
\end{equation*}
\end{prop}

\begin{lem}\label{NablaK}
The operator obtained from $F_{0,n}$ remains unchanged:
\[
\Delta\left[ e_{F,n}^{(0)} \right]=\Delta_{(k)}\left[ e_{F,n}^{(0)} \right].
\]
In particular, $\sigma^{-k}\nabla_{\core(\lambda)}\sigma^k=\nabla_{\sigma^{-k}\core(\lambda)}$
\end{lem}

\begin{proof}
The components of (\ref{F0Calc}) interact with $\sigma$ as follows:
\begin{equation*}
\begin{aligned}
\sigma^{-1}\Omega[X^{(i)}z]\sigma&= \Omega[X^{(i-1)}],\\
\sigma^{-1}\TT[X^{(i)}z]\sigma&= \TT[X^{(i-1)}],\\
\sigma^{-1}z^{H_{i,0}}\sigma&= z^{H_{i-1,0}}.
\end{aligned}
\end{equation*}
Applying this and then making the change of variables $z_{i,a}\mapsto z_{i-k,a}$ leaves (\ref{F0Calc}) unchanged.
\end{proof}

\subsubsection{The map $\mathsf{V}^{(k)}$}
Let us set $\mathsf{V}^{(k)} \coloneq \sigma^{-k}\mathsf{V}\sigma^k$.
By Proposition \ref{EShift}, $\mathrm{span}\{\mathbb{E}_\lambda^{(k)}\otimes e^{\sigma^{-k}\core(\lambda)}\}$ is closed under the action of $\Delta^\dagger_{(k)}[f]$ for any $f\in\Lambda_{q,t}^{\otimes r}$.

\begin{prop}\label{VKProp}
The map $\mathsf{V}^{(k)}:\Lambda_{q,t}^{\otimes r}\rightarrow\mathrm{span}\{\mathbb{E}_\lambda^{(k)}\otimes e^{\sigma^{-k}\core(\lambda)}\}$ satisfies the following:
\begin{itemize}
\item $\mathsf{V}^{(k)}(fg\otimes e^{\alpha})=\Delta_{(-k)}^\dagger\left[ f[-\iota\sigma^k X^\bullet] \right]\left(\mathsf{V}^{(k)}(g\otimes e^{\alpha})\right)$,
\item $\mathsf{V}^{(k)}(1\otimes e^{\sigma^{-k}\alpha})=\mathbb{E}_{\alpha}^{(k)}\otimes e^{\sigma^{-k}\alpha}$.
\end{itemize}
\end{prop}

\begin{thm}\label{VMacK}
For $k\in\ZZ/r\ZZ$, we have
\[
\mathsf{V}^{(k)}\left(H_\lambda[\sigma^{-k}X^\bullet]\otimes e^{\sigma^{-k}\core(\lambda)}\right)=H_\lambda[\iota D_{w_0\core(\lambda)}^\bullet]
\mathbb{E}_\lambda^{(k)}\otimes e^{\sigma^{-k}\core(\lambda)}.
\]
\end{thm}
Note that this simply follows from the fact that $V^{(k)}$ is attained by conjugating $V$ by $\sigma^k$, so then
\[
\sigma^k V \sigma^{-k} \left( H_{\lambda}[\sigma^{-k} X^{\bullet}] \otimes e^{\sigma^{-k}\core(\lambda)} \right)
=
\sigma^k V \left(   H_{\lambda}[\sigma  X^{\bullet}] \otimes e^{ \core(\lambda)} \right).
\]

\begin{thm}\label{ShiftedTes}
The map $\mathsf{V}^{(k)}$ is given by
\[
\mathsf{V}^{(k)}=
\nabla\Omega\left[ \frac{X^{(-k)}}{(1-q\sigma^{-1})(t\sigma-1)} \right]\TT\left[ X^{(-k)} \right]\nabla.
\]
\end{thm}

Note that combining Corollary \ref{NablaEv} and Lemma \ref{NablaK}, we have
\[
\nabla\left( \frac{H_\lambda[\sigma^{-k}X^\bullet]}{H_\lambda\left[ \iota D_{w_0\core(\lambda)}^\bullet \right]}\otimes e^{\sigma^{-k}\core(\lambda)} \right)=
H_\lambda[\sigma^{-k}X^\bullet]\otimes e^{\sigma^{-k}\core(\lambda)}.
\]

\begin{rem}
We note here that, if one tries to deduce a reciprocity formula as we did in \ref{MapV}, one will simply rederive Corollary \ref{0Rec}.
\end{rem}

\subsection{Star versions}
We present a version of our results involving inverted characters.

\begin{thm}
For $k\in\ZZ/r\ZZ$, let
\[
{}^*\mathbb{E}_\lambda^{(k)} \coloneq 
\Omega\left[ \sum_{i\in\ZZ/r\ZZ}X^{(-i)}\left( \frac{-D_\lambda^\bullet}{(1-q)(t-1)} \right)^{(i+k)}_* \right]
\]
We then have
\begin{align}
\nonumber
{}^*\mathbb{E}_\lambda^{(k)}&\otimes e^{w_0\sigma^{-k}\core(\lambda)}\\
\label{Tesler*}
&=\nabla^{-1}\Omega\left[ \frac{-X^{(k)}}{(1-q^{-1}\sigma)(t^{-1}\sigma^{-1}-1)} \right]\TT[-X^{(k)}]\nabla^{-1}\left( \frac{H_\lambda^\dagger[\sigma^k X^\bullet]}{H_\lambda^\dagger\left[-(D_{w_0\core(\lambda)}^\bullet)_*\right]}\otimes e^{w_0\sigma^{-k}\core(\lambda)}  \right).
\end{align}
\end{thm}

\begin{proof}
We apply $\downarrow$ to Theorem \ref{ShiftedFourier}.
By Proposition \ref{DownArrowProp}, we get:
\begin{align*}
&{}^*\mathbb{E}_\lambda^{(k)}\otimes e^{w_0\sigma^{-k}\core(\lambda)}\\
&= \downarrow \mathbb{E}_\lambda^{(k)}\otimes e^{\sigma^{-k}\core(\lambda)}\\
&= \downarrow\nabla\Omega\left[ \frac{X^{(-k)}}{(1-q\sigma^{-1})(t\sigma-1)} \right]\TT\left[ X^{(-k)} \right]\nabla
\left(\frac{H_\lambda[\sigma^{-k}X^\bullet]}{H_\lambda[\iota D_{w_0\core(\lambda)}  ]}\otimes e^{\sigma^{-k}\core(\lambda)}\right)\\
&= \nabla^{-1}\Omega\left[ \frac{-X^{(k)}}{(1-q^{-1}\sigma)(t^{-1}\sigma^{-1}-1)} \right]\TT[-X^{(k)}]\nabla^{-1}\left( \frac{H_\lambda^\dagger[\sigma^k X^\bullet]}{H_\lambda\left[\iota (D_{w_0\core(\lambda)}^\bullet)_*; q^{-1},t^{-1}\right]} \otimes e^{w_0\sigma^{-k}\core(\lambda)}\right).
\end{align*}
The result follows from the fact that
\[
H_\lambda\left[\iota (D_{w_0\core(\lambda)}^\bullet)_*; q^{-1},t^{-1}\right]
= H_\lambda^\dagger\left[- (D_{w_0\core(\lambda)}^\bullet)_*\right].
\qedhere
\]
\end{proof}

As before, ${}^*\mathbb{E}_\lambda^{(k)}$ has a natural interpretation:

\begin{prop}
For any core $\alpha$,
\[
{}^*\mathbb{E}_\lambda^{(k)}=\sum_{\substack{\mu\\\core(\mu)=\alpha}}\frac{H_\mu[X^\bullet]H_\mu^\dagger\left[-\sigma^k qt(D_\lambda^\bullet)_* \right]}{N_\mu}
=\sum_{\substack{\mu\\\core(\mu)=\alpha}}\frac{H_\mu^\dagger[X^\bullet] H_\mu\left[-\sigma^k qt(D_\lambda^\bullet)_* \right]}{N_\mu}.
\]
Therefore, for any $f\in\Lambda_{q,t}^{\otimes r}$, 
\[
\langle f, {}^*\mathbb{E}_\lambda^{(k)}\rangle_{q,t}'=
f\left[-\sigma^k qt(D_\lambda^\bullet)_*  \right].
\]
\end{prop}

\begin{proof}
We obtain $\mathbb{E}_\lambda^*$ from $\Omega_{q,t}\left[ X^{-\bullet}Y^\bullet \right]$ by (1) applying $\sigma^k$ to $Y^{(i)}$ and then (2) evaluating $Y^{(i)}\mapsto -qt(D_\lambda^{(i)})_*$:
\begin{align*}
\Omega_{q,t}\left[ X^{-\bullet}Y^\bullet \right]
&\overset{(1)}{\mapsto}\Omega\left[ \sum_{i\in\ZZ/r\ZZ} X^{(-i)}\left( \frac{Y^{(i+k)}}{(1-q\sigma^{-1})(t\sigma-1)} \right) \right]\\
&= \Omega\left[ \sum_{i\in\ZZ/r\ZZ} X^{(-i)}\left( \frac{q^{-1}t^{-1}Y^{(i+k)}}{(1-q^{-1}\sigma)(t^{-1}\sigma^{-1}-1)} \right) \right]\\
&\overset{(2)}{\mapsto} \Omega\left[ \sum_{i\in\ZZ/r\ZZ}X^{(-i)}\left( \frac{-D_\lambda^\bullet}{(1-q)(t-1)} \right)^{(i+k)}_* \right].\qedhere
\end{align*}
\end{proof}

\subsubsection{Delta-star operators}
Next, we consider how the operator in (\ref{Tesler*}) commutes with multiplication operators.
To that end, for any $f\in\Lambda_{q,t}^{\otimes r}$, we define
\begin{equation*}
\begin{aligned}
{}^*\Delta_{(k)}[f] \coloneq \downarrow \Delta_{(k)}^\dagger\left[\downarrow f[\iota X^\bullet]\right]\downarrow && \text{ and } &&
{}^*\Delta_{(k)}^\dagger[f] \coloneq \downarrow\Delta_{(k)}\left[\downarrow f[\iota X^\bullet]\right]\downarrow.
\end{aligned}
\end{equation*}
Combining Propositions \ref{DownArrowProp} and \ref{ShiftedDeltaOp} yields:
\begin{prop}
For any $f\in\Lambda_{q,t}^{\otimes r}$,
\begin{align*}
{}^*\Delta_{(k)}[f]\left( H_\lambda[\sigma^{-k}X^\bullet]\otimes e^{\sigma^{-k}\core(\lambda)} \right)
&= f\left[ (D_\lambda^\bullet)_* \right]H_\lambda[\sigma^{-k}X^\bullet]\otimes e^{\sigma^{-k}\core(\lambda)}\\
{}^*\Delta_{(k)}^\dagger[f]\left( H_\lambda^\dagger[\sigma^{k}X^\bullet]\otimes e^{w_0\sigma^{-k}\core(\lambda)} \right)
&= f\left[ (D_\lambda^\bullet)_* \right]H_\lambda^\dagger[\sigma^{k}X^\bullet]\otimes e^{w_0\sigma^{-k}\core(\lambda)}.
\end{align*}
\end{prop}

\subsubsection{The map $\mathsf{V}_*^{(k)}$}
Setting $\mathsf{V}_*^{(k)}=\downarrow\mathsf{V}^{(k)}\downarrow$, we have
\[
\mathsf{V}_*^{(k)}=\nabla^{-1}\Omega\left[ \frac{-X^{(k)}}{(1-q^{-1}\sigma)(t^{-1}\sigma^{-1}-1)} \right]\TT[-X^{(k)}]\nabla^{-1}.
\]
\begin{prop}
$\mathsf{V}_*^{(k)}$ satisfies the following:
\begin{itemize}
\item $\mathsf{V}_*^{(k)}(fg\otimes e^{\alpha})={}^*\Delta_{(-k)}\left[ f[-\sigma^{-k} X^\bullet] \right]\left(\mathsf{V}_*^{(k)}(g\otimes e^{\alpha})\right)$,
\item $\mathsf{V}_*^{(k)}(1\otimes e^{w_0\sigma^{-k}\alpha})={}^*\mathbb{E}_{\alpha}^{(k)}\otimes e^{w_0\sigma^{-k}\alpha}$.
\end{itemize}
\end{prop}

\begin{proof}
Only the first condition is a bit tricky.
We have
\begin{align*}
\mathsf{V}_*^{(k)}f&= \downarrow\mathsf{V}^{(k)}(\downarrow f)\downarrow\\
&= \downarrow \Delta_{(-k)}^\dagger\left[(\downarrow f)[-\iota\sigma^k X^\bullet]\right]\mathsf{V}^{(k)}\downarrow\\
&= \downarrow \Delta_{(-k)}^\dagger\left[(\downarrow f)[-\iota\sigma^k X^\bullet]\right]\downarrow\mathsf{V}^{(k)}_*.
\end{align*}
Now, note that
\[
(\downarrow f)[-\iota\sigma^k X^\bullet]=f[\iota\sigma^{k}\iota X^\bullet; q^{-1},t^{-1}]=f[\sigma^{-k}X^\bullet; q^{-1}, t^{-1}]
=\downarrow\left(f[-\iota\sigma^{-k}X^\bullet]\right).\qedhere
\]
\end{proof}

\section{Consequences}\label{Consequence}
The dependence of $\mathsf{V}^{(k)}$ on the core tensorand $e^\alpha$ only plays a role when we apply $\nabla$.
In what follows, we will work with $\Lambda_{q,t}^{\otimes r}$, instead of $\Lambda_{q,t}^{\otimes r}\otimes \CC[Q]$, and encode the core-dependence into $\nabla_\alpha$.

\subsection{Reciprocal identities}
We have already seen one instance of reciprocity (Corollary \ref{0Rec}).
Combining Theorem \ref{ShiftedFourier} with the pairing $\langle -, - \rangle_{q,t}'$, we obtain other versions.

\subsubsection{A Macdonald--Koornwinder duality}
We will use a variable $u$ to record degrees. 
We note that with respect to $\langle -, -\rangle_{q,t}'$, if we insert a grading variable $u$, we have
\[\langle f[uX^{\bullet}], \TT[A X^{\bullet}] g[X^{\bullet}] \rangle_{q,t}'
=
\langle f[X^{\bullet}], \TT[A X^{\bullet} u^{-1}]g[uX^{\bullet}] \rangle_{q,t}'
\]

We have the following strong form of reciprocity:
\begin{thm}\label{MKThm}
For $k\in\ZZ/r\ZZ$ and partitions $\lambda,\mu$ with $\core(\mu)=w_0\sigma^{-k}\core(\lambda)$, we have
\begin{equation}
\begin{aligned}
\frac{H_\lambda\left[1+u \sigma^k\iota D_{\mu}^\bullet \right]}{\displaystyle\prod_{\substack{\square\in \lambda\backslash\core(\lambda)\\ \bar{c}_\square =k}}\left( 1-u\chi_\square \right)}
&= 
\frac{H_{\mu}\left[1+u \sigma^k\iota D_{\lambda}^\bullet \right]}{\displaystyle\prod_{\substack{\square\in \mu\backslash\core(\mu)\\ \bar{c}_\square =k}}\left( 1-u\chi_\square \right)}.
\end{aligned}
\label{MKDual}
\end{equation}
\end{thm}

\begin{proof}
By Proposition \ref{ShiftDelta},
\begin{align}
\label{MKDual1}
\Omega\left[ -u\left(\frac{D_\lambda^\bullet}{(1-q)(1-t)} \right)^{(k)} \right]&H_{\mu}\left[ 1+u\sigma^k\iota D_\lambda^\bullet \right]\\
\label{MKDual2}
&= 
\left\langle
\Omega\left[ \frac{uX^{(0)}}{(1-q\sigma^{-1})(t\sigma-1)} \right]
\TT\left[X^{(0)}u^{-1}\right]H_{\mu}[u X^\bullet],
\mathbb{E}_\lambda^{(k)}
\right\rangle_{q,t}'.
\end{align}
By Theorem \ref{ShiftedFourier}, we can write
\[
\mathbb{E}_\lambda^{(k)}=
\nabla_{\sigma^{-k}\core(\lambda)}\Omega\left[ \frac{X^{(-k)}}{(1-q\sigma^{-1})(t\sigma-1)} \right]\TT\left[X^{(-k)}\right]H_\lambda[\sigma^{-k}X^\bullet].
\]
Putting this into (\ref{MKDual2}), we use Corollary \ref{NablaAdjoint} to move $\nabla_{\sigma^{-k}\core(\lambda)}$ to the other side, where it becomes $\nabla_{w_0\sigma^{-k}\core(\lambda)}$.
Moving the variable $u$ to the other side, (\ref{MKDual2}) is equal to
\begin{align}
\nonumber
&\left\langle
\mathbb{E}_{\mu},
\Omega\left[ \frac{u  X^{(-k)}}{(1-q\sigma^{-1})(t\sigma-1)} \right]\TT\left[X^{(-k)}u^{-1} \right]H_\lambda[u\sigma^{-k}X^\bullet]
\right\rangle_{q,t}'\\
\nonumber
&=\left\langle
\mathbb{E}_{\mu}[\sigma^{-k}X^\bullet], 
\Omega\left[ \frac{uX^{(0)}}{(1-q\sigma^{-1})(t\sigma-1)} \right]
\TT\left[X^{(0)}u^{-1}\right]
H_\lambda[u X^\bullet]
\right\rangle_{q,t}'\\ 
\nonumber
&=\left\langle
\mathbb{E}_{\mu}^{(k)},
\Omega\left[ \frac{u X^{(0)}}{(1-q\sigma^{-1})(t\sigma-1)} \right] H_\lambda[1+u X^\bullet]
\right\rangle_{q,t}'\\
\label{MKDual3}
&= \Omega\left[ -u\left( \frac{D_{\mu}^\bullet}{(1-q)(1-t)} \right)^{(k)} \right]H_\lambda[1+u\sigma^kD_{\mu}^\bullet].
\end{align}
The $\Omega$-terms in (\ref{MKDual1}) (\ref{MKDual3}) are equal to
\begin{equation}
\begin{aligned}
\Omega\left[ -u\left(\frac{D_\lambda^\bullet}{(1-q)(1-t)} \right)^{(k)} \right]
&=\Omega\left[ -uB_\lambda^{(k)}\right]\Omega\left[u\left( \frac{1}{(1-q)(1-t)} \right)^{(k)} \right] \text{ and }\\
\Omega\left[ -u\left(\frac{D_{\mu}^\bullet}{(1-q)(1-t)} \right)^{(k)} \right]
&=\Omega\left[ -uB_{\mu}^{(k)}\right]\Omega\left[u\left( \frac{1}{(1-q)(1-t)} \right)^{(k)} \right].
\end{aligned}
\label{MKDual4}
\end{equation}
The common factors in (\ref{MKDual4}) will cancel.
Finally, because $\core(\mu)=w_0\sigma^{-k}\core(\lambda)$, by Proposition \ref{RevProp}, $\core(\mu)$ and $\core(\lambda)$ have the same color $k$ boxes, so their corresponding factors will cancel as well.
Equation (\ref{MKDual}) follows after rearranging the terms.
\end{proof}

\subsubsection{Evaluation formulas}
In Theorem \ref{MKThm}, considering the case where $\mu=w_0 \sigma^{-k}\core(\lambda)$ yields a concrete evaluation formula:

\begin{cor}
For $k\in\ZZ/r\ZZ$, we have:
\begin{equation}
H_\lambda\left[ 1+u\sigma^k\iota D_{w_0\sigma^{-k}\core(\lambda)}^\bullet \right]
=
\prod_{\substack{\square\in\lambda\backslash\core(\lambda)\\\bar{c}_\square=k}}(1-u\chi_\square).
\label{Eval1}
\end{equation}
\end{cor}

In the limit, as $u\rightarrow\infty$, (\ref{Eval1}) becomes:
\begin{cor}\label{NablaEvK}
For $k\in\ZZ/r\ZZ$, we have:
\[
H_\lambda\left[ \sigma^k\iota D_{w_0\sigma^{-k}\core(\lambda)}^\bullet \right]=\prod_{\substack{\square\in\lambda\backslash\core(\lambda)\\\bar{c}_\square=k}}(-\chi_\square).
\]
\end{cor}

\subsubsection{Shifted reciprocity}
Taking the limit as $u\rightarrow\infty$ in Theorem \ref{MKThm} and applying Corollary \ref{NablaEvK} to the denominators, we obtain a shifted analogue of Corollary \ref{0Rec}:

\begin{cor}
For $k\in\ZZ/r\ZZ$ and partitions $\lambda,\mu$ with $\core(\mu)=w_0\sigma^{-k}\core(\lambda)$, we have:
\begin{equation*}
\frac{H_\lambda\left[\sigma^k\iota D_{\mu}^\bullet \right]}{H_\lambda\left[ \sigma^k\iota D_{w_0\sigma^{-k}\core(\lambda)}^\bullet\right]}
= 
\frac{H_{\mu}\left[ \sigma^k\iota D_{\lambda}^\bullet \right]}{H_{\mu}\left[ \sigma^k\iota D_{w_0\sigma^{-k}\core(\mu)}^\bullet \right]}.
\end{equation*}
\end{cor}

\subsubsection{An expansion of $e_n$}
For $r>1$, the single box partition $(1)$ is an $r$-core.
Here, we will derive a curious expansion of $e_n[X^{(k)}]$ in terms of $\{H_\lambda^\dagger\,|\, \core(\lambda)=w_0(1)\}$.
To make things concrete, note that $w_0(1)=(r^1 1^{r-1})$.

\begin{prop}
    For $r\ge 3$, we have:
    \begin{equation}
     e_n[X^{(k)}]= (-1)^n\sum_{\substack{\lambda\\ \core(\lambda)=w_0(1)\\ |\quot(\lambda)|=n}}
    \frac{H_\lambda^\dagger[\sigma^k X^\bullet]}{N_\lambda}
    \prod_{\substack{\square\in\lambda\backslash\core(\lambda)\\\bar{c}_\square=0}}(1- \chi_\square).       
    \label{EnExpand}
    \end{equation}
\end{prop}

\begin{proof}
It suffices to consider the case $k=0$, as we can obtain the other cases by applying $\sigma^k$.
The first step is to take the limit as $u \rightarrow 1$ in (\ref{Eval1}): 
for $\lambda$ with $\core(\lambda)=w_0(1)$,
\begin{equation}
 H_\lambda\left[ 1+ \iota D_{(1)}^\bullet \right]
=
\prod_{\substack{\square\in\lambda\backslash w_0(1)\\\bar{c}_\square=0}}(1- \chi_\square).   
\label{U1}
\end{equation}
Now observe that for any partition $\mu$,
\begin{align*}
 1+u D_\mu^\bullet \Big|_{u=1} & =    \left( (1-q)(1-t)B_\mu^\bullet  \right).
\end{align*}
Equation \eqref{U1} becomes:
\begin{equation}
 H_{\lambda} \left[  \iota (1-q)(1-t) B_{(1)}^{\bullet}\right] = \prod_{\substack{\square\in\lambda\backslash w_0(1)\\\bar{c}_\square=0}}(1-\chi_\square).
 \label{U12}
\end{equation}

In the Cauchy kernel, we substitute $Y^{\bullet}$ by $\iota (1-q)(1-t) B_{(1)}^{\bullet}$ by (1) applying $ \iota$, then (2) making the substitution of $Y^{(i)}$ at $((1-q)(1-t)B_{(1)}^\bullet)^{(i)}$:
\begin{align*}
\Omega_{q,t}[X^{-\bullet}Y^\bullet]
&= \Omega\left[ X^{(0)}\left( \frac{Y^{(0)}}{(1-q\sigma^{-1})(t\sigma-1)}\right)+\cdots +X^{(r-1)}\left( \frac{Y^{(1)}}{(1-q\sigma^{-1})(t\sigma-1)} \right) \right]\\
&\overset{(1)}{\mapsto}\Omega\left[ X^{(0)}\left( \frac{Y^{(0)}}{(1-q\sigma)(t\sigma^{-1}-1)}\right)+\cdots +X^{(r-1)}\left( \frac{Y^{(r-1)}}{(1-q\sigma)(t\sigma^{-1}-1)} \right) \right]\\
&\overset{(2)}{\mapsto} \Omega\left[ X^{(0)}\left( \frac{(1-q)(1-t)B_{(1)}^\bullet}{(1-q)(t-1)}\right)^{(0)}+\cdots +X^{(r-1)}\left( \frac{(1-q)(1-t)B_{(1)}^\bullet}{(1-q)(t-1)} \right)^{(r-1)} \right] 
\\
& = \Omega\left[ -X^{(0)}B_{(1)}^{(0)}-\cdots -X^{(r-1)}B_{(1)}^{(r-1)} \right]\\
&=
\Omega\left[ -  X^{\bullet} B^{\bullet}_{(1)} \right]\\
&=
\sum_{n \geq 0} (-1)^n
\sum_{\vec{\lambda} \vdash n } 
e_{\vec{\lambda}}[ X^{\bullet}] m_{\vec{\lambda}}[B_{(1)}^{\bullet}]\\
&=\sum_{n\ge 0}(-1)^n e_n[X^{(0)}].
\end{align*}
On the other hand, performing the same steps on
\[
\Omega_{q,t}[X^{-\bullet}Y^\bullet]=\sum_{\substack{\lambda\\ \core(\lambda)=w_0(1)}} \frac{H_\lambda^\dagger[X^\bullet]H_\lambda[Y^\bullet]}{N_\lambda}
\]
yields
\[
\sum_{\substack{\lambda\\ \core(\lambda)=w_0(1)}}
\frac{H_\lambda^\dagger[X^\bullet]}{N_\lambda}\prod_{\substack{\square\in\lambda\backslash w_0(1)\\\bar{c}_\square=0}}(1-\chi_\square)
\]
by \eqref{U12}.
\end{proof}

\begin{rem}
Observe that we can take any eigenoperator for $\{H_\lambda^\dagger[\sigma^kX^\bullet]\,|\, \core(\lambda)=w_0(1)\}$ and easily apply it to the right-hand-side of \eqref{EnExpand}.
By Lemma \ref{NablaK}, one such operator is provided by $\nabla_{\sigma^kw_0(1)}$.
A more interesting operator would be some version of the ``geometric nabla'' defined by Haiman (cf. \cite{OSWreath}).
It would also be interesting to remove the restriction on $\core(\lambda)$.

In the $r=1$ case, $\nabla e_n$ has a well-known expansion in terms of $\{H_\lambda\}$.
One familiar with modified Macdonald polynomials will note a big difference in \eqref{EnExpand} compared to the case when $r=1$. 
This stems from the fact that when $r=1$, $H_{(1)} = e_1$, but for $r >1$, $(1)$ is a core, meaning that $H_{(1)} = 1$. 

This forces, for instance, Theorem \ref{MKThm} to specialize differently when $r=1$ and $\mu = 1$:
\[
H_\lambda [1 - u + u (1-q)(1-t)] 
= 
(1 - u +u (1-q)(1-t)B_\lambda ) \prod_{ \square\in \lambda\backslash (1)  }\left( 1-u\chi_\square \right)
\]
At $u=1$, this then gives
\[
H_\lambda[ (1-q)(1-t) ] = (1-q)(1-t) B_\lambda \prod_{\square \in \lambda \backslash (1) } (1-\chi_\square),
\]
and therefore
\[
e_n[X] = (-1)^n\sum_{\lambda \vdash n} \frac{H_\lambda[X] (1-q)(1-t) B_\lambda }{N_\lambda}\prod_{\square \in \lambda \backslash (1)}( 1-\chi_\square).
\]   
\end{rem}

\subsection{Interpolation}
For an $r$-core $\alpha$, and any $f\in\Lambda_{q,t}^{\otimes r}$, define
\begin{equation}
f^{\mathfrak{F}}_{\alpha} \coloneq \nabla_{w_0\alpha}^{-1}\TT[ -X^{(0)}](f).
\label{IntF}
\end{equation}
This is the analogue of $\mathbf{\Pi}'_f$ from \cite{GHT}.
The following is an analogue of (I.16) from \textit{loc. cit.}

\begin{lem}
Let $\lambda$ be a partition with $\core(\lambda)=\alpha$.
For $f\in\Lambda_{q,t}^{\otimes r}$, we have
\begin{align}
\label{IntLem0}
f_{\alpha}^{\mathfrak{F}}\left[ \iota D_\lambda^\bullet \right]
&= \left\langle f,H_\lambda[1+X^\bullet]\right\rangle_{q,t}'\\
&=\left\langle \Omega\left[ \frac{X^{(0)}}{(1-q\sigma^{-1})(t\sigma-1)} \right] f, H_\lambda[X^\bullet]\right\rangle_{q,t}'.
\label{IntLem}
\end{align}
\end{lem}

\begin{proof}
We begin with Theorem \ref{ShiftedFourier}:
\begin{align}
\nonumber
f_{\alpha}^{\mathfrak{F}}\left[ \iota D_\lambda^\bullet \right]
&= \left\langle \nabla_{w_0\alpha}^{-1}\TT[ -X^{(0)}](f),\mathbb{E}_\lambda\right\rangle_{q,t}'\\
\label{IntLem1}
&= \left\langle \nabla_{w_0\alpha}^{-1}\TT[ -X^{(0)}](f),\nabla_{\alpha}\Omega\left[ \frac{X^{(0)}}{(1-q\sigma^{-1})(t\sigma-1)} \right]\TT[X^{(0)}]H_\lambda[X^\bullet]\right\rangle_{q,t}'.
\end{align}
Applying adjoints, (\ref{IntLem1}) becomes
\begin{align}
\label{IntLem2}
\left\langle f,\TT[X^{(0)}]H_\lambda[X^\bullet]\right\rangle_{q,t}'&= \left\langle f,H_\lambda[1+X^\bullet]\right\rangle_{q,t}'
\end{align}
On the other hand, applying adjunction to the left-hand-side of (\ref{IntLem2}) yields (\ref{IntLem}).
\end{proof}

\subsubsection{Fourier pairing}\label{FourPair}
To motivate the definition (\ref{IntF}), we present an identity generalizing a result in \cite[5.5]{CNO}.
In \textit{loc. cit.}, the authors show that their identity is an infinite-variable analogue of \cite[Theorem 1.2]{ChereMM}.

\begin{thm}\label{CMMId}
For an $r$-core $\alpha$, consider the pairing $\langle -, -\rangle^{\mathfrak{F}}_{\alpha}$ on $\Lambda_{q,t}^{\otimes r}$ given by
\[
\langle f, g,\rangle^{\mathfrak{F}}_{\alpha} \coloneq \left\langle \TT[ X^{(0)} ]\nabla_{w_0\alpha} (f), \TT[ X^{(0)} ]\nabla_{\alpha}(g)\right\rangle_{q,t}'.
\]
Then for partitions $\lambda$ and $\mu$ with $\core(\lambda)=\alpha$ and $\core(\mu)=w_0\core(\lambda)$, we have:
\begin{align}
\label{CMM1}
\left\langle H_{\mu}[X^\bullet], H_\lambda[X^\bullet]\right\rangle^{\mathfrak{F}}_{\alpha}
&=H_{\mu}\left[\iota D_{\alpha}^\bullet\right]H_\lambda\left[\iota D_{\mu}^\bullet\right]\\
&=H_\lambda\left[\iota D_{w_0\alpha}^\bullet\right]H_{\mu}\left[\iota D_{\lambda}^\bullet\right].
\label{CMM2}
\end{align}
\end{thm}

\begin{rem}
Observe that
\[
\left\langle f_{\alpha}^{\mathfrak{F}}, g_{w_0\alpha}^{\mathfrak{F}}\right\rangle_{\alpha}^{\mathfrak{F}}=\langle f, g\rangle_{q,t}'.
\]
\end{rem}

\begin{proof}
Taking adjoints, we have:
\begin{align*}
\left\langle H_{\mu}[X^\bullet], H_\lambda[X^\bullet]\right\rangle^{\mathfrak{F}}_{\alpha}
&=
\left\langle \TT[ X^{(0)} ]\nabla_{w_0\alpha} H_{\mu}[X^\bullet], \TT[ X^{(0)} ]\nabla_{\alpha}H_\lambda[X^\bullet]\right\rangle_{q,t}'\\
&= 
\left\langle H_{\mu}[X^\bullet], \nabla_\alpha\Omega\left[ \frac{X^{(0)}}{(1-q\sigma^{-1})(t\sigma-1)} \right]\TT[ X^{(0)} ]\nabla_{\alpha}H_\lambda[X^\bullet]\right\rangle_{q,t}'\\
&=
H_\lambda[\iota D_{w_0\alpha}^\bullet]
\left\langle H_{\mu}[ X^\bullet], \mathbb{E}_\lambda\right\rangle_{q,t}' \\
&=H_\lambda[\iota D_{w_0\alpha}^\bullet] H_{\mu}\left[ \iota D_{\lambda}^\bullet \right] 
\end{align*}
This yields (\ref{CMM2}).
The proof of (\ref{CMM1}) is similar.
\end{proof}

\subsubsection{Skew functions}
We take a little interlude to introduce skew wreath Macdonald polynomials.
Because we have $H_\lambda$ and $H_\lambda^\dagger$, there are actually four flavors of wreath Macdonald skew functions; we will only focus on one.
For partitions $\lambda,\mu,\nu$ with equal core $\alpha$, define the coefficients $c_{\mu^\dagger\nu^\dagger}^{\lambda^\dagger}$ by
\begin{align*}
H_\mu^\dagger H_\nu^\dagger=\sum_{\core(\lambda)  = \alpha} c_{\mu^\dagger\nu^\dagger}^{\lambda^\dagger}H_\lambda^\dagger.
\end{align*}
By Lemma \ref{P1Supp}, these coefficients are nonzero only if $\lambda\supset\mu,\nu$.
Note that on partitions, the embellishment $(-)^\dagger$ is only being used to mark which version of wreath Macdonald the partition is attached to.

\begin{defn}
The \textit{wreath Macdonald skew functions} $H_{\lambda/\mu^\dagger}$ are defined to be
\begin{align*}
H_{\lambda/\mu^\dagger}& \coloneq N_\lambda\sum_{\nu} \frac{c_{\mu^\dagger\nu^\dagger}^{\lambda^\dagger}}{N_\nu}H_\nu.
\end{align*}
\end{defn}

We have that $H_{\lambda/\mu^\dagger}\not=0$ only if $\lambda\supset\mu$ and $\core(\lambda)=\core(\mu)$.
By construction, $H_{\lambda/\mu\dagger}$ satisfies
\[
\langle H_\mu^\dagger f, H_\lambda \rangle_{q,t}'=\langle f,  H_{\lambda/\mu^\dagger} \rangle_{q,t}',
\]
which can be checked for $f = H_\nu^\dagger$.
The following is an analogue of \cite[1.53]{GHT}:
\begin{lem}\label{SkewLem}
Given two alphabets $X^\bullet$ and $Y^\bullet$, we have
\[
H_\lambda\left[ X^\bullet+Y^\bullet \right]=\sum_{\mu}\frac{H_\mu[X^\bullet]H_{\lambda/\mu^\dagger}[Y^\bullet]}{N_\mu}.
\]
\end{lem}

\begin{proof}
Let $\alpha=\core(\lambda)$.
As in \textit{loc. cit.}, the proof proceeds by considering
\begin{equation}
\Omega_{q,t}\left[ \left( X+Y \right)^{-\bullet}Z^\bullet \right]=\Omega_{q,t}\left[ X^{-\bullet}Z^\bullet \right]\Omega_{q,t}\left[ Y^{-\bullet}Z^\bullet \right].
\label{CauchyXYZ}
\end{equation}
Expanding both sides of (\ref{CauchyXYZ}) using Proposition \ref{CauchyProp}, we have
\begin{align*}
\sum_{\substack{\lambda\\\core(\lambda)=\alpha}}\frac{H_\lambda[X^\bullet+Y^\bullet]H_\lambda^\dagger[Z^\bullet]}{N_\lambda}
&= \sum_{\substack{\mu,\nu\\\core(\mu)=\core(\nu)=\alpha}}
\frac{H_{\mu}[X^\bullet]H_{\nu}[Y^\bullet]H_{\mu}^\dagger[Z^\bullet]H_\nu^\dagger[Z^\bullet]}{N_\mu N_\nu}\\
&= \sum_{\substack{\lambda,\mu,\nu\\\core(\lambda)=\core(\mu)=\core(\nu)=\alpha}}
\frac{c_{\mu^\dagger\nu^\dagger}^{\lambda^\dagger}H_{\mu}[X^\bullet]H_{\nu}[Y^\bullet]H_\lambda^\dagger[Z^\bullet]}{N_\mu N_\nu}\\
&= \sum_{\substack{\lambda,\mu\\\core(\lambda)=\core(\mu)=\alpha}}
\frac{H_{\mu}[X^\bullet]H_{\lambda/\mu^\dagger}[Y^\bullet]H_\lambda^\dagger[Z^\bullet]}{N_\mu N_\lambda}.
\end{align*}
The result follows from applying $\langle -, H_\lambda[Z^\bullet]\rangle_{q,t}'$ to both sides.
\end{proof}

\subsubsection{Wreath Macdonald interpolation polynomials}
Finally, we present wreath analogues of (modified) Macdonald interpolation polynomials \cite{SahiInt,KnopInt}.
Our approach and results are as in \cite[Theorem I.4]{GHT}.

\begin{defn}\label{InterDef}
The \textit{wreath Macdonald interpolation polynomial} $\mathfrak{H}_\mu\in\Lambda_{q,t}^{\otimes r}$ is defined to satisfy the following conditions: 
\begin{itemize}
\item $\mathfrak{H}_\mu$ has highest degree $|\quot(\mu)|$;
\item $\mathfrak{H}_\mu[\iota D_\lambda^\bullet]=0$ if $\core(\lambda)=\core(\mu)$, $|\lambda|\le|\mu|$, and $\lambda\not=\mu$;
\item $\mathfrak{H}_\mu[\iota D_\mu^\bullet]=1$.
\end{itemize}
\end{defn}

\begin{lem}
If $\mathfrak{H}_\mu$ exists, then it is unique.
\end{lem}

\begin{proof}
Set $\alpha=\core(\mu)$.
Let $\mathfrak{H}_\mu$ and $\mathfrak{H}_\mu'$ both satisfy the conditions in Definition \ref{InterDef}.
We then have that $\mathfrak{D}_\mu \coloneq \mathfrak{H}_\mu-\mathfrak{H}_\mu'$ satisfies $\mathfrak{D}_\mu[\iota D_\lambda^\bullet]=0$ for all $\lambda$ with $\core(\lambda)=\alpha$ and $|\lambda|\le|\mu|$.
By Theorem \ref{CMMId}, this implies that
\[
\langle \mathfrak{D}_\mu,H_\lambda\rangle_{\alpha}^{\mathfrak{F}}=
\left\langle \TT[X^{(0)}]\nabla_{w_0\alpha}\mathfrak{D}_\mu, \TT[X^{(0)}]\nabla_\alpha H_\lambda\right\rangle_{q,t}'=
0
\]
for all such $\lambda$.
This implies that $\mathfrak{D}_\mu=0$ because:
\begin{itemize}
\item $\TT[X^{(0)}]\nabla_{w_0\alpha}\mathfrak{D}_\mu$ has highest degree at most $|\quot(\mu)|$;
\item the set 
\[
\left\{\TT[X^{(0)}]\nabla_\alpha H_\lambda\, \middle|\, \core(\lambda)=\alpha,\, |\lambda|\le|\mu|\right\}
\]
is a basis for the subspace of $\Lambda_{q,t}^{\otimes r}$ consisting of elements with degree no larger than $|\quot(\mu)|$;
\item $\TT[X^{(0)}]\nabla_{w_0\alpha}$ is invertible.\qedhere
\end{itemize}
\end{proof}

\begin{thm}
The interpolation polynomial $\mathfrak{H}_\mu$ exists and is given by 
\[
\mathfrak{H}_\mu= \frac{1}{N_\mu}\left( H_\mu^\dagger \right)_{\alpha}^{\mathfrak{F}}=\frac{1}{N_\mu}\nabla_{w_0\core(\mu)}^{-1}H_\mu^\dagger[X^\bullet-1].
\]
Moreover, it satisfies
\[
\mathfrak{H}_\mu\left[ \iota D_\lambda^\bullet \right]=\frac{H_{\lambda/\mu^\dagger}[1]}{N_\mu}
\]
for all $\lambda$ with $\core(\lambda) = \core(\mu)$. 
In particular, the evaluation above is nonzero only if $\lambda\supset\mu$.
\end{thm}

\begin{proof}
By (\ref{IntLem0}) and Lemma \ref{SkewLem}, we have
\begin{align*}
\mathfrak{H}_\mu[\iota D_\lambda^\bullet]&= \frac{1}{N_\mu}\langle H_\mu^\dagger, H_\lambda[1+X^\bullet]\rangle_{q,t}'\\
&= \frac{1}{N_\mu}\left\langle H_\mu^\dagger, \sum_{\nu} \frac{H_\nu[X^\bullet]H_{\lambda/\nu^\dagger}[1]}{N_\nu}\right\rangle_{q,t}'\\
&= \frac{1}{N_\mu}H_{\lambda/\mu^\dagger}[1].
\end{align*}
Finally, note that $H_{\mu/\mu^\dagger}=N_\mu$.
\end{proof}

\begin{rem}
Note that $\{\mathfrak{H}_\mu\, |\, \core(\mu)=\alpha\}$ is not an orthogonal basis for $\langle-,-\rangle_{\alpha,0}^{\mathfrak{F}}$.
A dual orthogonal basis is given by
\[
\mathfrak{H}_\mu^\dagger \coloneq \nabla_{\alpha}^{-1}H_\mu\left[ X^\bullet-1 \right].
\]
\end{rem}

\subsection{Kostka coefficients}
The \textit{wreath $(q,t)$-Kostka coefficients} are the coefficients $K_{\vec{\gamma},\mu}=K_{\vec{\gamma},\mu}(q,t)$ appearing in the expansion
\[
H_\mu = \sum_{\vec{\gamma},\mu} K_{\vec{\gamma},\mu} s_{\vec{\gamma}}[X^{\bullet}].
\]
For $\vec{\gamma}=(\gamma^{(0)},\gamma^{(1)},\ldots,\gamma^{(r-1)})$, let
\[
\iota\vec{\gamma}=(\gamma^{(0)},\gamma^{(r-1)},\gamma^{(r-2)},\ldots,\gamma^{(1)}).
\]
Because
\begin{equation}
s_{\vec{\gamma}}=s_{\gamma^{(0)}}[X^{(0)}]\cdots s_{\gamma^{(r-1)}}[X^{(r-1)}],
\label{SchurFactor}
\end{equation}
we have $s_{\vec{\gamma}}[\iota X^\bullet]=s_{\iota\vec{\gamma}}$.
Thus, in terms of the $\iota$-twisted Hall scalar product, we can write:
\[
K_{\vec{\gamma},\mu} = \langle s_{\iota\vec{\gamma}}, H_\lambda \rangle^*.
\]
Finally, in terms of the wreath Macdonald pairing, we have 
\begin{align}
\label{KostkaPair}
K_{\vec{\gamma},\mu} & = \left\langle s_{\iota\vec{\gamma}}\left[ \frac{X^{\bullet}}{(1-q\sigma^{-1})(t\sigma-1)}  \right], H_\mu[ X^{\bullet}] \right\rangle_{q,t}' 
\end{align}
Here, we will give a plethystic formula for $K_{\vec{\gamma}, \mu}$ analogous to \cite[Theorem I.1]{GHT}.

\subsubsection{Bernstein operators}
Recall the following creation operator for Schur functions.

\begin{prop}[\protect{\cite[Example I.5.29]{Mac}}]
Suppose that $\lambda=(\lambda_1,\lambda_2,\ldots)$ and set $\lambda_{>1}=(\lambda_2,\lambda_3,\ldots)$.
We then have
\[
s_\lambda=\left\{ z^{-\lambda_1}\Omega\left[ X z\right]s_{\lambda_{>1}}[X-z^{-1}] \right\}_0.
\]
\end{prop}

Because $s_{\vec{\gamma}}$ is factored according to color as in (\ref{SchurFactor}), we can apply these creation operators for each color to add a longest row.
We will only do this for color zero.
Taking the plethysm in (\ref{KostkaPair}) into account, we have

\begin{cor}
For $\vec{\gamma}=(\gamma^{(0)},\gamma^{(1)},\ldots, \gamma^{(r-1)})$, let $\gamma_{>1}^{(0)}=(\gamma_2^{(0)}, \gamma_3^{(0)},\ldots)$.
Then
\begin{equation}
\begin{aligned}
s_{\iota\vec{\gamma}}\left[ \frac{X^{(0)}}{(1-q\sigma^{-1})(t\sigma-1)} \right]
&=
\prod_{i=1}^{r-1}s_{\gamma^{(i)}}\left[ \frac{X^{(-i)}}{(1-q\sigma^{-1})(t\sigma-1)} \right]\\
&\quad\times
\left\{ z^{-\gamma_1^{(0)}}\Omega\left[ \frac{X^{(0)}z}{(1-q\sigma^{-1})(t\sigma-1)} \right] 
s_{\gamma^{(0)}_{>1}}\left[ \frac{X^{(0)}}{(1-q\sigma^{-1})(t\sigma-1)}-z^{-1} \right]
\right\}_0.
\end{aligned}
\label{BernCor}
\end{equation}

\end{cor}

\subsubsection{Plethystic formula}
Recall the operation in (\ref{IntF}).
Given an $r$-core $\alpha$, for $\vec{\gamma}=(\gamma^{(0)},\ldots, \gamma^{(r-1)})$, set
\[
k_{\vec{\gamma}}^\alpha \coloneq 
\left(s_{\gamma^{(0)}}\left[ \frac{X^{(0)}}{(1-q\sigma^{-1})(t\sigma-1)}-1 \right]
\prod_{i=1}^{r-1}s_{\gamma^{(i)}}\left[ \frac{X^{(-i)}}{(1-q\sigma^{-1})(t\sigma-1)} \right]\right)_\alpha^{\mathfrak{F}}.
\]
Let $\vec{\gamma}_{>1}=(\gamma^{(0)}_{>1},\gamma^{(1)},\ldots, \gamma^{(r-1)})$, i.e. $\vec{\gamma}$ minus its longest row in color zero.
If $\gamma^{(0)}=\varnothing$, then we set $\vec{\gamma}_{>1}=\vec{\gamma}$.

\begin{thm}
Let $\mu$ be a partition with $\core(\mu)=\alpha$.
Given $\vec{\gamma}$ with $|\quot(\mu)|=|\vec{\gamma}|$, we have
\[
K_{\vec{\gamma},\mu}=k_{\vec{\gamma}_{>1}}^\alpha\left[ \iota D_\mu^\bullet \right].
\]
\end{thm}

\begin{proof}
Applying (\ref{IntLem}), we have:
\begin{align}
\nonumber
&k_{\vec{\gamma}_{>1}}^\alpha[\iota D_\mu^\bullet]\\
\label{KGamma1}
&= 
\left\langle
\Omega\left[ \frac{X^{(0)}}{(1-q\sigma^{-1})(t\sigma-1)} \right] 
s_{\gamma^{(0)}_{>1}}\left[ \frac{X^{(0)}}{(1-q\sigma^{-1})(t\sigma-1)}-1 \right]
\prod_{i=1}^{r-1}s_{\gamma^{(i)}}\left[ \frac{X^{(-i)}}{(1-q\sigma^{-1})(t\sigma-1)} \right]
, H_\mu
\right\rangle_{q,t}'
\end{align}
Note that $H_\mu$ is homogeneous of degree $|\quot(\mu)|$.
The summand of the same degree in the left argument of (\ref{KGamma1}) is exactly the right-hand-side of (\ref{BernCor}).
The theorem follows from (\ref{KostkaPair}).
\end{proof}

\subsection{Global eigenfunctions}
Finally, we describe an analogue of Cherednik's \textit{global spherical function} \cite{ChereMM} (called the \textit{basic hypergeometric function} in \cite{StokC}).
Our analogy is imperfect because we work in the infinite-variable rather than the finite-variable case.
Nonetheless, we wish to initiate the study of such an object.

We will work with two alphabets of variables, $X^\bullet$ and $Y^\bullet$.
When it is unclear, we will embellish an operator with $(-)^X$ or $(-)^Y$ to specify which family of variables it is acting on.

\begin{defn}
For $k\in\ZZ/r\ZZ$ and an $r$-core $\alpha$, we define
\begin{equation}
\begin{aligned}
\mathcal{H}^{(k)}_\alpha\left[ X^\bullet,Y^\bullet \right] & \coloneq 
\left(\nabla_{w_0\sigma^{-k}\alpha}^Y\right)^{-1}\TT\left[ -Y^{(0)} \right]\Omega\left[\frac{-Y^{(0)}}{(1-q\sigma^{-1})(t\sigma-1)}  \right]\\
&\quad\times
\left(\nabla_{\sigma^{-k}\alpha}^X\right)^{-1}\TT\left[ -X^{(-k)} \right]\Omega\left[\frac{-X^{(-k)}}{(1-q\sigma^{-1})(t\sigma-1)}  \right]
\left(\nabla_{\sigma^{-k}\alpha}^X\right)^{-1}\Omega_{q,t}[X^{-\bullet}Y^\bullet].
\end{aligned}
\label{GlobalDef}
\end{equation}
\end{defn}

\begin{rem}\label{TwoShifts}
Notice that we only introduce shifts in the $X$-variables.
If we wish to shift the $Y$-variables by $\sigma^{-j}$ for $j\in\ZZ/r\ZZ$, then by Lemma \ref{NablaK}, that is just
\[
\mathcal{H}^{(j+k)}_\alpha[\sigma^j X^\bullet, \sigma^{-j}Y^\bullet].
\]
\end{rem}

\subsubsection{Symmetry}
$\mathcal{H}^{(k)}_\alpha[X^\bullet,Y^\bullet]$ exhibits an $X\leftrightarrow Y$ symmetry, but we must be careful when accounting for the cores and shifts.

\begin{thm}\label{GlobSym}
For $k\in\ZZ/r\ZZ$ and an $r$-core $\alpha$, we have
\[
\mathcal{H}_\alpha^{(k)}[X^\bullet,Y^\bullet]=\mathcal{H}_{w_0\sigma^{-k}\alpha}^{(k)}\left[ \sigma^{k}Y^\bullet, \sigma^{-k}X^\bullet \right].
\]
\end{thm}

\begin{proof}
By Corollary \ref{NablaAdjoint}, we have
\begin{equation*}
\begin{aligned}
\left(\nabla_{\sigma^{-k}\alpha}^X\right)^{-1}\Omega_{q,t}[X^{-\bullet}Y^\bullet]
&=\sum_{\substack{\lambda \\ \core(\lambda)=\sigma^{-k}\alpha}} \frac{\left(\nabla_{\sigma^{-k}\alpha}^X\right)^{-1}H_\lambda[X^\bullet]H_\lambda^\dagger[Y^\bullet]}{N_\lambda}\\
&= \sum_{\substack{\lambda \\ \core(\lambda)=\sigma^{-k}\alpha}} \frac{H_\lambda[X^\bullet]\left(\nabla_{w_0\sigma^{-k}\alpha}^Y\right)^{-1}H_\lambda^\dagger[Y^\bullet]}{N_\lambda}\\
&=\left(\nabla_{w_0\sigma^{-k}\alpha}^Y\right)^{-1}\Omega_{q,t}[X^{-\bullet}Y^\bullet] .
\end{aligned}
\label{NablaCauchy}
\end{equation*}
Thus, we also have
\begin{equation}
\begin{aligned}
\mathcal{H}^{(k)}_\alpha\left[ X^\bullet,Y^\bullet \right] & = 
\left(\nabla_{\sigma^{-k}\alpha}^X\right)^{-1}\TT\left[ -X^{(-k)} \right]\Omega\left[\frac{-X^{(-k)}}{(1-q\sigma^{-1})(t\sigma-1)}  \right]\\
&\quad\times
\left(\nabla_{w_0\sigma^{-k}\alpha}^Y\right)^{-1}\TT\left[ -Y^{(0)} \right]\Omega\left[\frac{-Y^{(0)}}{(1-q\sigma^{-1})(t\sigma-1)}  \right]
\left(\nabla_{w_0\sigma^{-k}\alpha}^Y\right)^{-1}\Omega_{q,t}[X^{-\bullet}Y^\bullet].
\end{aligned}
\label{Sym1}
\end{equation}
We use Lemma \ref{NablaK} to appropriately change the nablas when computing $\mathcal{H}_{w_0\sigma^{-k}\alpha}^{(k)}[\sigma^k Y^\bullet, \sigma^{-k}X^\bullet]$, which yields (\ref{Sym1}).
\end{proof}

\begin{rem}
Should we have chosen to define $\mathcal{H}_\alpha^{(k)}$ with a shift on the $Y$-variables as well, then by Remark {\ref{TwoShifts}}, the symmmetry becomes 
\[
\mathcal{H}_\alpha^{(j+k)}\left[\sigma^j X^\bullet, \sigma^{-j}Y^\bullet\right]
=\mathcal{H}_{w_0\sigma^{j+k}\alpha}^{(j+k)}\left[ \sigma^k Y^\bullet, \sigma^{-k}X^\bullet \right].
\]
\end{rem}

\subsubsection{Eigenfunction equations}
The following follows from applying $\sigma^{-k}$ Proposition \ref{MainPropertiesofV}.

\begin{lem}\label{DeltaSkew}
We have
\[
h_n^\perp[X^{(i-k)}]\mathsf{V}^{(k)}=\mathsf{V}^{(k)}\Delta_{(k)}\left[ h_n\left[ \frac{X^{(i)}}{(1-q\sigma)(1-t\sigma^{-1})} \right] \right].
\]
\end{lem}

Recall that our Delta operators depend on a core $\alpha$.
Let $\Delta_{(k)}[f]_\alpha$ denote the operator on $\Lambda_{q,t}^{\otimes r}$ such that
\[
\Delta_{(k)}[f]\left(g\otimes e^{\sigma^{-k}\alpha}\right)=\Delta_{(k)}[f]_\alpha(g)\otimes e^{\sigma^{-k}\alpha}.
\]

\begin{thm}\label{BispectralThm}
$\mathcal{H}_\alpha^{(k)}[X^\bullet,Y^\bullet]$ solves the bispectral problem: for any $f\in\Lambda_{q,t}^{\otimes r}$,
\begin{align}
\label{Bispectral1}
\Delta_{(k)}\left[f\right]^X_\alpha\left( \mathcal{H}_\alpha^{(k)}[X^\bullet, Y^\bullet] \right)&= \left(f\left[ -\sigma^k\iota Y^\bullet \right]\right)_{\sigma^{-k}\alpha}^\mathfrak{F}\mathcal{H}_\alpha^{(k)}[X^\bullet, Y^\bullet], \text{ and }\\
\label{Bispectral2}
\Delta_{(0)}\left[f\right]^Y_{w_0\sigma^{-k}\alpha}\left( \mathcal{H}_\alpha^{(k)}[X^\bullet, Y^\bullet] \right)&= \left(f\left[ -\iota X^\bullet \right]\right)_{w_0\sigma^{-k}\alpha}^{\mathfrak{F}}\mathcal{H}_\alpha^{(k)}[X^\bullet, Y^\bullet].
\end{align}
where $f_\alpha^{\mathfrak{F}}$ was defined in (\ref{IntF}).
\end{thm}

\begin{proof}
First consider (\ref{Bispectral1}).
Using Lemma \ref{DeltaSkew}, we have:
\begin{align}
\nonumber
&\Delta_{(k)}\left[ h_n\left[ \frac{X^{(i)}}{(1-q\sigma)(1-t\sigma^{-1})} \right] \right]^X_\alpha\left( \mathcal{H}_\alpha^{(k)}[X^\bullet, Y^\bullet] \right)\\
\label{Bi1}
&=
\left(\nabla_{w_0\sigma^{-k}\alpha}^Y\right)^{-1}\TT\left[ -Y^{(0)} \right]\Omega\left[\frac{-Y^{(0)}}{(1-q\sigma^{-1})(t\sigma-1)}  \right]\\
\label{Bi2}
&\quad\times
\left(\nabla_{\sigma^{-k}\alpha}^X\right)^{-1}\TT\left[ -X^{(-k)} \right]\Omega\left[\frac{-X^{(-k)}}{(1-q\sigma^{-1})(t\sigma-1)}  \right]
\left(\nabla_{\sigma^{-k}\alpha}^X\right)^{-1}h_n^\perp[X^{(i-k)}]\Omega_{q,t}[X^{-\bullet}Y^\bullet].
\end{align}
Notice that
\[
\TT[X^{(i-k)}u]\left(\Omega_{q,t}[X^{-\bullet}Y^\bullet]\right)=\Omega_{q,t}[X^{-\bullet}Y^\bullet]\Omega\left[ \frac{Y^{(k-i)}u}{(1-q\sigma^{-1})(t\sigma-1)} \right].
\]
Therefore, in (\ref{Bi2}),
\begin{align*}
h_n^\perp[X^{(i-k)}]\Omega_{q,t}[X^{-\bullet}Y^\bullet]
&=h_n\left[ \frac{Y^{(k-i)}}{(1-q\sigma^{-1})(t\sigma-1)} \right]\Omega_{q,t}[X^{-\bullet}Y^\bullet]\\
&= h_n\left[\sigma^k\iota\left( \frac{Y^{(i)}}{(1-q\sigma)(t\sigma^{-1}-1)}\right) \right]\Omega_{q,t}[X^{-\bullet}Y^\bullet].
\end{align*}
Commuting past (\ref{Bi1}) yields
\[
\Delta_{(k)}\left[ h_n\left[ \frac{X^{(i)}}{(1-q\sigma)(1-t\sigma^{-1})} \right] \right]^X_\alpha\left( \mathcal{H}_\alpha^{(k)}[X^\bullet, Y^\bullet] \right)
=
\left(h_n\left[\sigma^k\iota\left( \frac{Y^{(i)}}{(1-q\sigma)(t\sigma^{-1}-1)}\right) \right]\right)_{\sigma^{-k}\alpha}^{\mathfrak{F}}\mathcal{H}_\alpha^{(k)}.
\]
Because these modified $h_n$ generate $\Lambda_{q,t}^{\otimes r}$, (\ref{Bispectral1}) follows for general $f$.
The proof of (\ref{Bispectral2}) is similar, after applying (\ref{Sym1}).
\end{proof}

\begin{rem}
In the spirit of Remark \ref{TwoShifts}, we have:
\begin{align*}
\Delta_{(k)}\left[f\right]^X_\alpha\left( \mathcal{H}_\alpha^{(j+k)}[\sigma^jX^\bullet, \sigma^{-j}Y^\bullet] \right)&= \left(f\left[ -\sigma^k\iota Y^\bullet \right]\right)_{\sigma^{-k}\alpha}^\mathfrak{F}\mathcal{H}_\alpha^{(j+k)}[\sigma^jX^\bullet, \sigma^{-j}Y^\bullet]\\
\Delta_{(j)}\left[f\right]^Y_{w_0\sigma^{-k}\alpha}\left( \mathcal{H}_\alpha^{(j+k)}[\sigma^j X^\bullet, \sigma^{-j}Y^\bullet] \right)&= \left(f\left[ -\sigma^j\iota X^\bullet \right]\right)_{w_0\sigma^{-k}\alpha}^{\mathfrak{F}}\mathcal{H}_\alpha^{(j+k)}[\sigma^jX^\bullet, \sigma^{-j}Y^\bullet].
\end{align*}
\end{rem}

\begin{rem}
In the finite-variable setting of \cite{ChereMM}, the analogue of $\mathsf{V}$ is written in terms of the \textit{gaussian}, which commutes with multiplication operators.
Here, we can drop the left two nabla operators in the definition of $\mathcal{H}_\alpha^{(k)}$, and then $f$ would not be modified by $(-)_\alpha^{\mathfrak{F}}$ as in Theorem \ref{BispectralThm}. 
\end{rem}

\subsubsection{Kernel of the Fourier pairing}
Finally, let us mention a simpler variant, which is the reproducing kernel for the Fourier pairing $\langle -, - \rangle_{\alpha}^{\mathfrak{F}}$ defined in \ref{FourPair}.

\begin{defn}
For an $r$-core $\alpha$, we set
\[
\mathcal{F}_\alpha[X^\bullet,Y^\bullet] \coloneq (\nabla_\alpha^X)^{-1}\TT[-X^{(0)}](\nabla_{w_0\alpha}^Y)^{-1}\TT[-Y^{(0)}]\Omega_{q,t}[X^{-\bullet}Y^\bullet].
\]
\end{defn}

\begin{thm}
$\mathcal{F}_\alpha$ satisfies the following properties:
\begin{enumerate}
\item \textit{Symmetry}:
\[
\mathcal{F}_\alpha[X^\bullet, Y^\bullet]
=\mathcal{F}_{w_0\alpha}[Y^\bullet, X^\bullet].
\]
\item \textit{Eigenfunction equations}: for $f\in\Lambda_{q,t}^{\otimes r}$,
\begin{align*}
\Delta[f]_\alpha^X\mathcal{F}_\alpha[X^\bullet, Y^\bullet]
&= f[-\iota Y^\bullet]\mathcal{F}_\alpha[X^\bullet, Y^\bullet]~ \text{ and }\\
\Delta[f]_{w_0\alpha}^Y\mathcal{F}_\alpha[X^\bullet, Y^\bullet]
&= f[-\iota X^\bullet]\mathcal{F}_\alpha[X^\bullet, Y^\bullet].
\end{align*}
\item \textit{Specialization to Macdonald}: for $\lambda$ with $\core(\lambda)=\alpha$,
\begin{align*}
\mathcal{F}_\alpha[X^\bullet, \iota D_\lambda^\bullet]&= \frac{H_\lambda[X^\bullet]}{H_\lambda[\iota D_{w_0\alpha}^\bullet]}~ \text{ and }\\
\mathcal{F}_\alpha[\iota D_{w_0\lambda}^\bullet, Y^\bullet]&= \frac{H_{w_0\lambda}[Y^\bullet]}{H_{w_0\lambda}[\iota D_\alpha^\bullet]}.
\end{align*}
\end{enumerate}
\end{thm}

\begin{proof}
Property (1) is obvious.
The other two follow from manipulations similar to those done in the proof of Theorem \ref{GlobSym}.
Namely,
\begin{align*}
\mathcal{F}_\alpha[X^\bullet,Y^\bullet]&=(\nabla_\alpha^X)^{-1}\TT[-X^{(0)}](\nabla_{w_0\alpha}^Y)^{-1}\TT[-Y^{(0)}]\Omega_{q,t}[X^{-\bullet}Y^\bullet]\\
&= (\nabla_\alpha^X)^{-1}\TT[-X^{(0)}]\Omega\left[ \frac{-X^{(0)}}{(1-q\sigma^{-1})(t\sigma-1)} \right](\nabla_{w_0\alpha}^Y)^{-1}\Omega_{q,t}[X^{-\bullet}Y^\bullet]\\
&= (\nabla_\alpha^X)^{-1}\TT[-X^{(0)}]\Omega\left[ \frac{-X^{(0)}}{(1-q\sigma^{-1})(t\sigma-1)} \right](\nabla_{\alpha}^X)^{-1}\Omega_{q,t}[X^{-\bullet}Y^\bullet],
\end{align*}
which is $\mathsf{V}^{-1}$ at $e^\alpha$ applied to the $X$-variables. 
We prove the first eigenfunction equation as in Theorem \ref{BispectralThm}, using the second property in Proposition \ref{MainPropertiesofV}. The first specialization to Macdonalds is the fourth property of Proposition \ref{MainPropertiesofV}.
A similar conclusion for the $Y^{\bullet}$ can be made from the following equalities:
\begin{align*}
\mathcal{F}_\alpha[X^\bullet,Y^\bullet]&=(\nabla_\alpha^X)^{-1}\TT[-X^{(0)}](\nabla_{w_0\alpha}^Y)^{-1}\TT[-Y^{(0)}]\Omega_{q,t}[X^{-\bullet}Y^\bullet]\\
&= (\nabla_{w_0\alpha}^Y)^{-1}\TT[-Y^{(0)}]\Omega\left[ \frac{-Y^{(0)}}{(1-q\sigma^{-1})(t\sigma-1)} \right](\nabla_{w_0\alpha}^Y)^{-1}\Omega_{q,t}[X^{-\bullet}Y^\bullet].\qedhere
\end{align*}
\end{proof}

\section{Acknowledgements}
We'd like to thank
Michele D'Adderio,
Hunter Dinkins,
Mark Haiman, 
Seung Jin Lee, 
Jaeseong Oh, 
Daniel Orr, 
Mark Shimozono,
and 
Meesue Yoo for helpful discussions.
Computations were done on MAPLE using a package developed by Hunter Dinkins.
We must especially thank Anton Mellit, whose ingenious methods in the $r=1$ case inspired the breakthrough idea for our work. Last, but most definitely not least, we thank Adriano Garsia for his numerous lessons regarding Macdonald polynomials and plethystic substitution.
This work was supported by ERC consolidator grant No. 101001159 ``Refined invariants in combinatorics, low-dimensional topology and geometry of moduli spaces''.

\bibliographystyle{alpha}
\bibliography{Wreath}

\begin{thebibliography}{FJMM13}

\bibitem[AD24]{AyersDinkins}
Jeffrey Ayers and Hunter Dinkins.
\newblock Wreath {M}acdonald polynomials, quiver varieties, and quasimap counts.
\newblock {\em arXiv preprint arXiv:2410.07399}, 2024.

\bibitem[BF14]{BezFink}
Roman Bezrukavnikov and Michael Finkelberg.
\newblock Wreath {M}acdonald polynomials and the categorical {M}c{K}ay correspondence.
\newblock {\em Camb. J. Math.}, 2(2):163--190, 2014.
\newblock With an appendix by Vadim Vologodsky.

\bibitem[Che95]{ChereEv}
Ivan Cherednik.
\newblock Macdonald's evaluation conjectures and difference {F}ourier transform.
\newblock {\em Invent. Math.}, 122(1):119--145, 1995.

\bibitem[Che97]{ChereMM}
Ivan Cherednik.
\newblock Difference {M}acdonald-{M}ehta conjecture.
\newblock {\em Internat. Math. Res. Notices}, (10):449--467, 1997.

\bibitem[CNO14]{CNO}
Erik Carlsson, Nikita Nekrasov, and Andrei Okounkov.
\newblock Five dimensional gauge theories and vertex operators.
\newblock {\em Mosc. Math. J.}, 14(1):39--61, 170, 2014.

\bibitem[FJMM13]{FJMMRep}
B.~Feigin, M.~Jimbo, T.~Miwa, and E.~Mukhin.
\newblock Representations of quantum toroidal {${\germ{gl}}_n$}.
\newblock {\em J. Algebra}, 380:78--108, 2013.

\bibitem[FT16]{FeiTsym}
Boris Feigin and Alexander Tsymbaliuk.
\newblock Bethe subalgebras of {$U_q(\widehat{\germ{gl}}_n)$} via shuffle algebras.
\newblock {\em Selecta Math. (N.S.)}, 22(2):979--1011, 2016.

\bibitem[GH96]{GarsiaHaimanqtCat}
A.~M. Garsia and M.~Haiman.
\newblock A remarkable {$q,t$}-{C}atalan sequence and {$q$}-{L}agrange inversion.
\newblock {\em J. Algebraic Combin.}, 5(3):191--244, 1996.

\bibitem[GHT99]{GHT}
Adriano~M. Garsia, Mark Haiman, and Glenn Tesler.
\newblock Explicit plethystic formulas for {M}acdonald {$q,t$}-{K}ostka coefficients.
\newblock {\em S\'{e}m. Lothar. Combin.}, 42:Art. B42m, 45, 1999.
\newblock The Andrews Festschrift (Maratea, 1998).

\bibitem[GM19]{Garsia-Mellit}
Adriano Garsia and Anton Mellit.
\newblock Five-term relation and {M}acdonald polynomials.
\newblock {\em J. Combin. Theory Ser. A}, 163:182--194, 2019.

\bibitem[GT96]{GarsiaTesler}
A.~M. Garsia and G.~Tesler.
\newblock Plethystic formulas for {M}acdonald {$q,t$}-{K}ostka coefficients.
\newblock {\em Adv. Math.}, 123(2):144--222, 1996.

\bibitem[Hai03]{Haiman}
Mark Haiman.
\newblock Combinatorics, symmetric functions, and {H}ilbert schemes.
\newblock In {\em Current developments in mathematics, 2002}, pages 39--111. Int. Press, Somerville, MA, 2003.

\bibitem[Jin94]{JingMac}
Nai~Huan Jing.
\newblock {$q$}-hypergeometric series and {M}acdonald functions.
\newblock {\em J. Algebraic Combin.}, 3(3):291--305, 1994.

\bibitem[Kno97]{KnopInt}
Friedrich Knop.
\newblock Integrality of two variable {K}ostka functions.
\newblock {\em J. Reine Angew. Math.}, 482:177--189, 1997.

\bibitem[Mac15]{Mac}
I.~G. Macdonald.
\newblock {\em Symmetric functions and {H}all polynomials}.
\newblock Oxford Classic Texts in the Physical Sciences. The Clarendon Press, Oxford University Press, New York, second edition, 2015.
\newblock With contribution by A. V. Zelevinsky and a foreword by Richard Stanley, Reprint of the 2008 paperback edition.

\bibitem[Mik99]{Miki}
Kei Miki.
\newblock Toroidal braid group action and an automorphism of toroidal algebra {$U_q({\rm sl}_{n+1,\rm tor})\ (n\geq 2)$}.
\newblock {\em Lett. Math. Phys.}, 47(4):365--378, 1999.

\bibitem[Neg20]{NegutTor}
Andrei Negu\cb{t}.
\newblock Quantum toroidal and shuffle algebras.
\newblock {\em Adv. Math.}, 372:107288, 60, 2020.

\bibitem[OS23]{OSWreath}
Daniel Orr and Mark Shimozono.
\newblock Wreath {M}acdonald polynomials, a survey.
\newblock {\em arXiv preprint arXiv:2308.12166}, 2023.

\bibitem[OSW22]{OSW}
Daniel Orr, Mark Shimozono, and Joshua~Jeishing Wen.
\newblock Wreath {M}acdonald operators.
\newblock {\em arXiv preprint arXiv:2211.03851}, 2022.

\bibitem[Sah96]{SahiInt}
Siddhartha Sahi.
\newblock Interpolation, integrality, and a generalization of {M}acdonald's polynomials.
\newblock {\em Internat. Math. Res. Notices}, (10):457--471, 1996.

\bibitem[Sai98]{Saito}
Yoshihisa Saito.
\newblock Quantum toroidal algebras and their vertex representations.
\newblock {\em Publ. Res. Inst. Math. Sci.}, 34(2):155--177, 1998.

\bibitem[Sto14]{StokC}
J.~V. Stokman.
\newblock The {$c$}-function expansion of a basic hypergeometric function associated to root systems.
\newblock {\em Ann. of Math. (2)}, 179(1):253--299, 2014.

\bibitem[SV13]{SchiffVass}
Olivier Schiffmann and Eric Vasserot.
\newblock The elliptic {H}all algebra and the {$K$}-theory of the {H}ilbert scheme of {$\Bbb A^2$}.
\newblock {\em Duke Math. J.}, 162(2):279--366, 2013.

\bibitem[Tsy19]{Tsym}
Alexander Tsymbaliuk.
\newblock Several realizations of {F}ock modules for toroidal {$\ddot {U}_{q,d}(\germ {sl}_n)$}.
\newblock {\em Algebr. Represent. Theory}, 22(1):177--209, 2019.

\bibitem[VV99]{VVCyclic}
M.~Varagnolo and E.~Vasserot.
\newblock On the {$K$}-theory of the cyclic quiver variety.
\newblock {\em Internat. Math. Res. Notices}, (18):1005--1028, 1999.

\bibitem[Wen19]{WreathEigen}
Joshua~Jeishing Wen.
\newblock Wreath {M}acdonald polynomials as eigenstates.
\newblock {\em arXiv preprint arXiv:1904.05015}, 2019.

\bibitem[Wen24]{WreathOrth}
Joshua~Jeishing Wen.
\newblock Shuffle approach to wreath {P}ieri operators.
\newblock {\em arXiv preprint arXiv:2402.06007}, 2024.

\end{thebibliography}

\end{document}